\documentclass[sn-mathphys-num]{sn-jnl}

\usepackage{enumitem}
\usepackage{graphicx}
\usepackage{multirow}
\usepackage{amsmath,amssymb,amsfonts}
\usepackage{amsthm}
\usepackage{mathrsfs}
\usepackage[title]{appendix}
\usepackage{xcolor}
\usepackage{textcomp}
\usepackage{manyfoot}
\usepackage{booktabs}
\usepackage{algorithm}
\usepackage{algorithmicx}
\usepackage{algpseudocode}
\usepackage{listings}
\usepackage{dsfont}

\newtheorem{theo}{Theorem}[section]

\newtheorem{coro}[theo]{Corollary}
\newtheorem{lem}[theo]{Lemma}

\theoremstyle{definition}
\newtheorem{defi}[theo]{Definition}
\newtheorem{ex}[theo]{Example}
\newtheorem{rem}[theo]{Remark}

\newtheorem*{nota}{Notation}
\newtheorem*{conjecture}{Conjecture}

\DeclareMathOperator{\Cov}{Cov}

\newcommand {\R}{\mathbb{R}}
\newcommand {\Z}{\mathbb{Z}}
\newcommand {\N}{\mathbb{N}}
\newcommand{\pP}{\mathbb{P}}

\newcommand{\E}{\mathbb{E}}

\newcommand{\CE}[2]{\mathbb{E}\left[#1\;|\;#2\right]}

\newcommand\norm[1]{\left\Vert#1\right\Vert}
\newcommand\abs[1]{\left|#1\right|}
\newcommand*{\bigchi}{\mbox{\scalebox{1.2}{$\chi$}}}
\newcommand{\cov}[2]{\mathrm{Cov}\left(#1,\,#2\right)}

\newcommand{\bb}[1]{\left(#1\right)}
\newcommand{\Cb}[1]{\left[#1\right]}
\newcommand{\cb}[1]{\left\{#1\right\}}
\newcommand{\desum}[2]{\sum_{
\begin{array}{c}
#1\\
#2
\end{array}
}}
\newcommand{\vn}{\varnothing}
\newcommand{\ubar}[1]{\underset{\bar{}}{#1}}
\newcommand{\ind}{\mathds{1}}

\begin{document}

\title[Article Title]{Time-Scaling, Ergodicity, and Covariance Decay of Interacting Particle Systems}

\author*[1]{\fnm{Maciej} \sur{Głuchowski}}\email{mj.gluchowski@student.uw.edu.pl}

\author[2]{\fnm{Georg} \sur{Menz}}\email{gmenz@math.ucla.edu}

\affil*[1]{\orgdiv{Faculty of Mathematics, Informatics and Mechanics}, \orgname{University of Warsaw}, \orgaddress{\street{Banacha 2}, \city{Warsaw}, \postcode{02-097}, \country{Poland}}}

\affil[2]{\orgdiv{Mathematics Department}, \orgname{University of California Los Angeles}, \orgaddress{\city{Los Angeles}, \postcode{90095-1555}, \state{California}, \country{USA}}}

\abstract{

The main focus of this article is the study of ergodicity of Interacting Particle Systems (IPS). We present a simple lemma showing that scaling time is equivalent to taking the convex combination of the transition matrix of the IPS with the identity. As a consequence, the ergodic properties of IPS are invariant under this transformation. Surprisingly, this simple observation has non-trivial implications: It allows to extend any result that does not respect this invariance, which we demonstrate with examples. Additionally, we develop a recursive method to deduce decay of correlations for IPS with alphabets of arbitrary (finite) size, and apply the Time-Scaling Lemma to that as well. As an application of this new criterion we show that certain one-dimensional IPS are ergodic answering an open question of Toom et al..}



\maketitle

\bmhead{Acknowledgements}

We thank Jacek Miekisz, Marek Biskup, Jacob Manaker, Roberto Schonmann, Jan Wehr, Tom Kennedy, Sunder Sethuraman, and Pablo Ferrari for their discussions and advice. Originated during a research visit at UCLA supported financially by initiative IV.2.3. by IDUB - University of Warsaw.

\bigskip \section{Introduction}

An Interacting Particle System $\bb{\text{IPS}}$ is a Markov process on the space of configurations $\mathcal{A}^\Lambda$ for some countable alphabet $\mathcal{A}$ and lattice $\Lambda$, where updates occur independently at each vertex (site), triggered by exponential clocks.
An IPS can be thought of as a continuous-time version of Probabilistic Cellular Automata $\bb{\text{PCA}}$, with the difference that in the PCA updates occur simultaneously. They share a lot of similarities with often analogous results present in both settings or derivable for the other process using similar arguments. A proper introduction to Interacting Particle Systems, including historical context, can be found in~\cite{Ligett:05}.\\

Our primary research focus is on exploring the sufficiency criteria for an IPS to be ergodic, i.e.~that the Markov process has a unique and attractive invariant measure. Whenever possible we want to obtain explicit rates of convergence to this measure. The existence of at least one invariant measure for all PCAs and IPS is a standard result that follows as a consequence of Schauder's fixed-point theorem or by construction, see e.g.~Proposition 1.8 in~\cite{Ligett:05}.\\

In Section~\ref{sec:one_d_ips}, we consider the simplest IPS for which the above problems are non-trivial. These are the one-dimensional IPS with one-sided, nearest-neighbor interactions and a homogenous rule. We give an informal characterization of the types of behavior such IPS can exhibit and discuss their ergodicity. This section serves a dual purpose as these IPS are central to the application of our results in Section~\ref{sec:application}. \\

The main insight of this work is the Time-Scaling Lemma of Section~$\ref{sec:time-scaling}$ (see Lemma~\ref{pro:time-scaling} below). The statement is simple but has far-reaching consequences. It states that for an IPS$(P)$ with transition matrix~$P$ slowing down time with a constant factor~$\lambda>0$ is equivalent to interpolating the transition matrix~$P$ with the identity~$I$. More precisly   after scaling time by~$\lambda$, the trajectories of the IPS$(P)$ has the same distribution as the IPS$(\lambda P + (1- \lambda ) I)$, where~$\lambda P + (1- \lambda ) I$ is the convex combination of $P$ with the identity~$I$. Of course, scaling time does not change the ergodic properties of an IPS. Therefore ergodic properties, like the set of invariant measures and whether the IPS is ergodic, are invariants of the family of IPS generated by the transition matrices $\lambda P + (1- \lambda ) I$ for~$\lambda>0$.
This allows to extend several theorems and criteria that establish ergodicity. The strategy is as follows: Consider a condition that forces an IPS to be ergodic. In order to establish the ergodicity of an IPS$(P)$ it then l
suffices to show that there exists $\lambda>0$ such that the IPS$(\lambda P + (1- \lambda ) I)$ satisfies this condition. In general, every result that is not invariant under this transformation can be extended. \\

This method is illustrated in Section~\ref{sec:extensions}, where we extend standard ergodic theorems formulated in Ligget's book~\cite{Ligett:05}. These include Griffeath's ergodic theorems about additive and cancellative IPS \cite{Griffeath:78} and a generalisation of the class of  \hyperref[def:monotonicity]{monotone}~IPS~\cite{Gray:82} to a larger class, which we call \hyperref[def:weak monotonicity]{weakly monotone} IPS.\\

In Section~\ref{sec:covdec}, we give a new criterion for exponential decay of correlations (see Theorem~\ref{pro:main-result} below) which implies ergodicity of the IPS. The criterion is inspired by a known theorem of Leontovitch and Vaserstein~\cite{VasLeo:70} developed in the discrete-time setting of PCAs (see Theorem~\ref{pro:leontovitch} below). The proof of Theorem~\ref{pro:main-result} is self-contained and uses a recursion. Compared to the original theorem for PCAs, we drop some unnecessary assumptions.  Additionally, our argument is general enough to treat IPS with alphabets of arbitrary size. We obtain a similar rate of \hyperref[def:covdec]{decay of correlations} but for a larger class of transition matrices due to the time-scaling trick. We use this generalized result to study the Two-Stage Contact Process (introduced in \cite{Krone:99}) and derive a new sufficiency criterion for the process to die out.\\

In Section~\ref{sec:application}, we apply the new criterion to an open problem: Determining whether all one-dimensional, homogeneous IPS with alphabet size two, one-sided nearest neighbor interactions, and positive rates are ergodic. Simulations carried out by~Głuchowski and~Miekisz suggest that this should be true. This is a particular case of the positive rates conjecture which was popularized by Ligget~\cite{Ligett:05}:\\

\begin{conjecture}
    Every IPS on a one-dimensional lattice with homogeneous, finite-range interaction and positive rates is ergodic.\\
\end{conjecture}

For very large neighborhoods, G\'acs published a  counterexample to the positive rates conjecture~\cite{Gács:86}. As this paper is very technical, the result is still discussed in the community. This prompted Gray to write a reader's guide~\cite{Gray:01} to G\'acs counterexample. We prove the affirmative for a symmetric subcase of this problem investigated by Toom et al.~in Section~$7$ of \cite{Toom:90}:\\

\begin{conjecture}
    Every additive IPS on a one-dimensional lattice with homogeneous, one-sided interaction, nearest neighbor interactions, and positive rates is ergodic.\\
\end{conjecture}

 For precise definitions we refer to Section~\ref{sec:application}. After verifying the above conjecture, we investigated if one can drop the assumption of additivity. Combining several known criteria we visualized the regions of the underlying four-dimensional parameter space that are known to yield an ergodic IPS. The hope was that, by the time-scaling trick, extensions of those criteria would now cover the full parameter space. Unfortunately, a small region remains where none of the results known to us seem to apply. We conjecture that a method using the one-dimensionality of $\Z$ in a strong way is required to cover the remaining region. \\

\bigskip \section{Notation and basic definitions}\label{sec:notation}

We use a general framework that defines an IPS using two ingredients. The first ingredient is a family of clocks with which we construct a set of points in space-time, where transitions occur. There is a freedom of how to choose the distribution of clocks, which corresponds to the choice between discrete (PCA) and continuous (IPS) time settings. The second ingredient is the transition matrix governing the dynamics of the IPS. This construction is most similar to the graphical representation of IPS introduced by Harris~\cite{Harris:78}, with the added bonus that it unifies the study of PCAs and IPS. The construction we propose is sufficiently general to cover all possible IPS and PCAs, at the cost of nice additive properties present in the graphical representation.\\

We will start by giving a quick rundown of fundamental concepts relevant to this paper. A reader familiar with the subject may glance over this section. The only non-standard definitions are the \hyperref[def:causal]{cone of dependence and causality}. These concepts too are not new. We just formalize them and give them names.\\

The content of this section is split into three parts. The first deals with space-time, clocks, and the construction of the set of update points. In the second subsection, an IPS is formally constructed. The third subsection contains definitions of the properties of IPS we will study, such as \hyperref[def:ergodicity]{ergodicity} and \hyperref[def:covdec]{covariance decay}.\\

\bigskip
\subsection{Space-time structure}\label{subsec:space-time}\hfill
To set the stage we first define the space-time in which the IPS lives. We then mark a countable subset of that space-time as the set of points at which transitions occur.\\
\begin{nota}\hfill
\begin{itemize}
    \item Let $\Lambda$ be a lattice and for every $j\in\Lambda$ let $\mathcal{N}_j$ be the set of vertices $j$ is connected to. \\
    \item Let $\mathcal{C}$ be any Borel probability distribution on $\bb{0,\infty}$ with finite mean. We will call it a clock. In this article, this will either be a Dirac $\delta_1$ measure, a geometric $\text{Geo}(p)$, or an exponential $\text{Exp}(\lambda)$ distribution.\\
    \item Let $\bb{\rho_{n}^j}_{n\geq 0}^{j\in \Lambda}$ be a family of i.i.d.~random variables with distribution $\mathcal{C}$. \\
    \item Let $\tau_{n}^j= \rho_1^j +\dots +\rho_n^j$ for all $n\geq 0$. The $\tau_n^j$ random variable is to be the stopping time of the $n$-th update at a given site $j$. We set $\tau_0^j = 0.$\\
    \item We will denote the set of all points in space-time at which an update occurs as
    $$\mathbf{Update} = \cb{(\tau_n^j,j)\in \R_+\times \Lambda\;:\;n\geq 0}.$$\\
    \end{itemize}    
\end{nota}

For a given point in space-time, we will construct the maximal set of points in its past with which it has a causal link.\\
 
\begin{defi}[Cone of dependence]\label{def:causal}
    For a lattice-clock pair $\bb{\Lambda, C}$ and a point in space-time \mbox{$(t,j)\in\R_+\times \Lambda$} we define the $\emph{cone of dependence}$ with root at $(t,j)$, $\mathbf{Cone}(t,j)\subseteq \R_+\times \Lambda$ as follows:\\
    
        \begin{itemize}
            \item Let $t_0 :=t$ and $K_0 := \cb{j}$.\\
            \item Having defined $K_0\subseteq\dots\subseteq K_n$ and $t_0\geq\dots\geq t_n$ let
            $$t_{n+1}:=\max\cb{0\leq s < t_n\,:\,\cb{s}\times K_n \cap \mathbf{Update}\neq \vn},$$
            $$J_{n+1} \text{ be the projection of }\cb{t_{n+1}}\times K_n \cap \mathbf{Update}\text{ onto vertices of } \Lambda,$$
            $$K_{n+1} := K_n \cup \bigcup_{j\in J_{n+1}} \mathcal{N}_j.$$
            \item Now set: 
            $$\mathbf{Cone}(t,j):=\bigcup_{n\geq 0}\Cb{t_n,t_n+1}\times K_n .$$
        \end{itemize}
    \end{defi}

    \begin{figure}[H]
        \centering
         \includegraphics[scale=0.16]{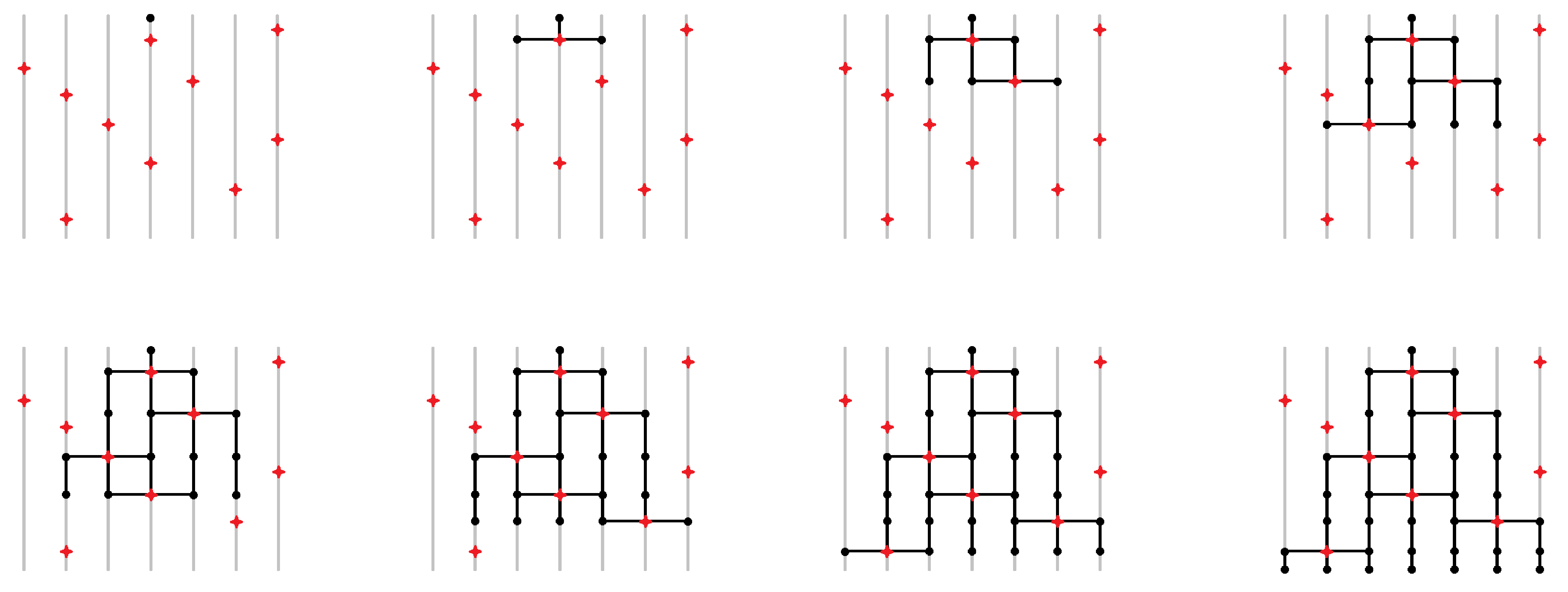}
        \caption{Iterative construction of the cone of dependence with $\Lambda = \Z$ and \mbox{$\mathcal{N}_j = j+\cb{-1,0,1}$}. The arrow of time points up. The red markers indicate the $\mathbf{Update}$ set and the black ones represent the consecutive sets $K_n$. The $\mathbf{Cone}$ consists of vertical black lines in the last panel.}
    \end{figure}

    The interpretation of these objects is as follows. Whenever a clock rings at site $j$ its state is updated according to some rule that takes into account the states of sites in $\mathcal{N}_j$. Thus to determine the state at points $\cb{t}\times K$ we must traverse the space-time backward in time, marking the relevant update times and expanding the set of points that have a \hyperref[def:causal]{causal} link with $(t,j)$. Here $(t_n)_{n\geq 0}$ are these update times and $(J_n)_{n\geq 0}$ are the sets of sites in the cone that are updated. We naturally extend this construction to finite $K\subseteq\Lambda$ by setting
    $$\mathbf{Cone}\bb{t,K} := \bigcup_{j\in K} \mathbf{Cone}\bb{t,j}.$$
    For our construction of IPS to make sense there a.s.~must be finitely many updates in the cone of dependence of any point. Otherwise it is not possible to represent the state of an IPS at a given point as a function of the initial configuration. This condition, which we term \textit{causality}, corresponds to the requirement that the recursive construction of all $\mathbf{Cone}\bb{t,j}$ almost surely halts after finitely many steps, i.e.~$t_n$ hits zero. This condition will force the trajectory of the IPS to have a uniquely defined distribution after specifying the initial configuration.\\
    
\begin{defi}[Causality]
        A lattice-clock pair $\bb{\Lambda, \mathcal{C}}$ is \emph{causal} if and only if
        $$\pP\bb{\mathbf{Cone}(t,j)\cap \cb{0}\times\Lambda\neq \vn}=1,$$
         for all $(t,j)\in \R_+\times \Lambda$.\\
    \end{defi}

\begin{ex} \mbox{}\\

\begin{itemize}
    \item For any finite lattice $\Lambda$ and any clock $\mathcal{C}$ the lattice-clock pair $\bb{\Lambda,\mathcal{C}}$ is causal.\\

       \item For any lattice $\Lambda$ the lattice-clock pair $\bb{\Lambda, \delta_{1}}$ is causal.\\

       \item If there exists finite $\mathcal{N}\subset \Z^d$ such that $\mathcal{N}_j \subseteq j+\mathcal{N}$ for all $j$ then the lattice-clock pair $\bb{\Z^d,\text{Exp}(1)}$ is causal.\\

        \item If the clock is such that $\mathcal{C}(\Cb{0,t})>0$ for all $t>0$ and the lattice $\Lambda$ is a tree of minimal order two then $\bb{\Lambda,\mathcal{C}}$ is not causal.\\

        \item If the clock is such that $\mathcal{C}(\Cb{0,t})>0$ for all $t>0$ and there exists $j\in\Lambda$ such that $\abs{\mathcal{N}_j}=\infty$ then $\bb{\Lambda,\mathcal{C}}$ is not causal.
        As a consequence, the exponential clock requires a lattice with finite neighborhoods.\\

    \item It is possible for $\bb{\Lambda,\text{Exp}(1)}$ to be causal even if $\sup_j \abs{\mathcal{N}_j}=\infty$.
    Take $\Lambda = \Z$ and $\mathcal{N}_j = \cb{0,\dots,j}$. Here every cone of dependence is spatially bounded i.e.~for any $t\geq 0$ we have $$\mathbf{Cone}(t,j) \subseteq [0,t] \times \cb{0,\dots,j}.$$ 
    This, of course, implies causality since locally the lattice is indistinguishable from a finite one.\\

    \item If $\bb{\Lambda,\mathcal{C}}$ is causal and $\hat{\Lambda}$ is a subgraph of $\Lambda$ then $\bb{\hat{\Lambda},\mathcal{C}}$ is causal.\\
    \end{itemize}
\end{ex}

\bigskip
\subsection{Construction of an interacting particle system}\hfill

    This subsection deals with the construction of an IPS over a lattice-clock pair $\bb{\Lambda,\mathcal{C}}$ with a given transition matrix. The construction is simple enough - the transition matrix is applied at every point in the $\mathbf{Update}$ set. For the conventional construction of an IPS using flip rates see Ligget's textbook \cite{Ligett:05} (Chapter I.$3$, page $20$). \\
    
\begin{defi}[Transition matrix]\label{tmatrix}\hfill
        Let $\mathcal{A}$ be a countable set which is called an alphabet. For every $j\in\Lambda$ and $\sigma \in \mathcal{A}^{\mathcal{N}_j}$ let $P_j\bb{\cdot\,|\,\sigma}$ be a probability distribution on $\mathcal{A}$.
        The collection $P = \bb{\bb{P_j\bb{\cdot\,|\,\sigma}}_{\sigma\in\mathcal{A}^{\mathcal{N}(j)}}}_{j\in\Lambda}$ is called transition matrix.\\
\end{defi} 

In this article, we only consider time-homogeneous IPS.  \\

\begin{defi}
The IPS is (time-)\emph{homogeneous} if the transition matrix at a specific site only depends on its neighborhood but not the location of the site~$i$. \\
\end{defi}

\begin{defi}[Trajectory of IPS]\label{def:general construction}
    For a given initial distribution $\mu_0$ on $\mathcal{A}^{\Lambda}$, the trajectory $\bb{\zeta_t}_{t\geq 0}$ of the IPS is constructed recursively:\\
    
    \begin{itemize}
        \item The initial configuration $\zeta_0$ is distributed according to the initial distribution $\mu_0$.\\
        \item The state at site $j$ can only be updated when its respective clock rings i.e.~\newline if $t\in \left[\tau_n^j, \tau_{n+1}^j\right)$ then $\zeta_t(j) = \zeta_{\tau_n^j}(j)$.\\
        \item The random variable $\zeta_{\tau_{n}^j}(j)$ is sampled from the  $P_j(\,\cdot\;|\;\zeta_{\tau_{n}^{j-}} )$ distribution independently of everything else. We use the convention that for a given stopping time $\tau$ the $\zeta_{\tau^-}$ denotes the left-side limit of $\bb{\zeta_t}_{t<\tau}$ random variables.\\
    \end{itemize} 
    Since $\bb{\Lambda, C}$ is \hyperref[def:causal]{causal}, the recursion above will define $\zeta_{\tau_n^j}(j)$ for any $(j,n)$ in finitely many steps. Additionally, since the range of $C$ lies in $(0,\infty)$ then for every $t\geq 0$ and $j\in \Lambda$ there exists $n$ such that $t\in \left[\tau_n^j, \tau_{n+1}^j\right)$. Thus $\zeta_t$ is well defined.\\
\end{defi}

\begin{defi} [Interacting particle system] \label{def:IPS}
    An \emph{interacting particle system} with a \hyperref[def:causal]{causal} lattice-clock pair $\bb{\Lambda, \mathcal{C}}$ and transition matrix $P$ is the set of distributions of $\bb{\zeta_t}_{t\geq 0}$ indexed by all initial distributions. We denote it by $\text{IPS}\bb{P, \mathcal{C}}$.\\ 
\end{defi}

\begin{nota} For convenience, we introduce the following notation:\smallskip

    \begin{itemize}
        \item If the distribution of a \hyperref[def:general construction]{trajectory}  $\bb{\zeta_t}_{t\geq 0}$ is a member of $\\\text{IPS}\bb{P,  \mathcal{C}}$ then we write $\bb{\zeta_t}_{t\geq 0} \in \text{IPS}\bb{P,  \mathcal{C}}.$\\
        \item We will denote $\mu P_t$ as the distribution of $\zeta_t$ given $\zeta_0$ is distributed according to $\mu$.\\
        \item $ \text{IPS}\bb{P}: = \text{IPS}\bb{P,\text{Exp}(1)}$\\
        \item $ \text{PCA}\bb{P} := \text{IPS}\bb{P,\delta_1}$.\\
        
     \end{itemize}
\end{nota}

The PCA stands for \emph{probabilistic cellular automata} - the discrete-time version of interacting particle systems. The clock being $\delta_1$ ensures that the updates happen at integer-valued times and are simultaneous for all sites.\\

\bigskip
\subsection{Ergodicity and covariance decay}
In this section we introduce the principal objects of our study of Interacting Particle Systems, namely invariant measures, ergodicity, and covariance decay. The motivation for this interest is discussed at length in the introduction to Ligget's textbook \cite{Ligett:05}.\\

\begin{defi}[Invariant measure]\label{def:invariant measure}
    An \emph{invariant measure} of an $\text{IPS}\bb{P,\mathcal{C}}$ is any distribution $\mu$ on $\mathcal{A}^\Lambda$ such that $ \mu P_t = \mu$ for all~$t \geq 0$.\\
\end{defi}

 An IPS with a memoryless clock always has at least one invariant measure. The $\bb{P_t}_{t\geq 0}$ (or $t\in\N$ when time is discrete) operators form a Markov semigroup on continuous functions on $\mathcal{A}^\Lambda$ equipped with the product topology of discrete $\mathcal{A}$. For geometric clocks (including $\delta_1$), the existence of an invariant measure follows immediately from Shauder's fixed-point theorem applied to the $P_1$ operator. The proof for exponential clocks may be found in \cite{Ligett:05}(Proposition $1.8$ on page 10).\\

\begin{defi}[Attractive invariant measure]
    An invariant measure $\mu$ of an $\text{IPS}\bb{P,\mathcal{C}}$ is \emph{attractive} if for any other distribution $\nu$ we have convergence in distribution $\nu P_t \Rightarrow \mu$ as $t \to \infty.$\\
\end{defi}
\begin{defi}[Ergodicity]\label{def:ergodicity}
    An $\text{IPS}(P,\mathcal{C})$ is said to be \emph{ergodic} if it has an attractive measure $\mu$.
    It is \emph{exponentially \hyperref[def:ergodicity]{ergodic} with rate} $\alpha$ if for every finite $S\subseteq \Lambda$ there exists a constant $C_S\geq 0$ such that for any $A\subset \mathcal{A}^\Lambda$ - cylinder set with support $S$ it holds that
    $$\sup_\nu \abs{\nu P_t\bb{A} - \mu(A)}\leq C_S\cdot e^{-\alpha t}.$$
\end{defi}

\begin{defi}[Local functions]
    Let us consider a lattice $\Lambda$ and alphabet $\mathcal{A}$. The domain~$\text{Dom}(f)$ of a function $f:\mathcal{A}^\Lambda\rightarrow\R$ is defined as the minimal set of coordinates on which the value of~$f$ depends. The function is called local if its domain is finite. The set of local functions is denoted as 
    $$\mathbf{F}(\mathcal{A},\Lambda) = \cb{f:\mathcal{A}^\Lambda\rightarrow\R\;:\;\abs{\text{Dom}(f)}<\infty}.$$
    We note that such functions are necessarily bounded and that they form a linear space. We use the short notation $\mathbf{F}$ when the alphabet and lattice are known from the context.\\
\end{defi}

\begin{defi}[Covariance decay]\label{def:covdec}
    An $\text{IPS}(P,\mathcal{C})$ has \emph{(temporal) covariance decay} if for all $f,g\in \mathbf{F}$ and initial distributions
    $$\cov{f\bb{\zeta_0}}{g\bb{\zeta_t}}\xrightarrow{t\rightarrow\infty}0.$$ 
 The decay is uniform if additionally for all $f,g\in \mathbf{F}$
$$\sup_{\bb{\zeta_t}_{t\geq 0}\in \text{IPS}(P,\mathcal{C})}\abs{\cov{f\bb{\zeta_0}}{g\bb{\zeta_t}}}\xrightarrow{t\rightarrow\infty}0.$$
The decay is exponential if there exist $\alpha >0$ and constants $C(f,g)$ such that for all initial distributions 
$$\abs{\cov{f\bb{\zeta_0}}{g\bb{\zeta_t}}}\leq C(f,g)\cdot e^{-\alpha t}.$$
\end{defi}
\mbox{}\\
We conclude this section with the following observation.\\

\begin{lem}\label{pro:ergodicity_equiv_covdec} If the clock is memoryless then \hyperref[def:covdec]{covariance decay} and \hyperref[def:ergodicity]{ergodicity} are equivalent. The same is true for exponential covariance decay and exponential ergodicity.
\end{lem}
\begin{proof}[Proof of Lemma~\ref{pro:ergodicity_equiv_covdec}]
    Since indicators of finite cylinder sets form a basis of $\mathbf{F}\bb{\mathcal{A},\Lambda}$ then ergodicity of an IPS is equivalent to the condition that for any local function $f\in\mathbf{F}$ the conditional expectation $\CE{f\bb{\zeta_t}}{\zeta_0}$ converges to a constant~$c$, i.e.
    $$\sup_{\xi} \abs{\CE{f\bb{\zeta_t}}{\zeta_0 = \xi} - c  }\xrightarrow{t\rightarrow\infty} 0.$$ Thus ergodicity implies temporal covariance decay, since by the law of total covariance 
    $$\cov{g(\zeta_0)}{f(\zeta_t)} = \cov{g(\zeta_0)}{\CE{f\bb{\zeta_t}}{\zeta_0}}\xrightarrow{t\rightarrow\infty} 0,$$
    for any initial distribution and local functions $f,g$. The rate of decay is clearly preserved.\\

    To show that temporal covariance decay implies ergodicity suppose that an IPS has covariance decay but there are measures $\mu$ and $\nu$ such that $\mu P_t$ and $\nu P_t$ do not asymptotically converge. Then there must exist a local functions $f,g $ such that $\E_{\nu}(g) \neq \E_{\mu}(g)$ and $\E_{\nu P_t}(f) $ does not asymptotically converge with $ \E_{\mu P_t}(f)$. Consider the initial measure $\frac{1}{2}\mu + \frac{1}{2}\nu$. By temporal covariance decay, we have
    $$\Cov_{\frac{1}{2}\mu + \frac{1}{2}\nu}\bb{g(\zeta_0),\;f(\zeta_t)}\xrightarrow{t\rightarrow\infty}0.$$
    But again by temporal covariance decay, this covariance is by the law of total covariances equal to
    \begin{align*}
        \Cov_{\frac{1}{2}\mu + \frac{1}{2}\nu}\bb{g(\zeta_0),\;f(\zeta_t)} & =  \frac{1}{4}\bb{\E_{\mu}(g)-\E_{\nu}(g)}\bb{\E_{\mu P_t}(f) - \E_{\nu P_t}(f)} \\
        & \qquad + \frac{1}{2}\Cov_{\mu}\bb{g(\zeta_0),\;f(\zeta_t)} + \frac{1}{2} \Cov
_{\nu}\bb{g(\zeta_0),\;f(\zeta_t)},
    \end{align*}
    which does not converge to zero. By contradiction such $\mu$ and $\nu$ cannot exist, so all measures converge asymptotically under $P_t$. In this case, the existence of an invariant measure implies ergodicity, and for IPS with memoryless clocks, an invariant measure always exists. The exponential decay rate is again preserved since we can bound 
    $$\abs{\E_{\mu P_t}(f) - \E_{\nu P_t}(f)}\leq \frac{8\cdot C_{f,g}}{\abs{\E_{\mu}(g)-\E_{\nu}(g)}}e^{-\alpha t}. \hfill \qedhere$$
\end{proof}

\section{IPS on a one-dimensional lattice with one-sided, nearest-neighbor interaction}\label{sec:one_d_ips}

In this section, we introduce a very simple class of IPS and discuss their distinct qualitative ergodic behavior for illustration.  This class of IPS also plays a central role later in  Section~\ref{sec:application}. The underlying lattice is $\bb{\Lambda,\mathcal{N}} = \bb{\Z,\cb{0,1}}$. The alphabet is $\mathcal{A}=\cb{0,1}$, which means that every site can either be unoccupied or occupied. The transition matrix $P$ is time and spatially homogeneous and the clocks are independent exponentially distributed with rate 1. In this simple situation, the IPS is determined by the following four parameters:
$$p_{1|11} := P(1\;|\;11),\quad p_{1|10}:= P(1\;|\;10),\quad p_{1|01}:= P(1\;|\;01),\quad p_{1|00}:= P(1\;|\;00).$$

What makes this class of IPS interesting is the fact that it is the most simple class for which ergodicity is not fully understood. For more details we refer to Section~\ref{sec:application} below. Since the construction of IPS via the transition matrix differs from the usual construction via flip rates, let us explain how they are related to each other. An IPS with an alphabet of size two and flip rates that have been normalized to satisfy 
$$\sup_{j\in\Lambda,\, \zeta\in\cb{0,1}^\Lambda}c_j\bb{\zeta} = 1$$
is equivalently given by the transition matrix

$$P_j\bb{1\,|\,\zeta} = \left\{ \begin{array}{cc}
     1- c_j\bb{\zeta},&  \zeta(j) = 1 \\
     c_j\bb{\zeta},& \zeta(j) = 0 
\end{array} \right..$$\\

Because of their simplicity, the trajectories of those IPS can only evolve in a few ways. In our study, we have identified five distinct types of one-dimensional, one-sided nearest neighbor IPS. In every subsequent figure depicting the evolution of a trajectory of an IPS the arrow of time points up. The color white represents the "$0$" state abn the color black the "$1$" state.

\subsection{Symmetric IPS}
These are IPS that treat both states similarly. This means that the transition matrix is roughly invariant under renaming zeroes and ones. Equivalently, the parameters of such IPS satisfy
$$p_{1|11}+ p_{1|00}\approx 1\approx p_{1|10}+ p_{1|01}.$$

Examples include the Stochastic Ising model with low external field, the "do nothing", "copy your neighbor", and "flip your neighbor" rules as well as convex combinations of these rules. In low noise regimes, the trajectories produce consist of wide stripes of each state with boundaries between them performing independent random walks. Correlations between states of neighboring sites are high but quickly dissipate with distance.\\

\begin{figure}[H]
        \centering
       \includegraphics[scale=0.175]{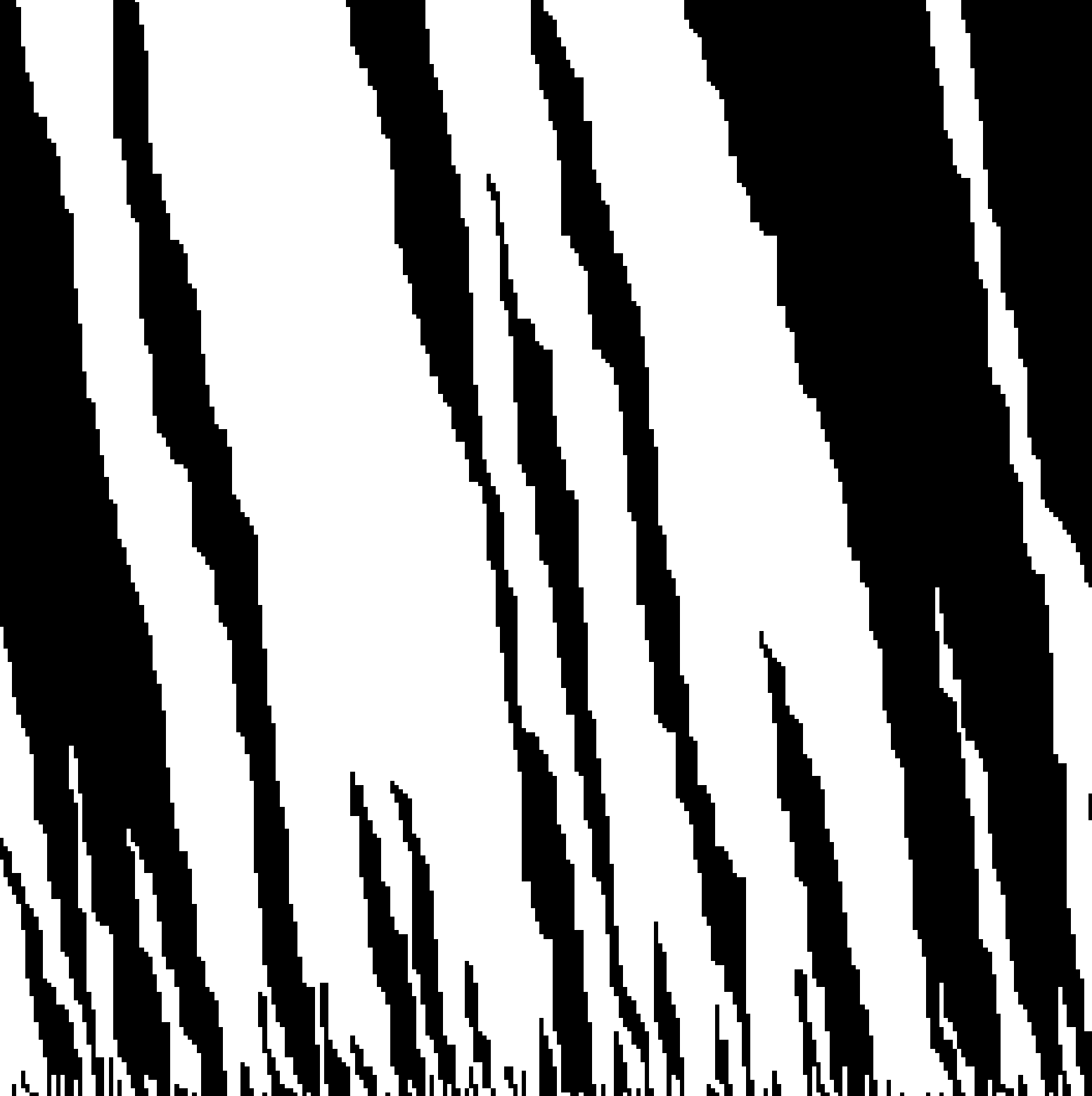}
       \hfill
       \includegraphics[scale=0.175]{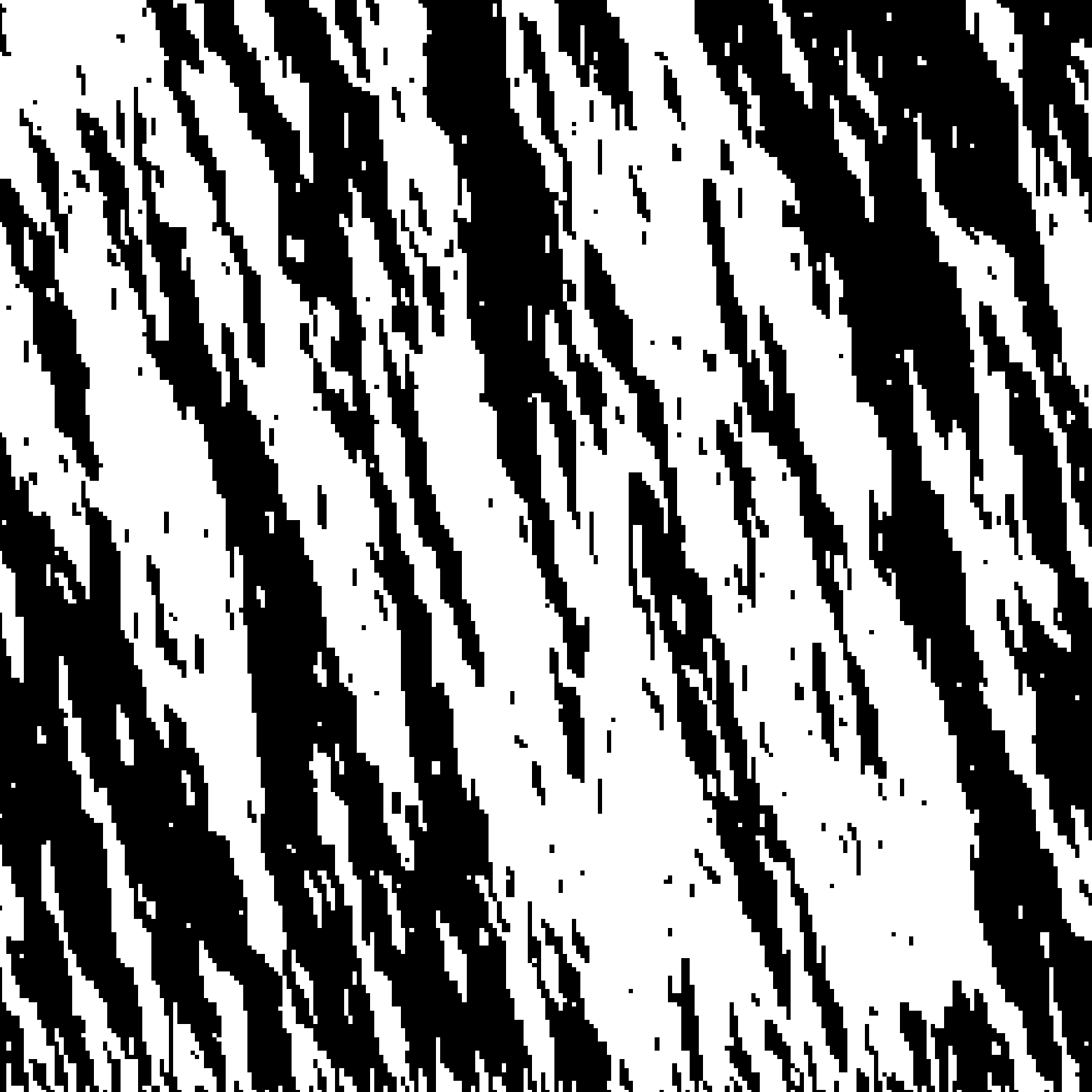}
        \caption{The Stochastic Ising model, where sites try to align their states. The semi-deterministic IPS on the left has parameters \mbox{$(p_{1|11},\,p_{1|10},\,p_{1|01},\,p_{1|00})=(1,\,0.8,\,0.2,\,0)$}, it's noisy version on the right \mbox{$(0.9,\,0.8,\,0.2,\,0.1)$}. In both simulations the initial state is random.}
    \end{figure}

    \begin{figure}[H]
        \centering
       \includegraphics[scale=0.13]{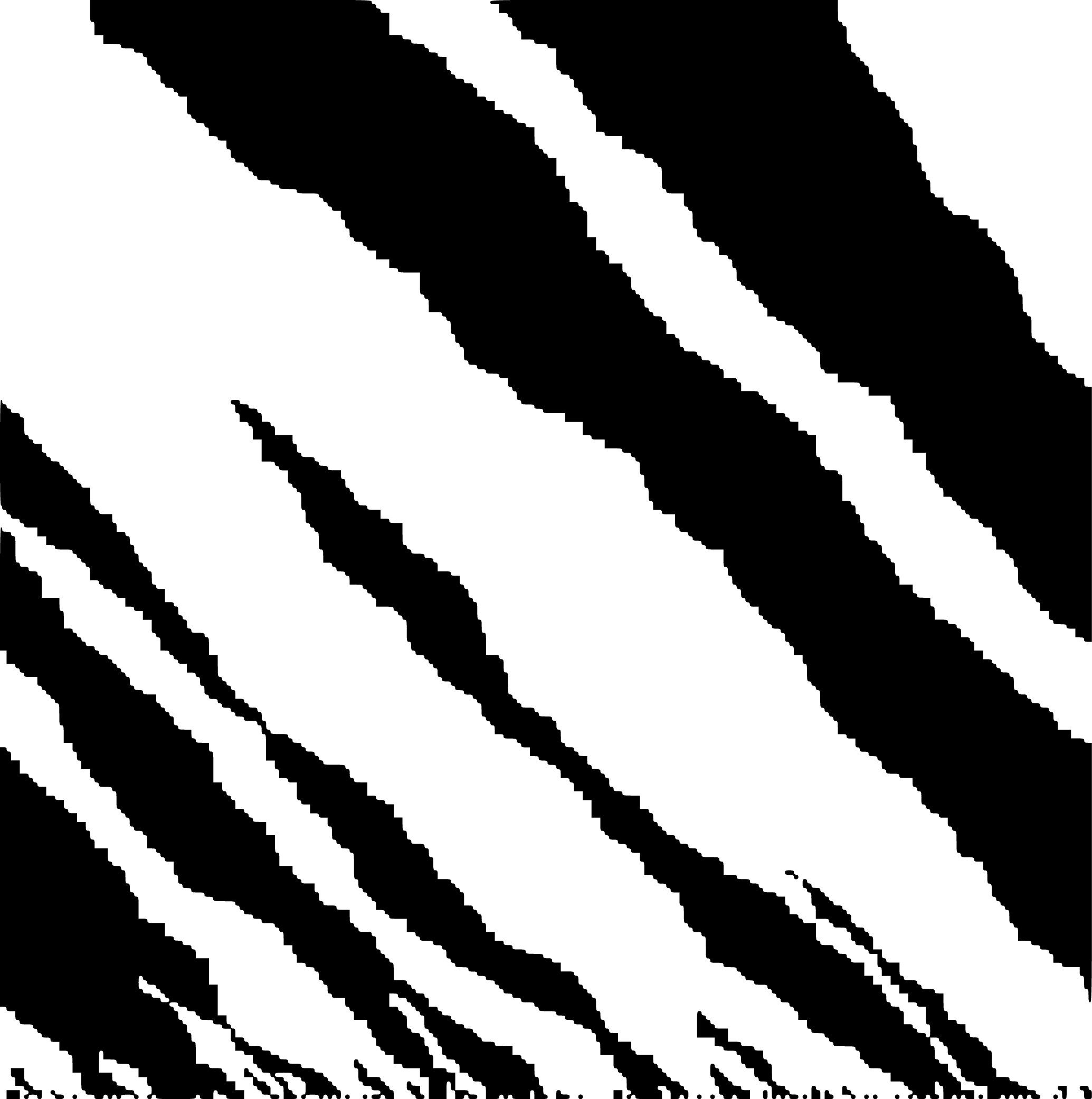}
       \hfill
       \includegraphics[scale=0.175]{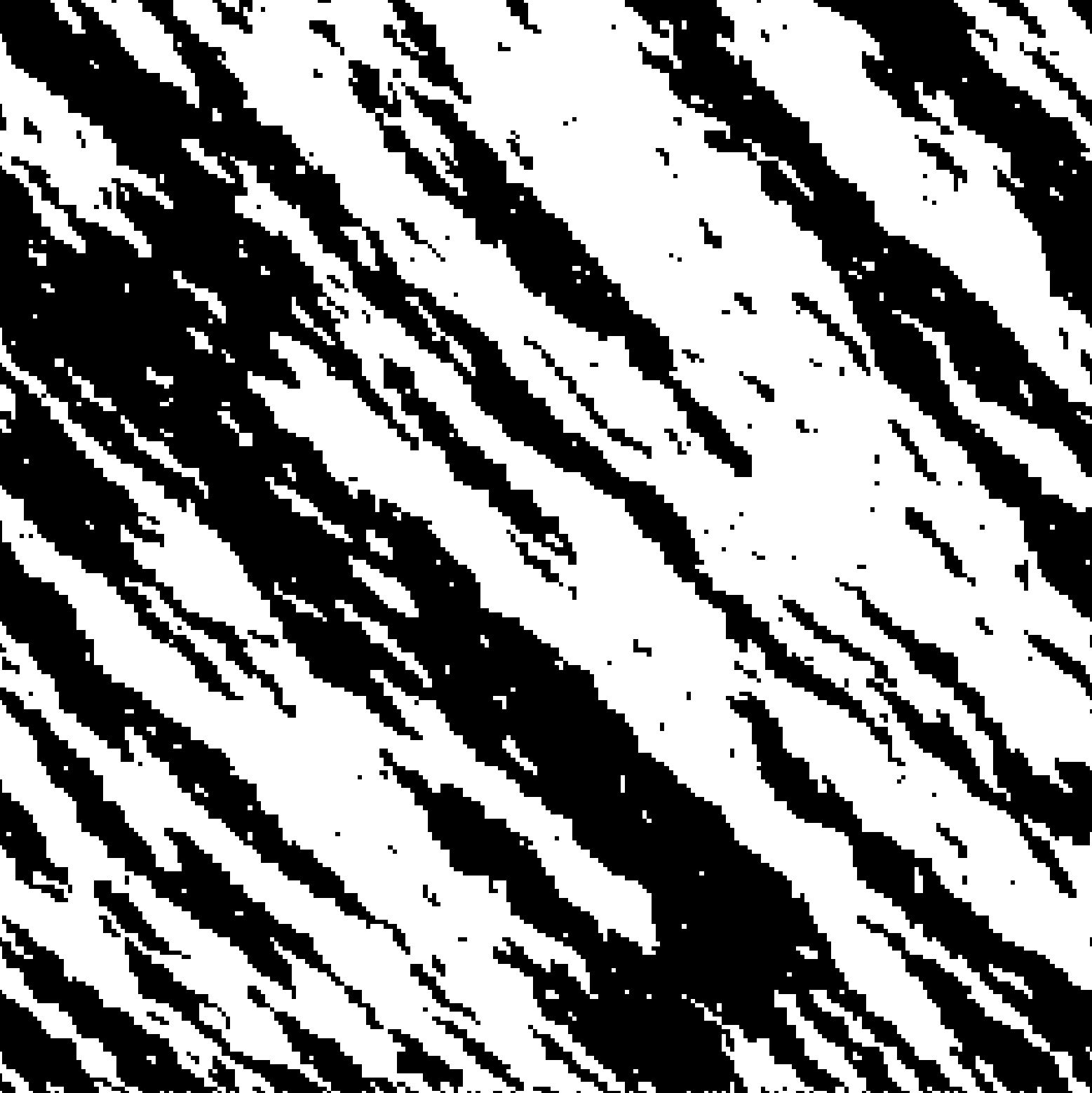}
        \caption{The "copy your neighbor" IPS, which acts as the Stochastic Ising model with added drift. The semi-deterministic version on the left has parameters \mbox{$(p_{1|11},\,p_{1|10},\,p_{1|01},\,p_{1|00})=(1,\,0.2,\,0.8,\,0)$}, it's noisy version on the right  \mbox{$(0.9,\,0.2,\,0.8,\,0.1)$}. In both simulations the initial state is random.}
    \end{figure}

The outlier here is the "flip your neighbor" rule, which seems to behave differently, producing a checkerboard pattern. This however is only an illusory distinction: If we rename zeroes and ones at every other site (formally yielding an IPS with a periodic rule) we get back the familiar stripes.

    \begin{figure}[H]
        \centering
       \includegraphics[scale=0.13]{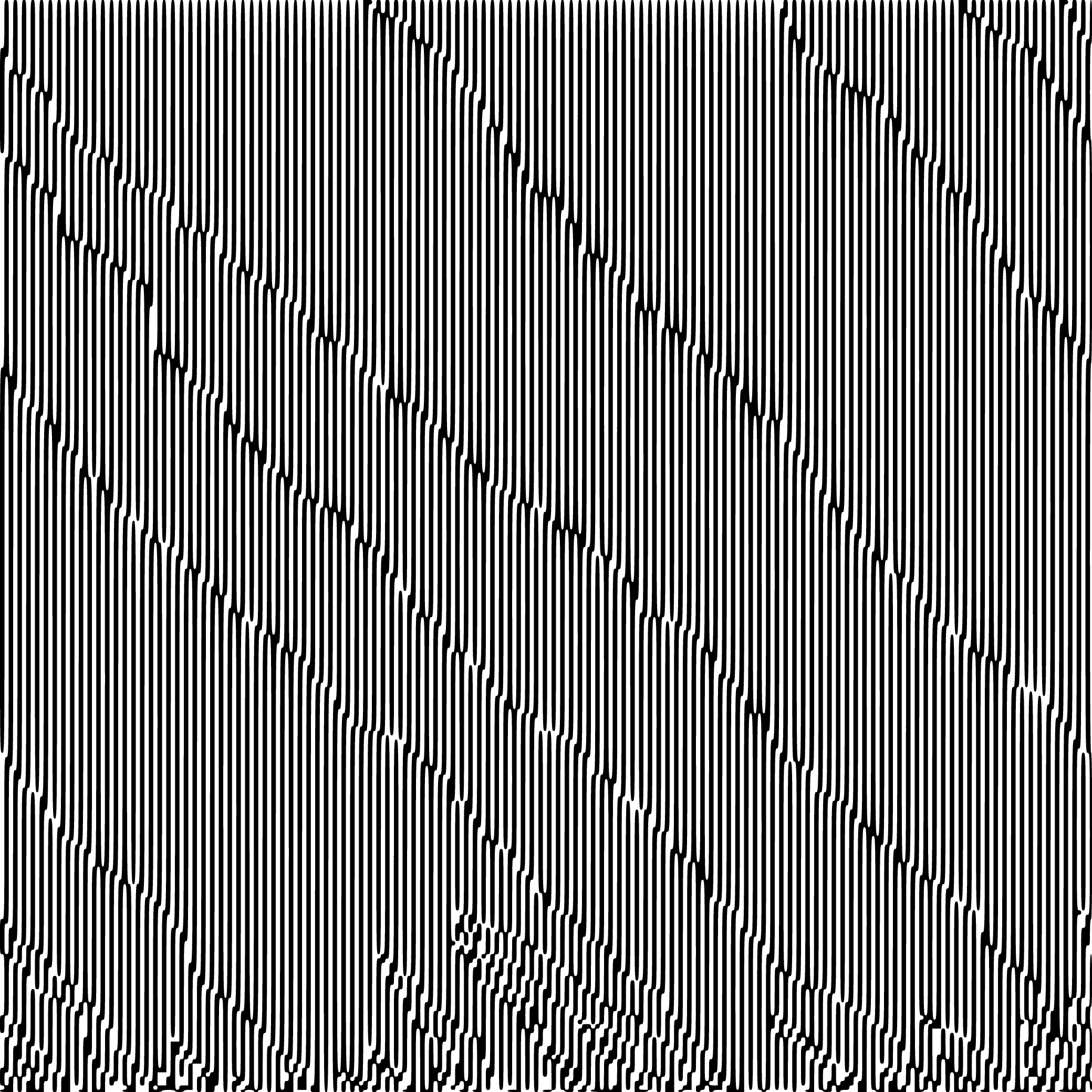}
       \hfill
       \includegraphics[scale=0.13]{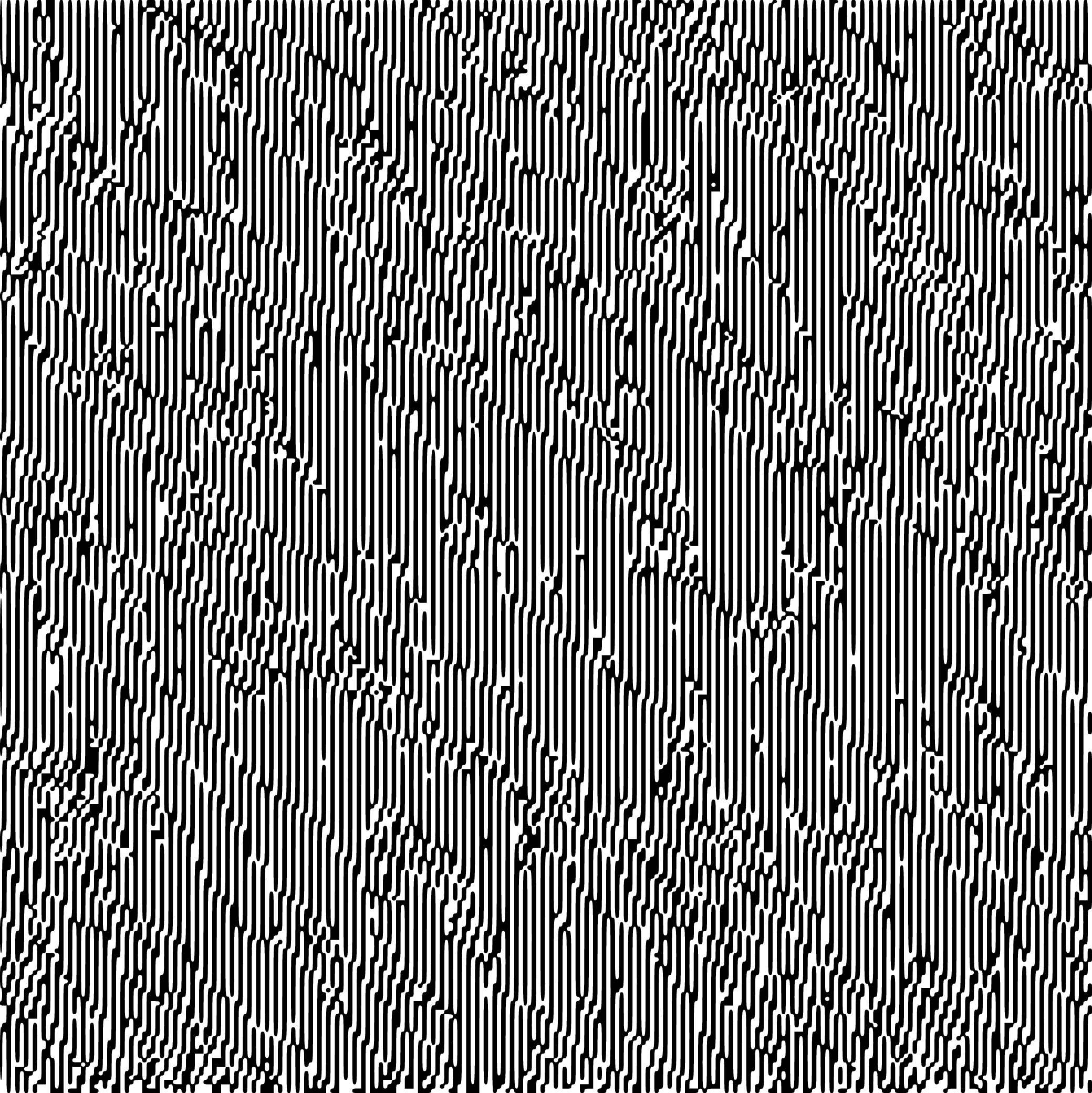}
        \caption{The "flip your neighbor" IPS, which produces stripes of checkerboard patterns with opposite parity. The semi-deterministic version on the left has parameters \mbox{$(p_{1|11},\,p_{1|10},\,p_{1|01},\,p_{1|00})=(0,\,0.8,\,0.2,\,1)$}, it's noisy version on the right  \mbox{$(0.1,\,0.8,\,0.2,\,0.9)$}. In both simulations the initial state is random.}
    \end{figure}

    \begin{figure}[H]
        \centering
       \includegraphics[scale=0.13]{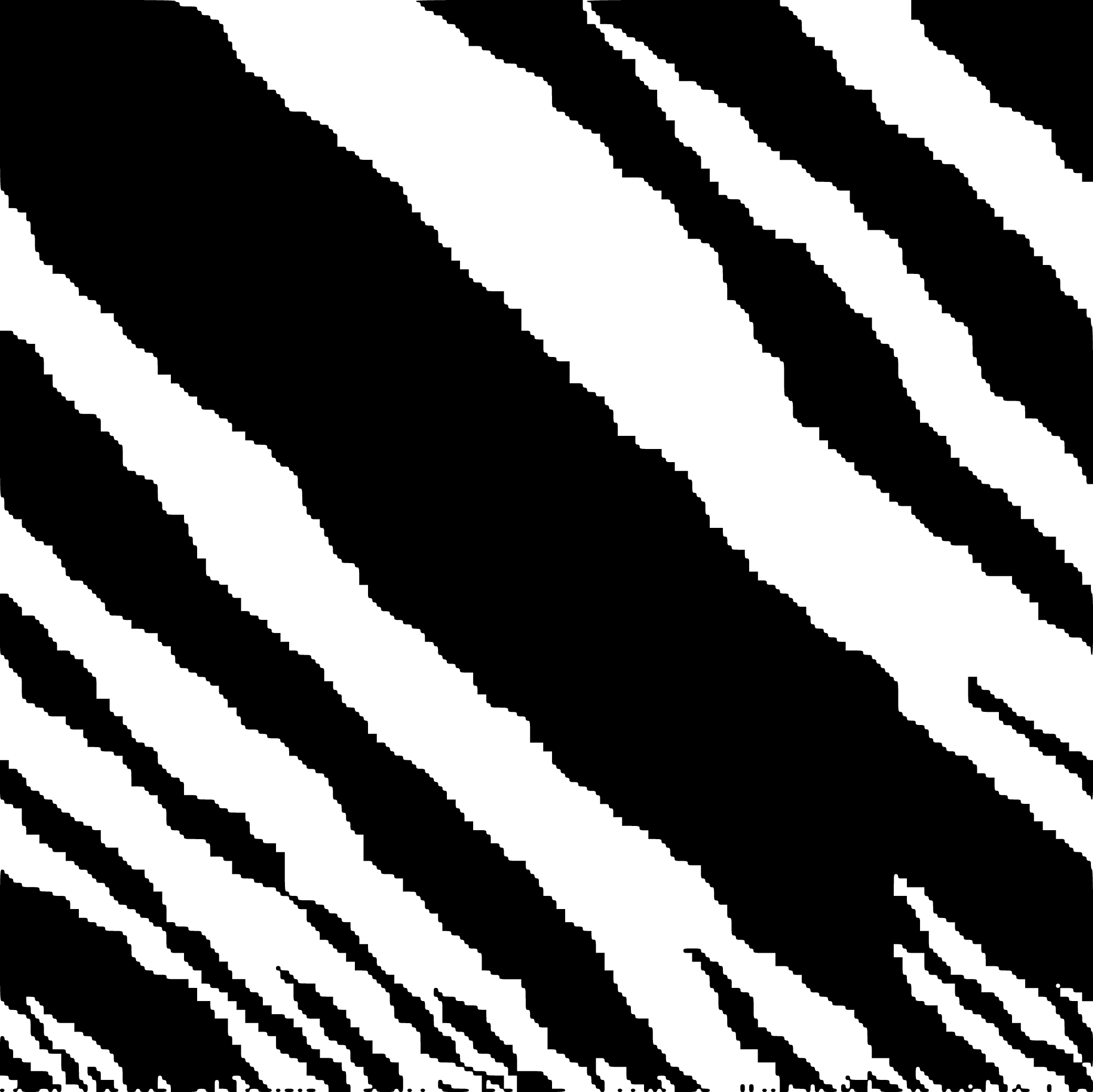}
       \hfill
       \includegraphics[scale=0.13]{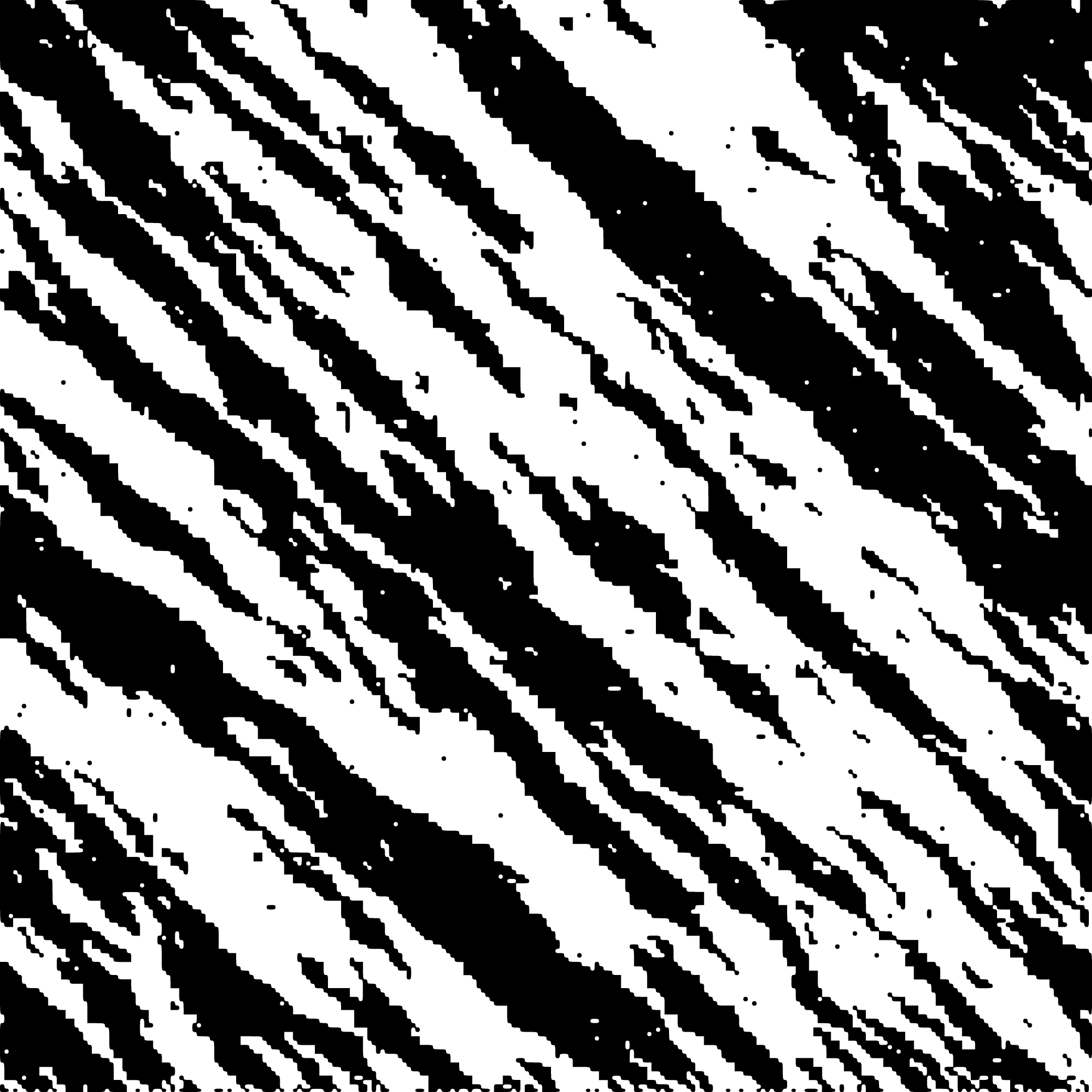}
        \caption{The same "flip your neighbor" IPSs after renaming zeroes and ones at every other site. The checkerboard patterns of opposite parity turned into homogenous stripes. In both simulations the initial state is random.}
    \end{figure}

    These IPS tend to be monotone in which case ergodicity in the presence of low noise follows by Gray~\cite{Gray:82} (see Theorem~\ref{pro:monoresult} below). Our contribution in Theorem~\ref{pro:main-result} yields ergodicity for the non-monotone cases.\\
    
\subsection{IPS with heavy bias}
These are the IPS in which one state is heavily favored. All sites try to switch to that state regardless of their neighborhood. The parameters of these IPS satisfy
$$p_{1|11}\approx p_{1|10}\approx p_{1|00} \approx 0 \;\text{ or }\; p_{1|11}\approx p_{1|01}\approx p_{1|00} \approx 1.$$

Simulations show that there is no need to additionally demand that~$p_{1|01} \approx 0$ or~$p_{1|10} \approx 0$. Only the dominating state will survive over time. Their ergodicity is immediately apparent as there is little interaction between sites. The unique invariant measure is a perturbation of Dirac's delta of "all zeroes" or "all ones" configurations.

\begin{figure}[H]
        \centering
       \includegraphics[scale=0.175]{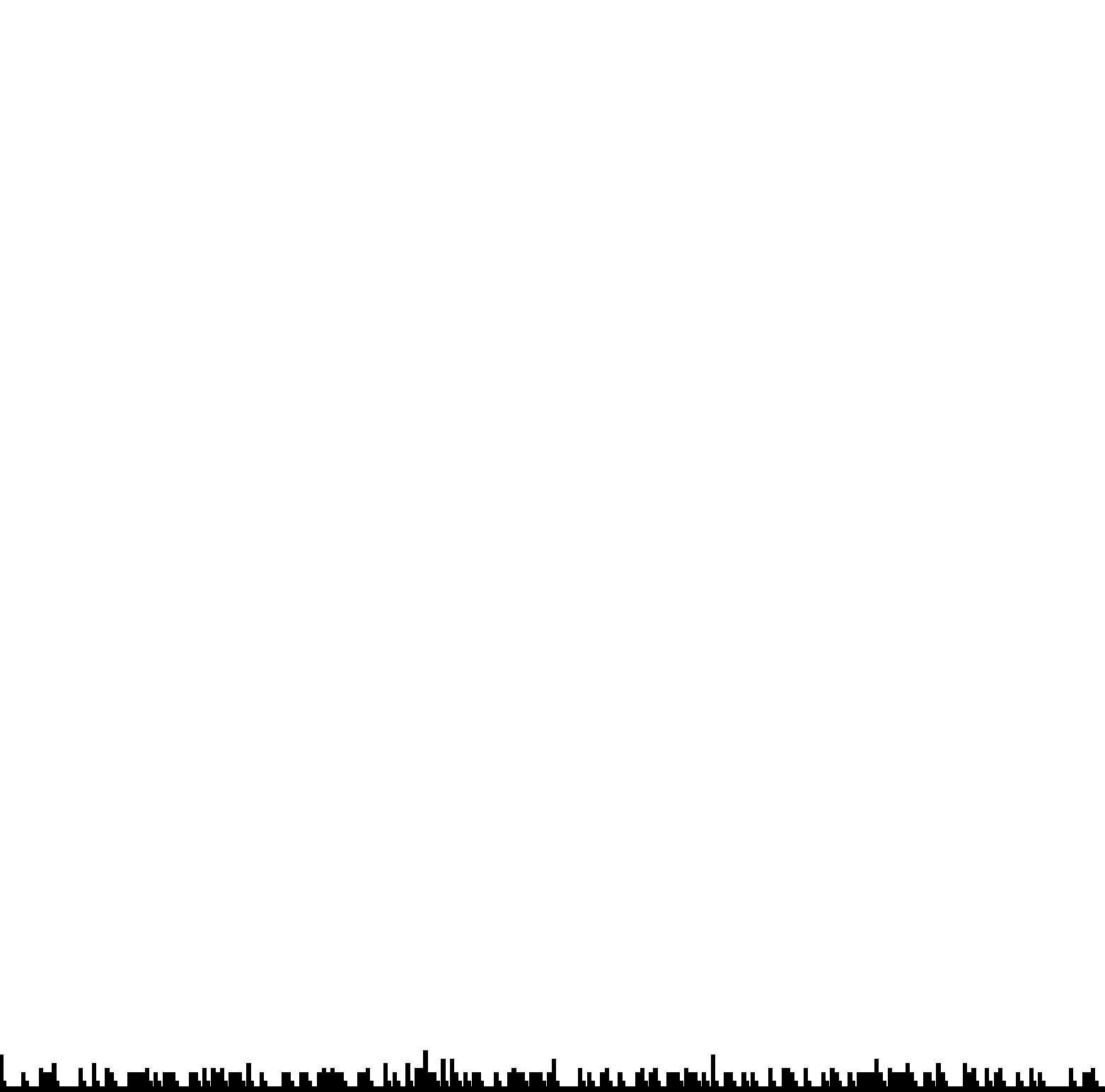}
       \hfill
       \includegraphics[scale=0.175]{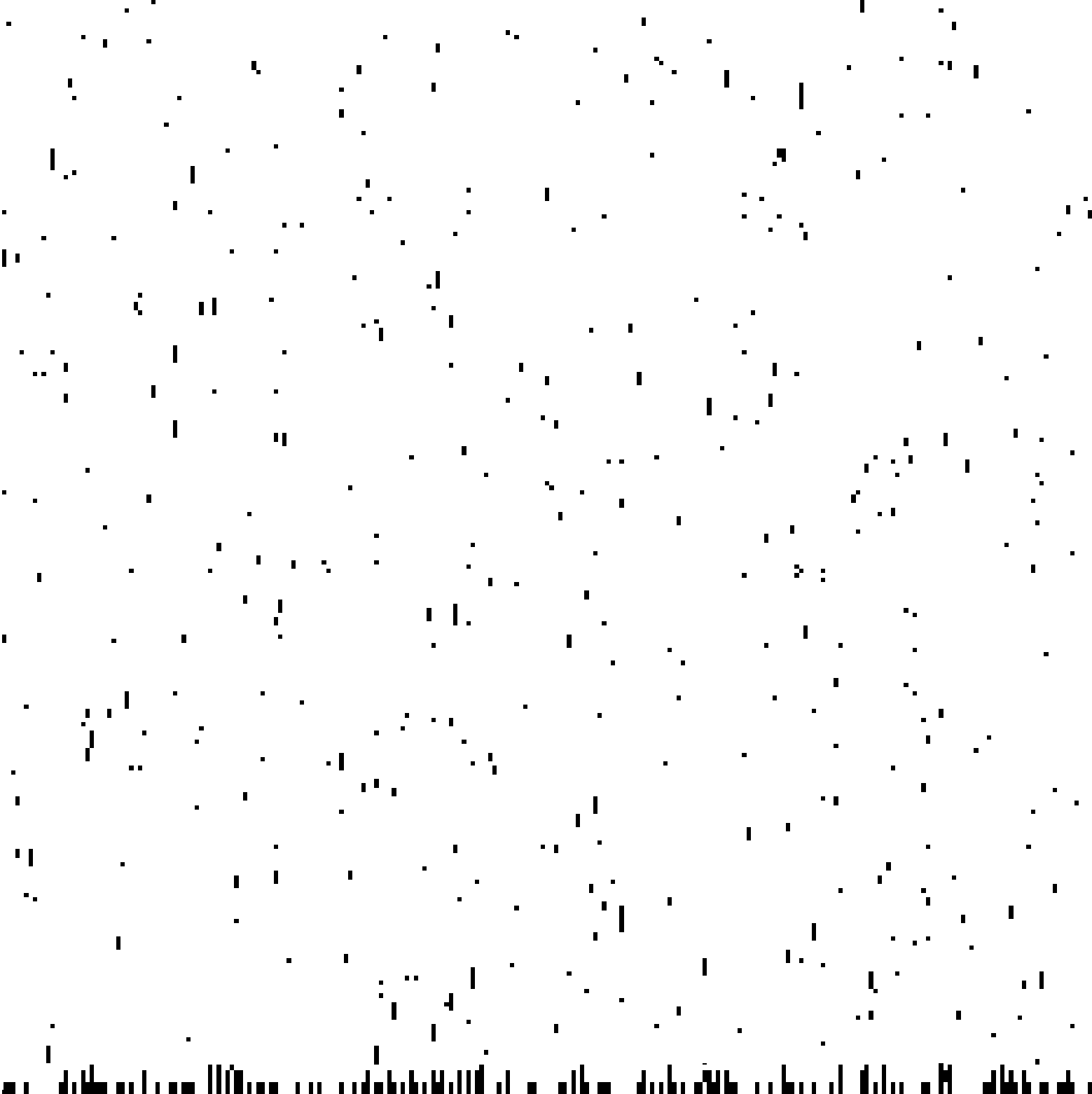}
        \caption{Evolution of the "turn to zero" IPS. The deterministic version on the left has parameters \mbox{$(p_{1|11},\,p_{1|10},\,p_{1|01},\,p_{1|00})=(0,\,0,\,0,\,0)$}, it's noisy version on the right  \mbox{$(0.1,\,0.1,\,0.1,\,0.1)$}. In both simulations the initial state is random.}
    \end{figure}
\begin{figure}[H]
        \centering
       \includegraphics[scale=0.175]{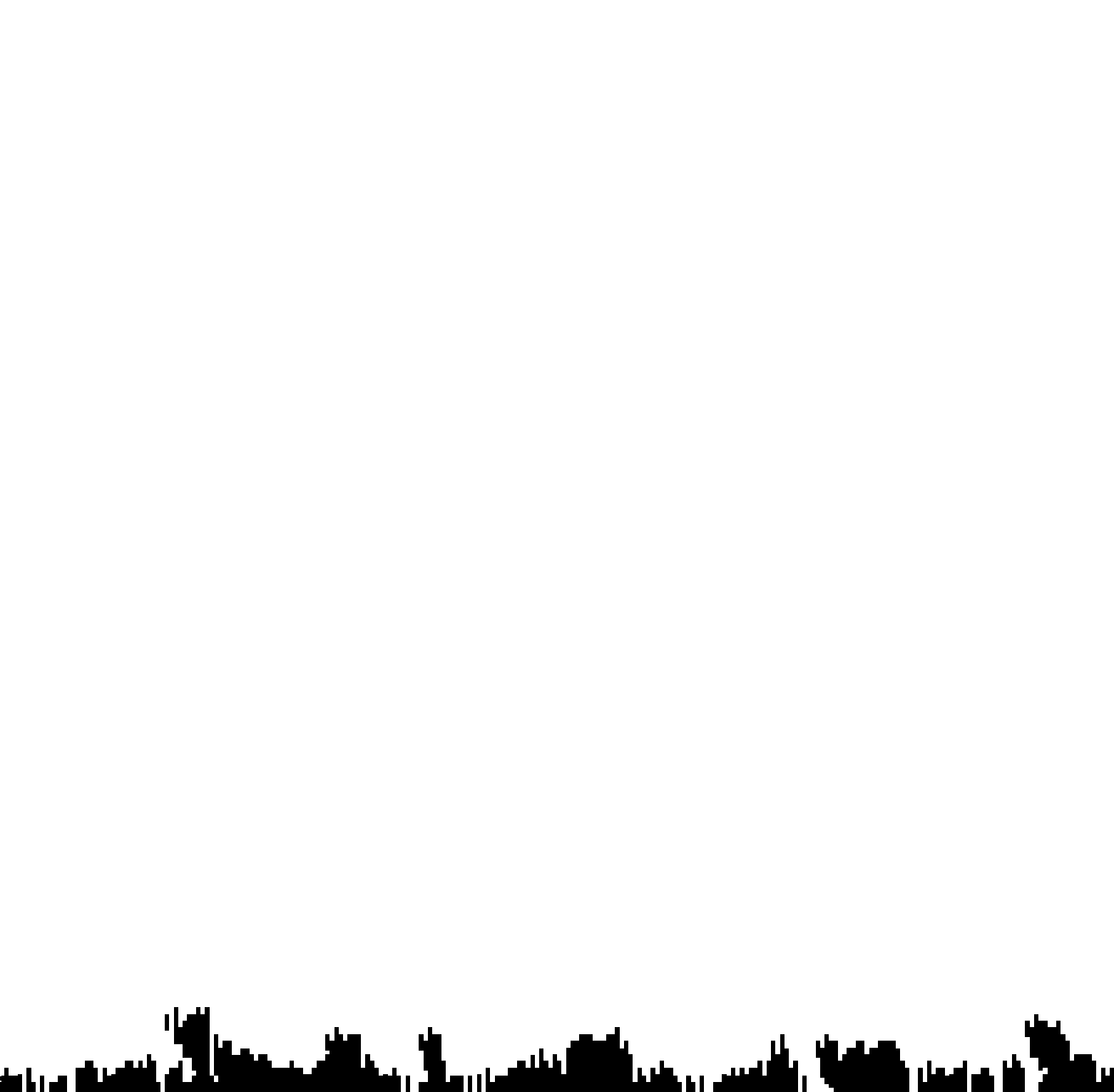}
       \hfill
       \includegraphics[scale=0.175]{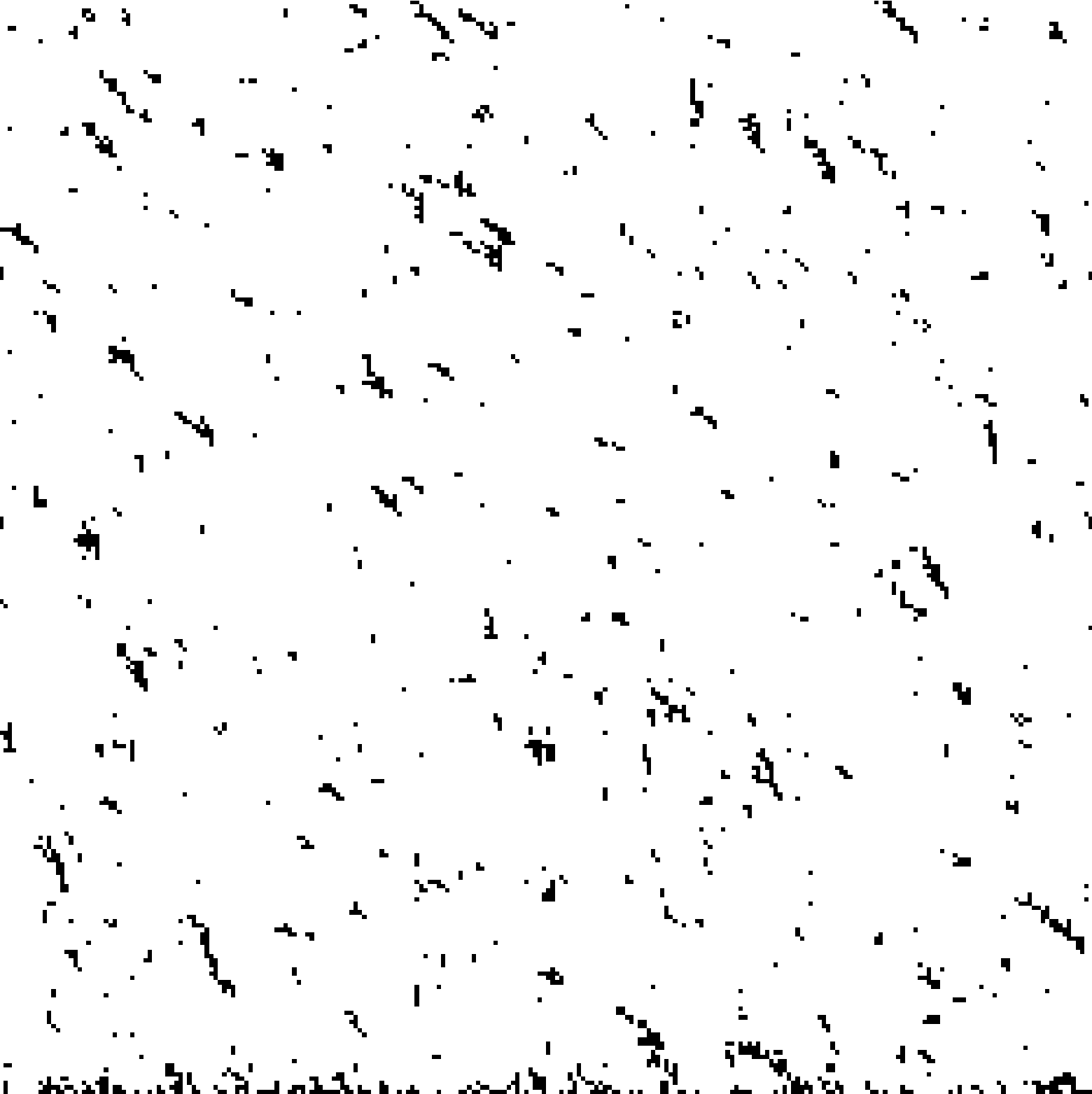}
        \caption{Evolution of a weaker version of the "turn to zero" IPS. The deterministic version on the left has parameters \mbox{$(p_{1|11},\,p_{1|10},\,p_{1|01},\,p_{1|00})=(0,\,0,\,1,\,0)$}, it's noisy version on the right  \mbox{$(0.1,\,0.1,\,0.9,\,0.1)$}. In both simulations the initial state is random.}
    \end{figure}

\subsection{Coalescing and Annihilating Contact Process}
The Coalescing Contact Process is a model of population growth where one state is interpreted as empty space and the other as members of a spreading population without spontaneous births. These IPS are given by parameters 
$$p_{1|11},\,p_{1|10},\,p_{1|01} \gg p_{1|00}\approx 0, \;\text{ or }\; 1\approx p_{1|11} \gg p_{1|10},\,p_{1|01},\,p_{1|00}.$$

Here, the second set of parameters just arised from renaiming zeros and ones. The contact process famously exhibits phase transition. The critical manifold separates the parameters for which the population dies out from the parameters for which the population spreads infinitely with positive probability. For details we refer to Figure~\ref{fig:contact_process}.
\begin{figure}[H]
        \centering
       \includegraphics[scale=0.13]{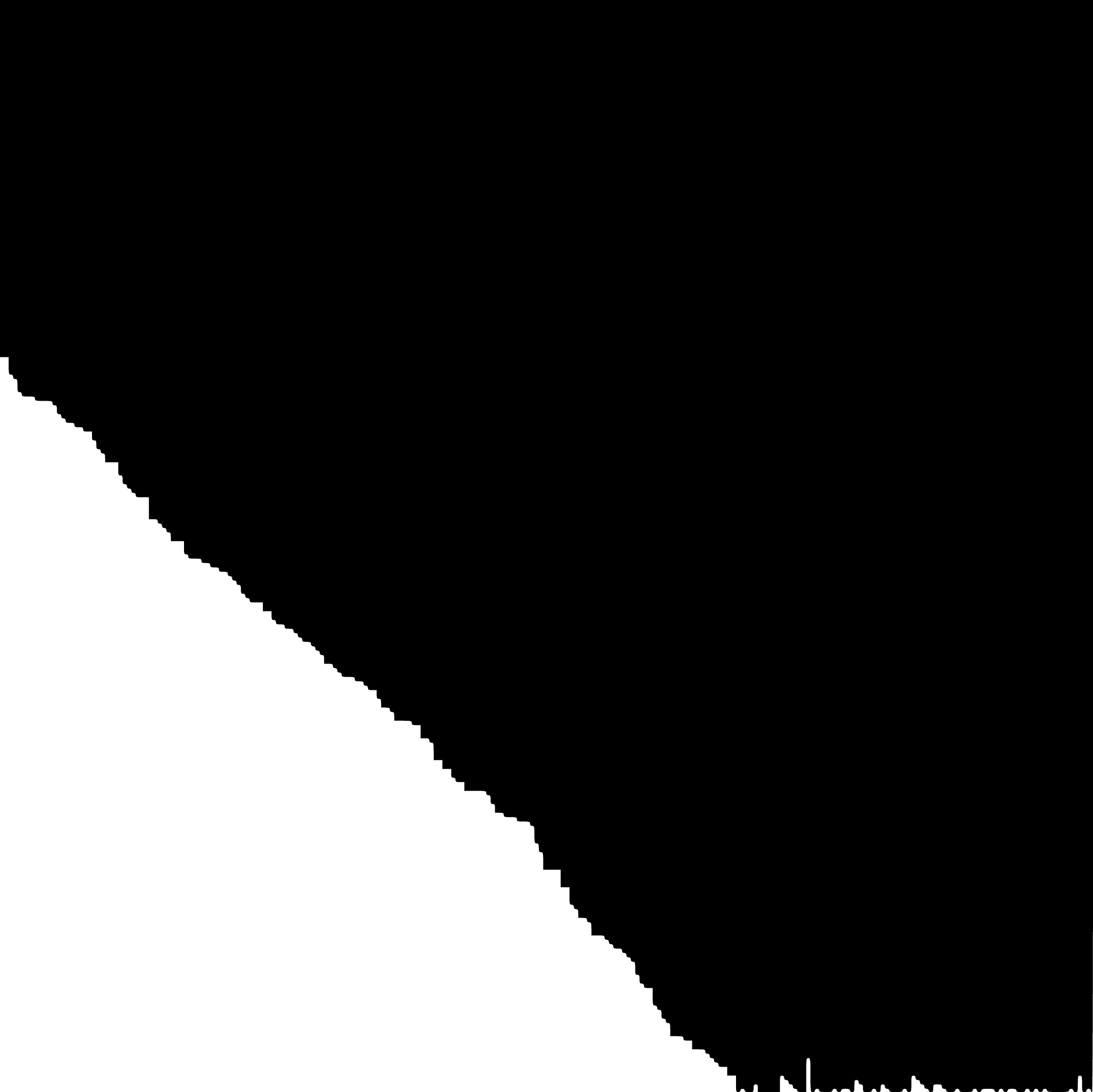}
       \hfill
       \includegraphics[scale=0.13]{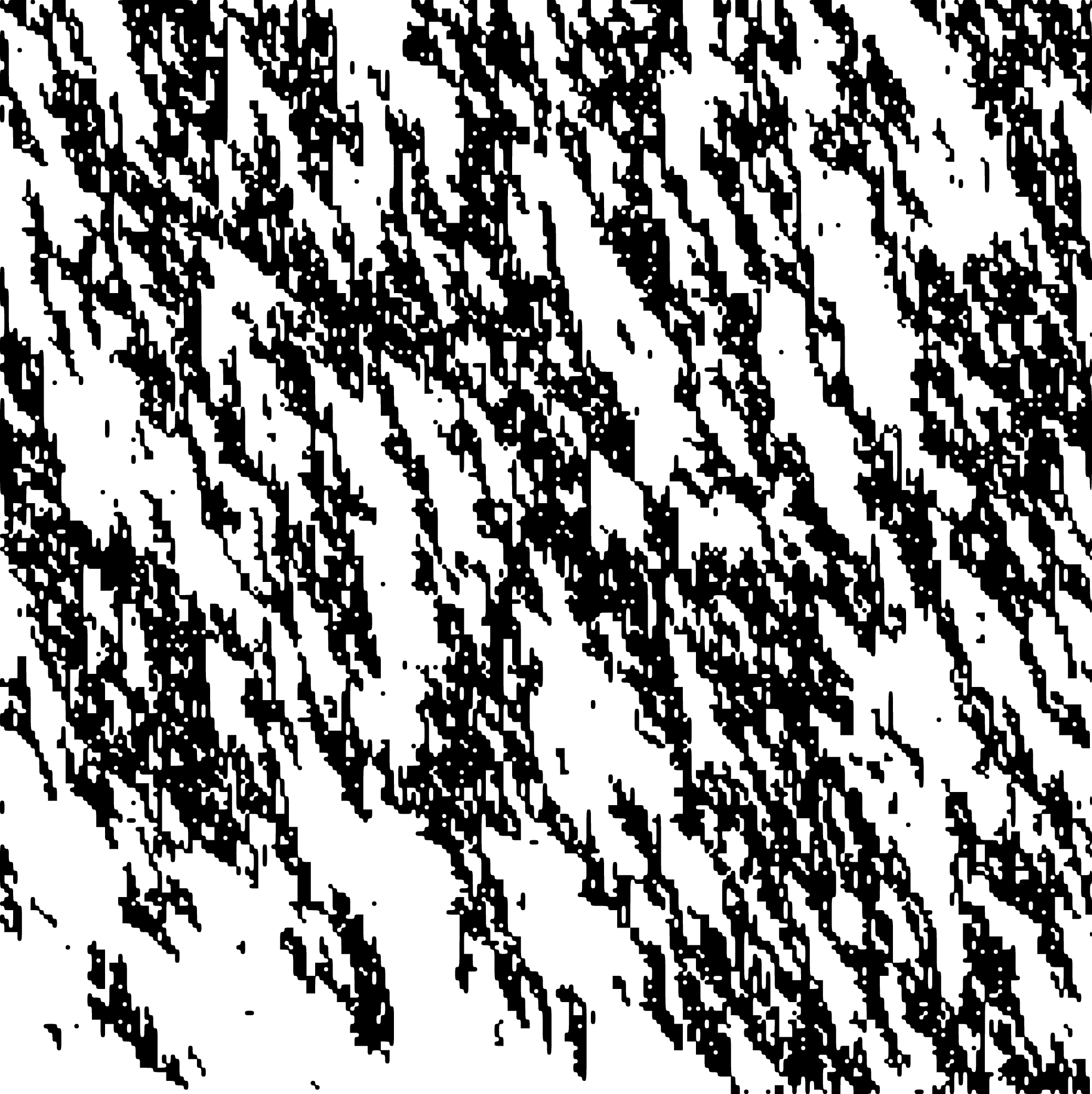}\\
       
       \includegraphics[scale=0.18]{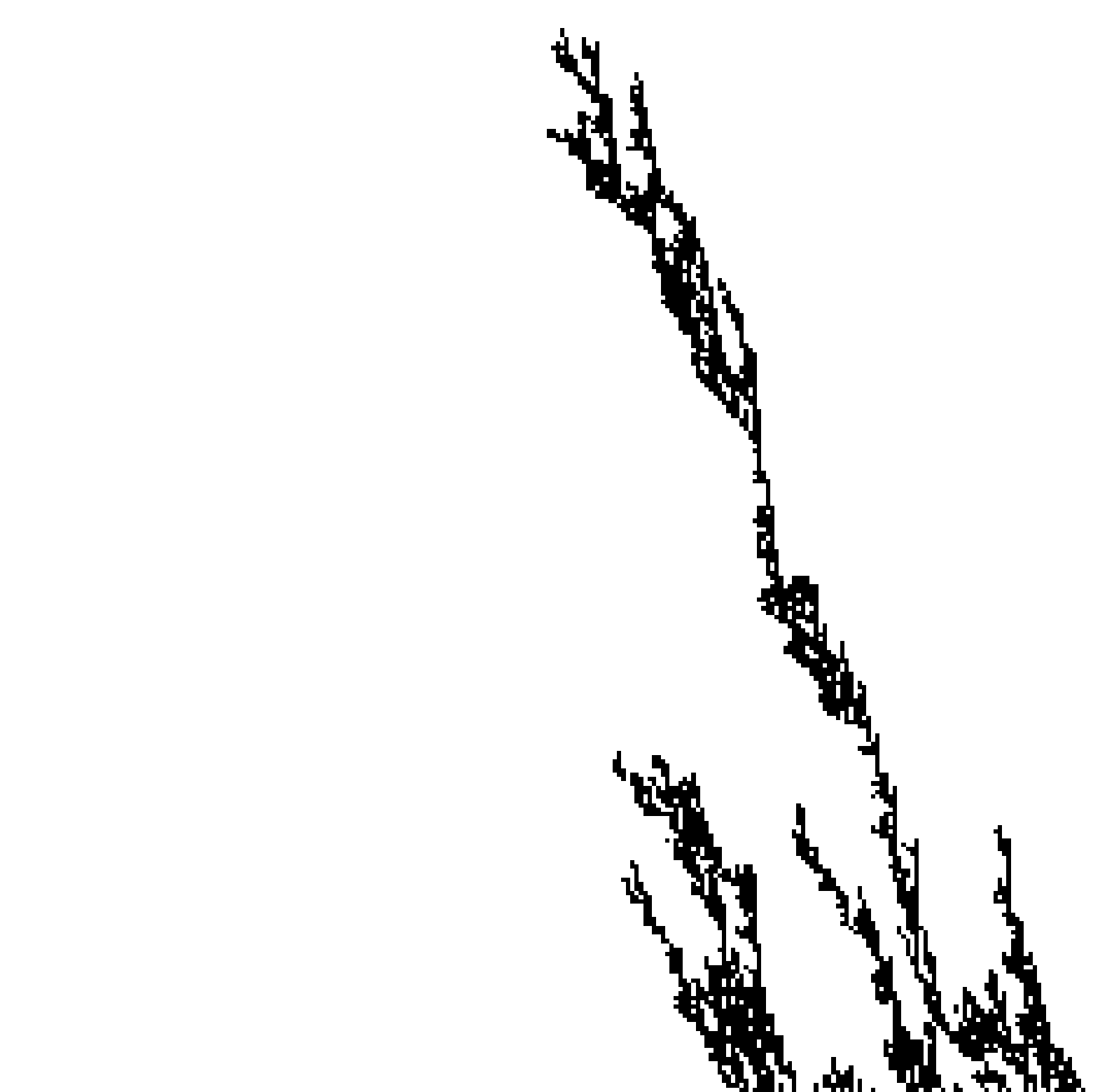}
       \hfill
       \includegraphics[scale=0.13]{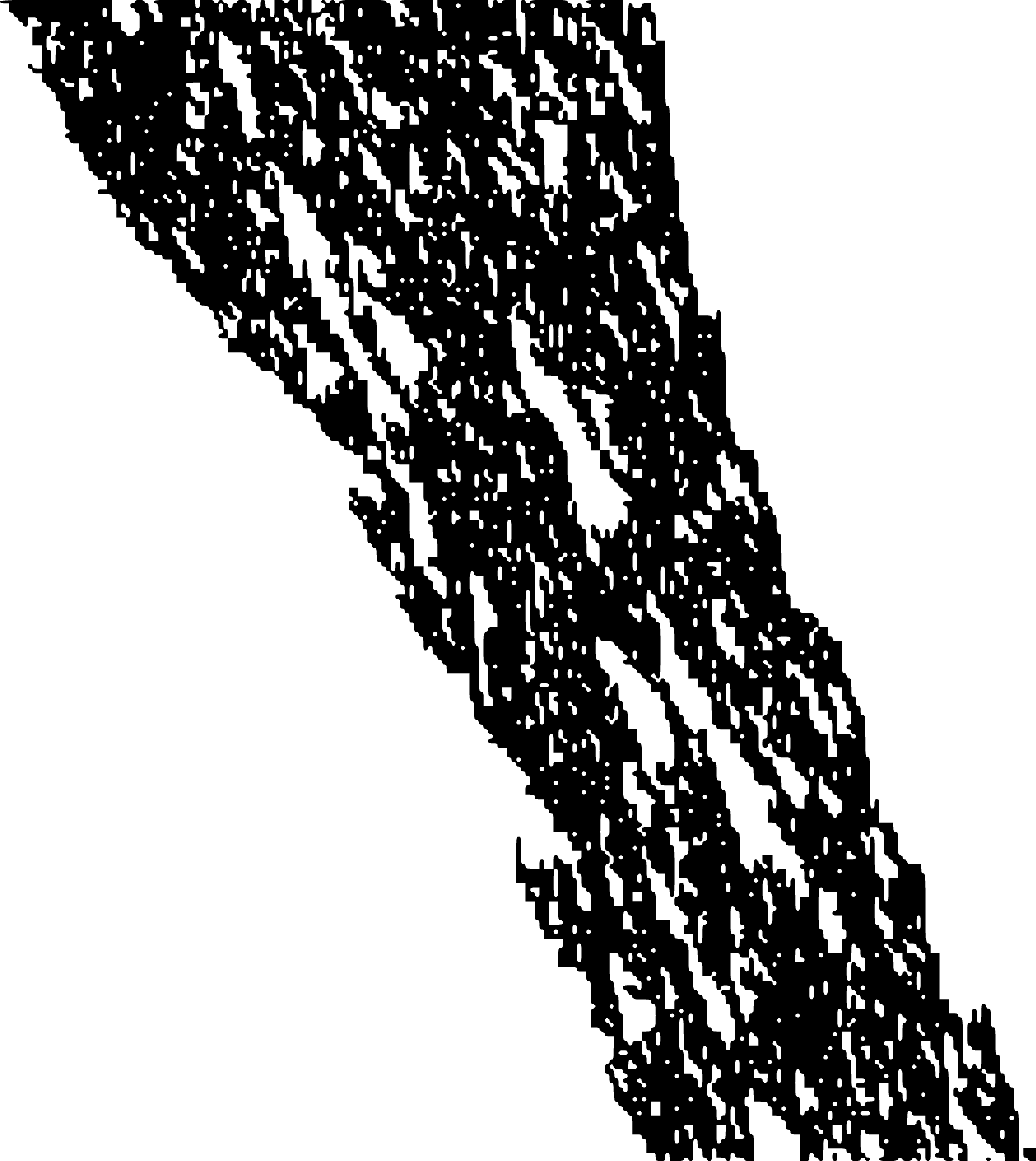}
        \caption{Evolution of the Coalescing Contact Process. In the upper left is the faultless version with parameters \mbox{$(p_{1|11},\,p_{1|10},\,p_{1|01},\,p_{1|00})=(1,\,0.7,\,0.7,\,0)$}. In the upper right the version with a positive probability of spontaneous deaths as well as births, given by parameters \mbox{$(0.75,\,0.75,\,0.75,\,0.01)$}.
        At the bottom we have the trajectories without spontaneous births but with spontaneous deaths, which represents the classic Contact Process, with two distinct behaviors, subcritical \mbox{$(0.75,\,0.75,\,0.75,\,0)$} and supercritical \mbox{$(0.8,\,0.8,\,0.8,\,0)$}. We note that for the pase transition one needs that~$p_{1|00} = 0$, which means that the underlying IPS do not have positive rates. The initial configuration consists of a randomly generated configuration on an interval on the right and zeroes everywhere else.} \label{fig:contact_process}
    \end{figure}

The Annihilating Contact Process is similar. The difference is if a site attempts to spread to an already occupied site it will be annihilated instead. This rule is described by parameters
$$\,p_{1|10},\,p_{1|01} \gg p_{1|11}\approx p_{1|00}\approx 0, \;\text{ or }\; 1\approx p_{1|11}\approx p_{1|00} \gg p_{1|10},\,p_{1|01}.$$
Like the coalescing variant, the Annihilating Contact Process exhibits a phase transition.
\begin{figure}[H]
        \centering
       \includegraphics[scale=0.1825]{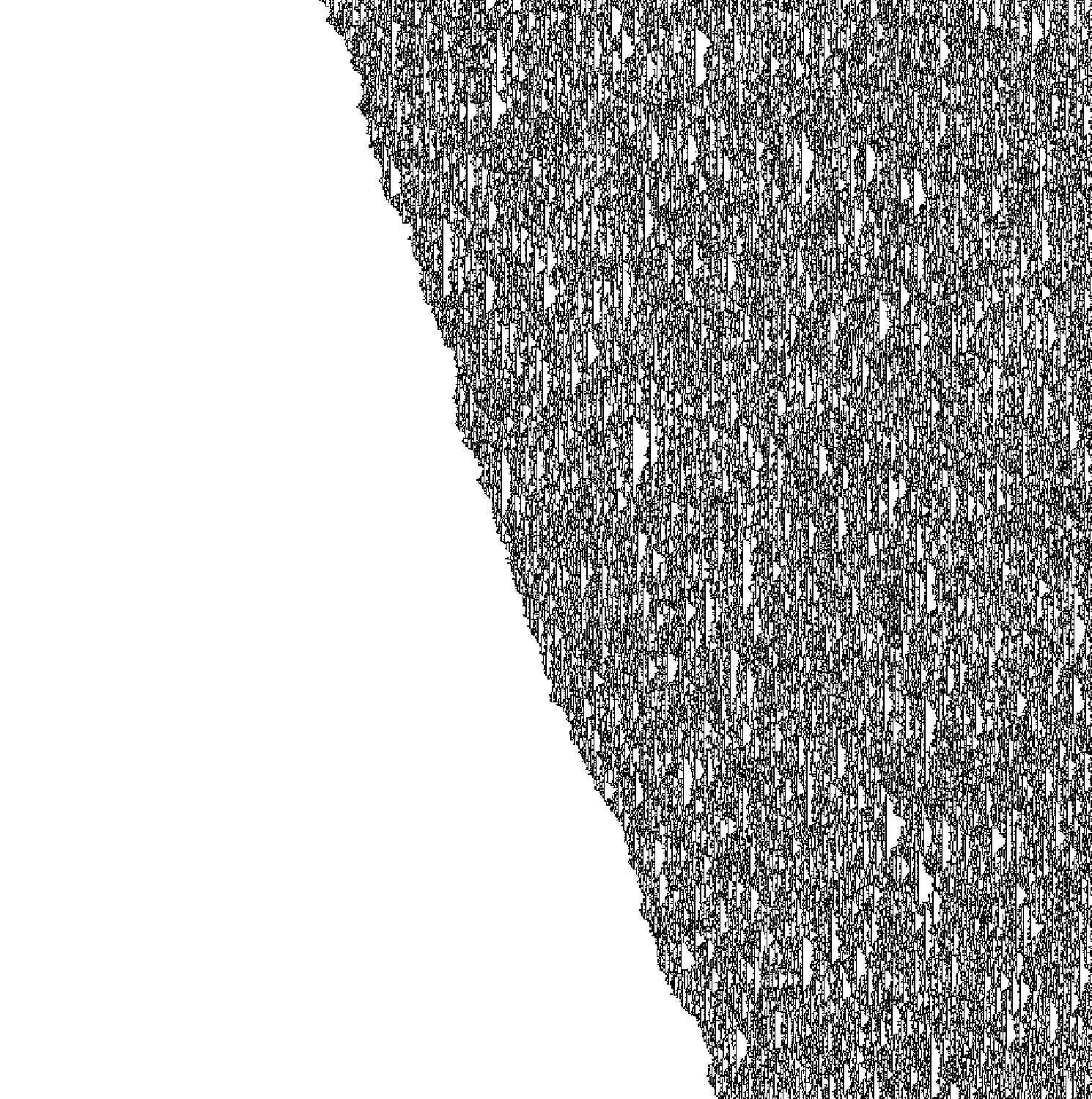}
       \hfill
       \includegraphics[scale=0.175]{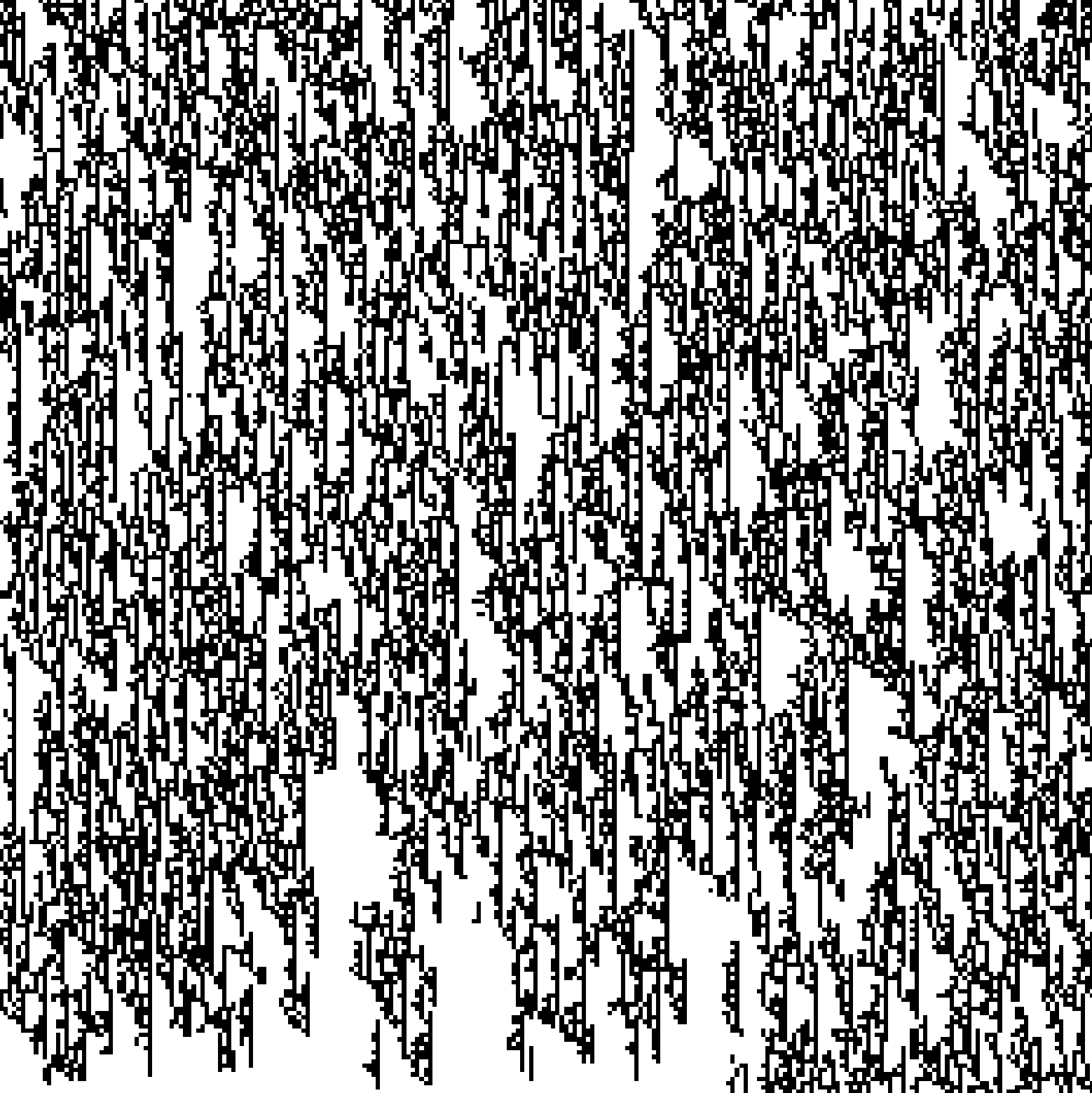}\\

       \includegraphics[scale=0.175]{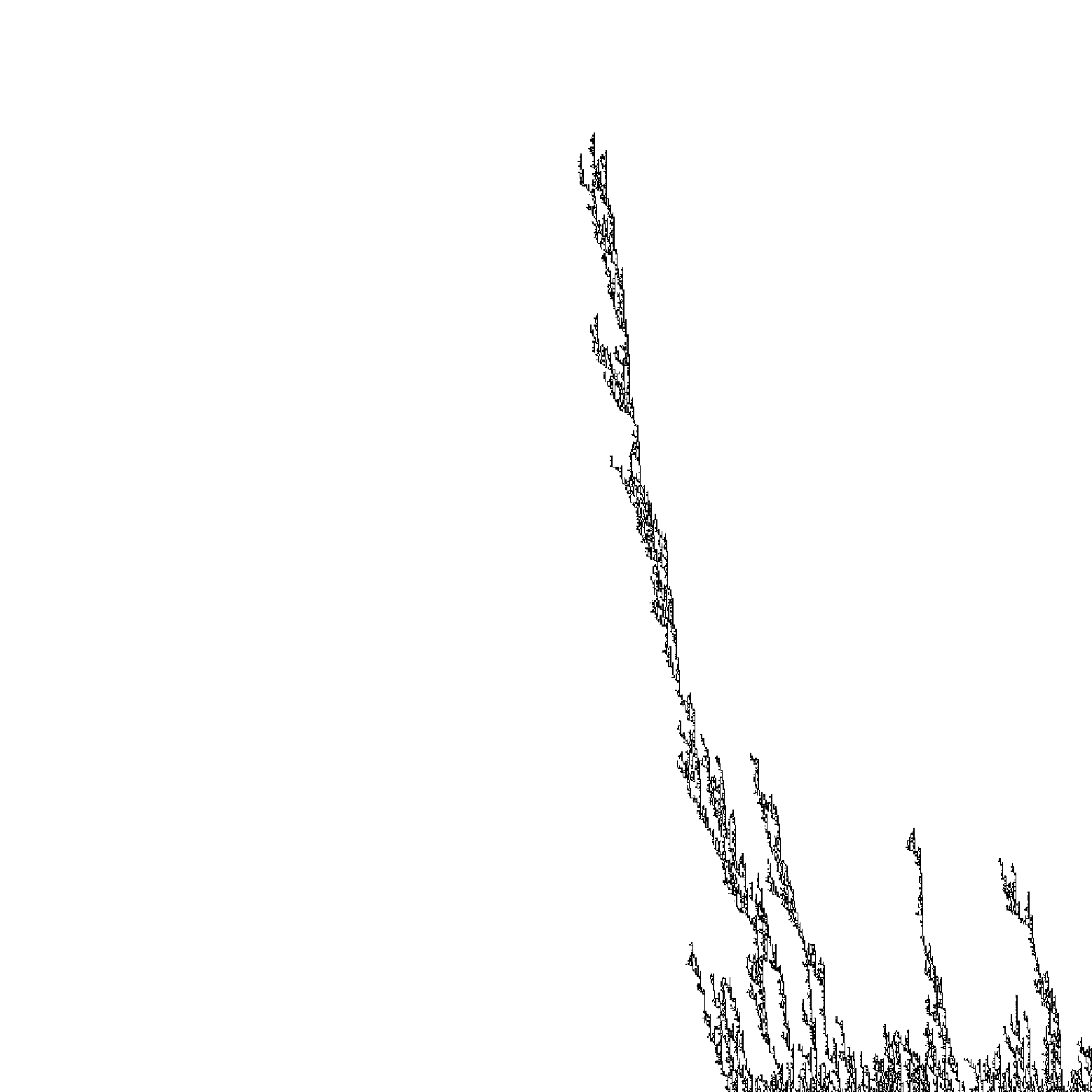}
       \hfill
       \includegraphics[scale=0.175]{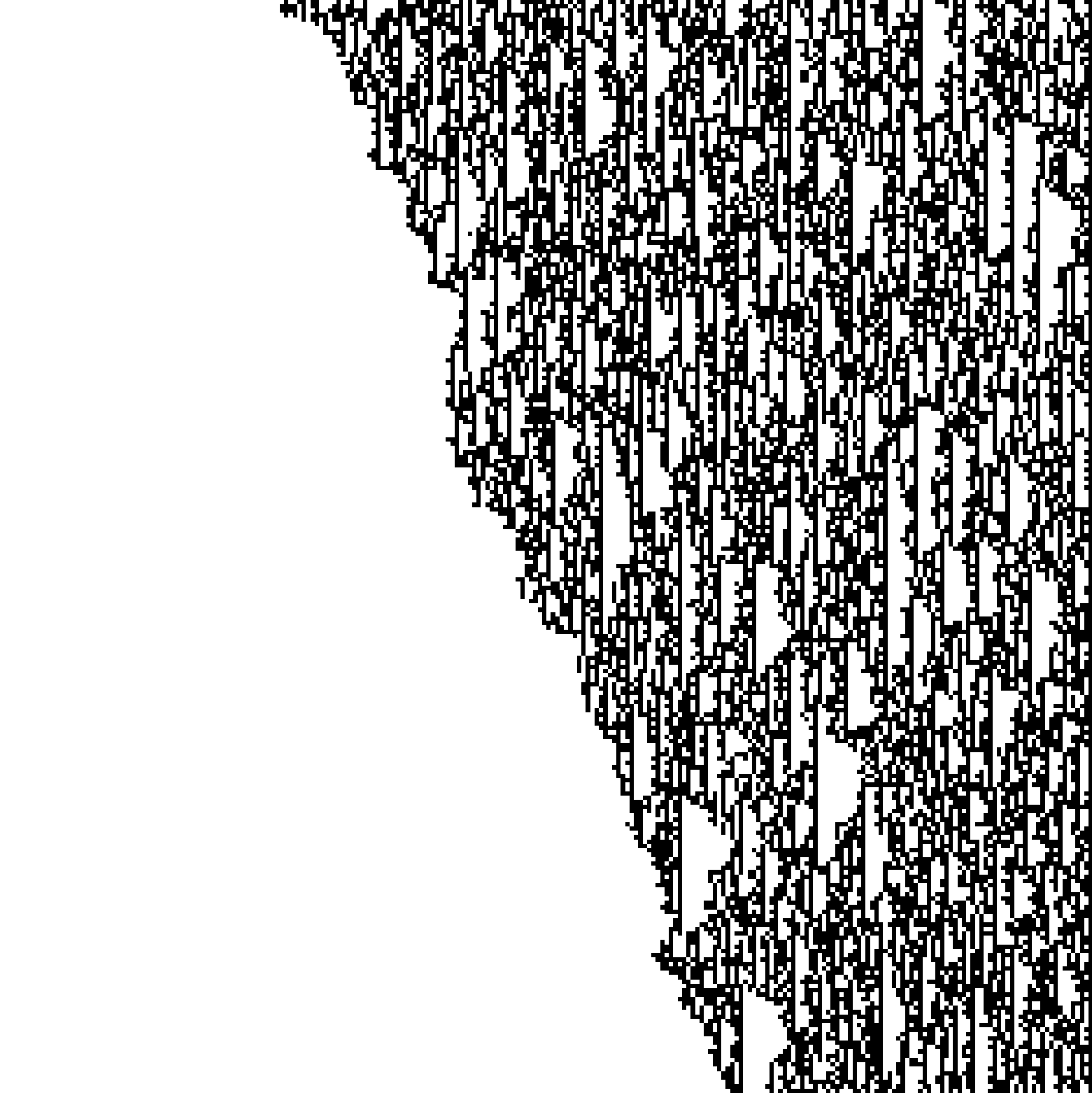}
        \caption{Evolution of the Annihilating Contact Process. In the upper right is the faultless version with parameters \mbox{$(p_{1|11},\,p_{1|10},\,p_{1|01},\,p_{1|00})=(0,\,1,\,1,\,0)$}. In the upper left the version with a positive probability of spontaneous births, given by parameters \mbox{$(0.01,\,0.9,\,0.9,\,0.01)$}.
        At the bottom, we have the trajectories without spontaneous births, which represents the classic Annihilating Contact Process, with two distinct behaviors, subcritical \mbox{$(0,\,0.9,\,0.9,\,0)$} and supercritical \mbox{$(0,\,0.95,\,0.95,\,0)$}. We observe that phase transition only occurs if the rates are not positive i.e. when~$p_{1|00} =0$. The initial configuration consists of a randomly generated configuration on an interval on the right and zeroes everywhere else.}
    \end{figure}

    If there are no spontaneous births (i.e. $p_{1|00}=0$ or $p_{1|11}=1$) then the ergodicity of these processes is equivalent to the population dying out almost surely. The work of Griffeath~\cite{Griffeath:78} on additive and cancellative IPSs suggests that these IPSs are exponentially ergodic if spontaneous births are possible (see Theorem~\ref{pro:additive} and Theorem~\ref{pro:cancellative} below). 

\subsection{Noisy IPS}
These are the IPS in which sites flip their states with high probability with each update, with little influence from neighbors. The parameters of these IPS satisfy
\begin{align*}
 & p_{1|11} \approx p_{1|10}\approx 0,\; p_{1|01}\approx p_{1|11}\approx 1 \;\text{ or }\; p_{1|11}\approx 0,\; p_{1|10}\approx p_{1|01}\approx p_{1|11}\approx 1 \\
 & \;\text{ or }\; p_{1|11}\approx p_{1|10}\approx p_{1|01}\approx 0,\; p_{1|00}\approx 1.
\end{align*}
As a consequence of this large frequency of flips, whether a given site has state zero or one largely depends on the parity of the number of updates at that site which is independent of the initial configuration. Thus exponential ergodicity follows.
\begin{figure}[H]
        \centering
       \includegraphics[scale=0.13]{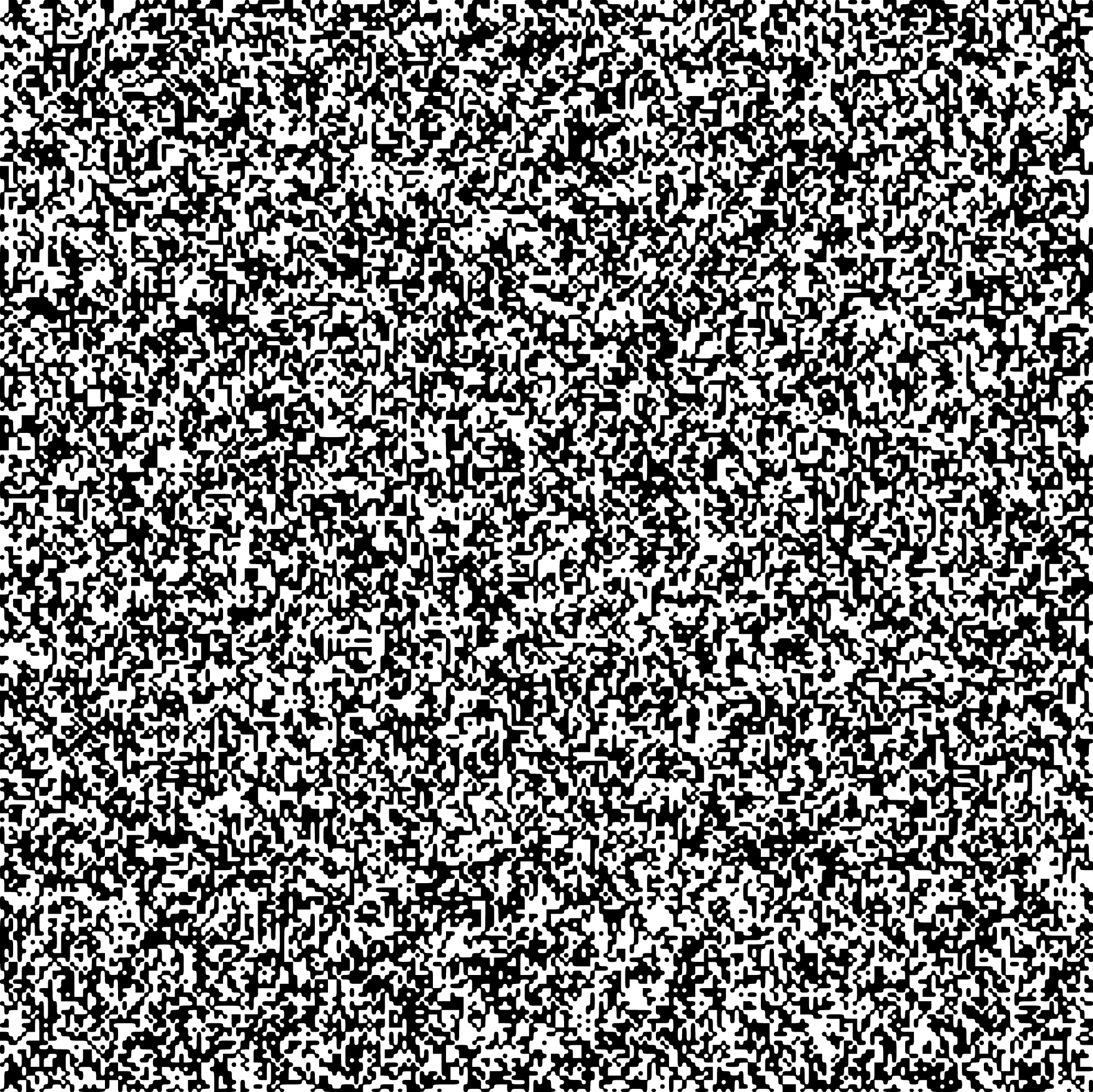}
       \hfill
       \includegraphics[scale=0.175]{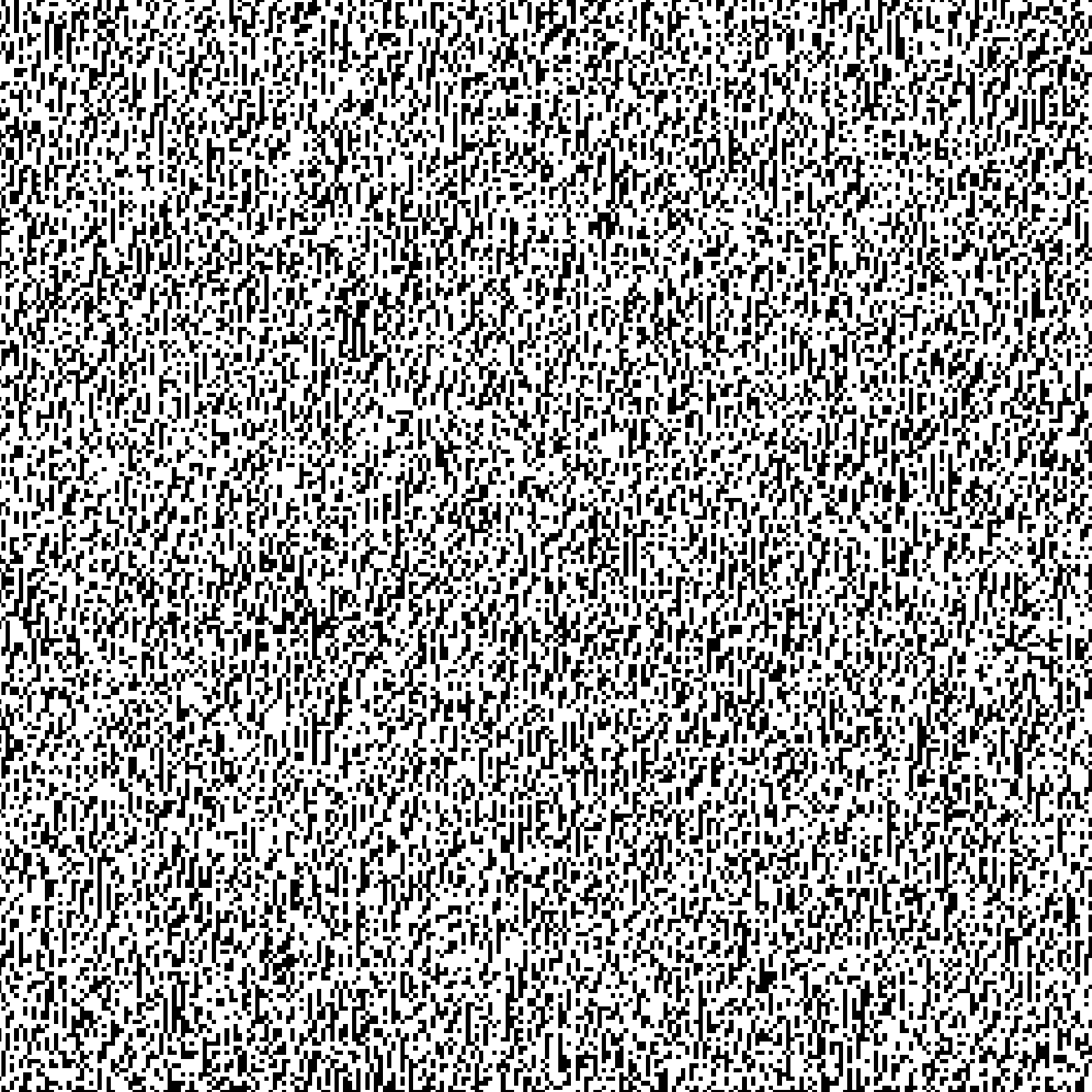}
        \caption{On the left evolution of the "flip your state" IPS i.e. one with parameters \mbox{$(p_{1|11},\,p_{1|10},\,p_{1|01},\,p_{1|00})=(0,\,0,\,1,\,1)$}. On the right the IPS \mbox{$(0,\,1,\,1,\,1)$}. Both produce a pattern of white noise, with the latter having slightly higher correlations between states of neighboring sites.}
    \end{figure}

\subsection{Walls IPS}\label{subsec:walls}

The rule of these IPS preserves one state with high probability regardless of the environment and preserves the other state only in presence of the first. The trajectories consist of long walls which depending on parameters coalesce into one infinite connected component. These IPS are given by parameters
$$p_{1|11} \approx p_{1|01}\approx p_{1|00}\approx 0,\; p_{1|10}\approx 1 \;\text{ or }\; p_{1|11}\approx p_{1|10}\approx p_{1|00} \approx 1,\; p_{1|01}\approx 0.$$
    \begin{figure}[H]\label{fig:walls}
        \centering
         \includegraphics[scale=0.175]{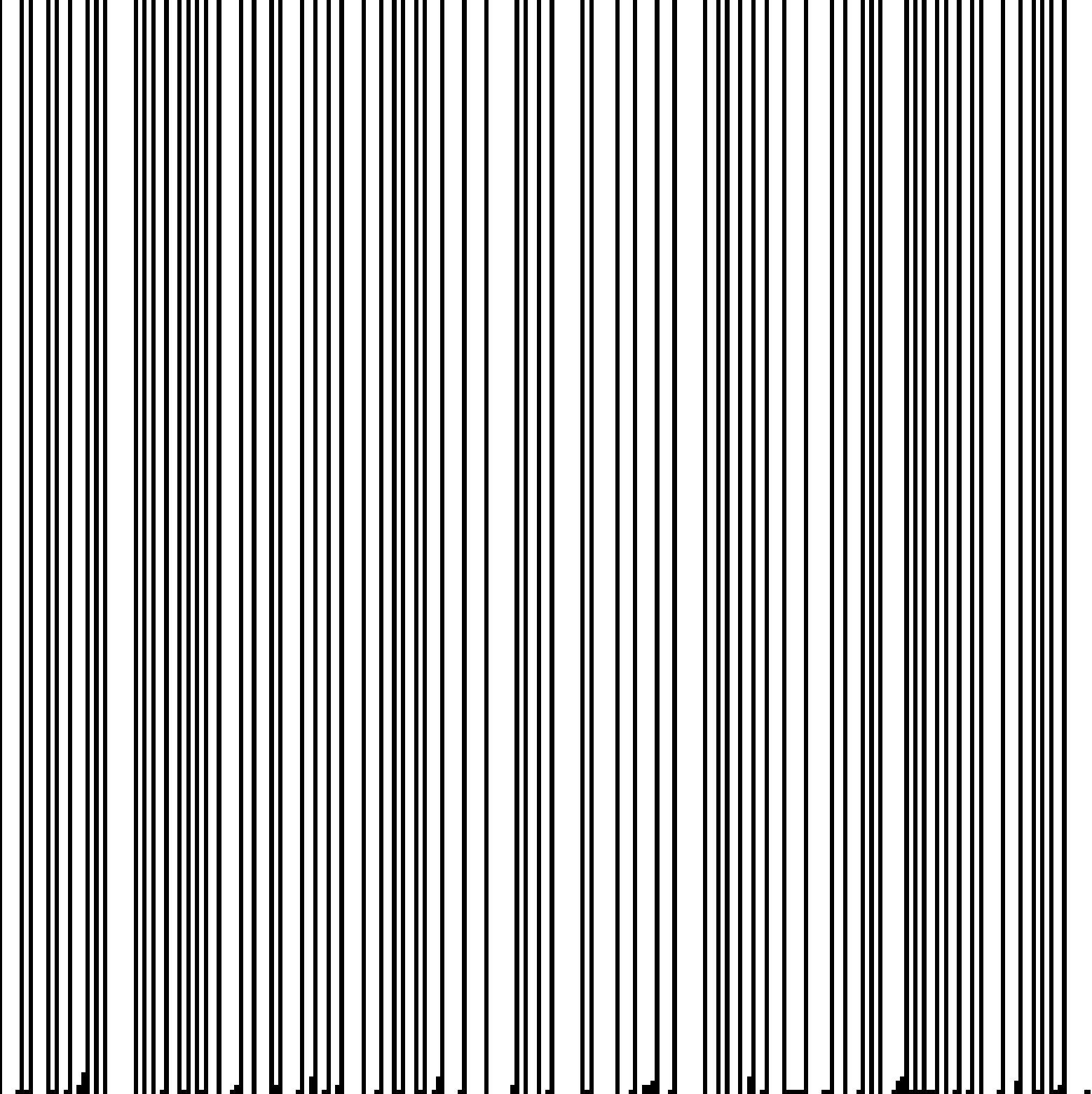}\\
       \includegraphics[scale=0.175]{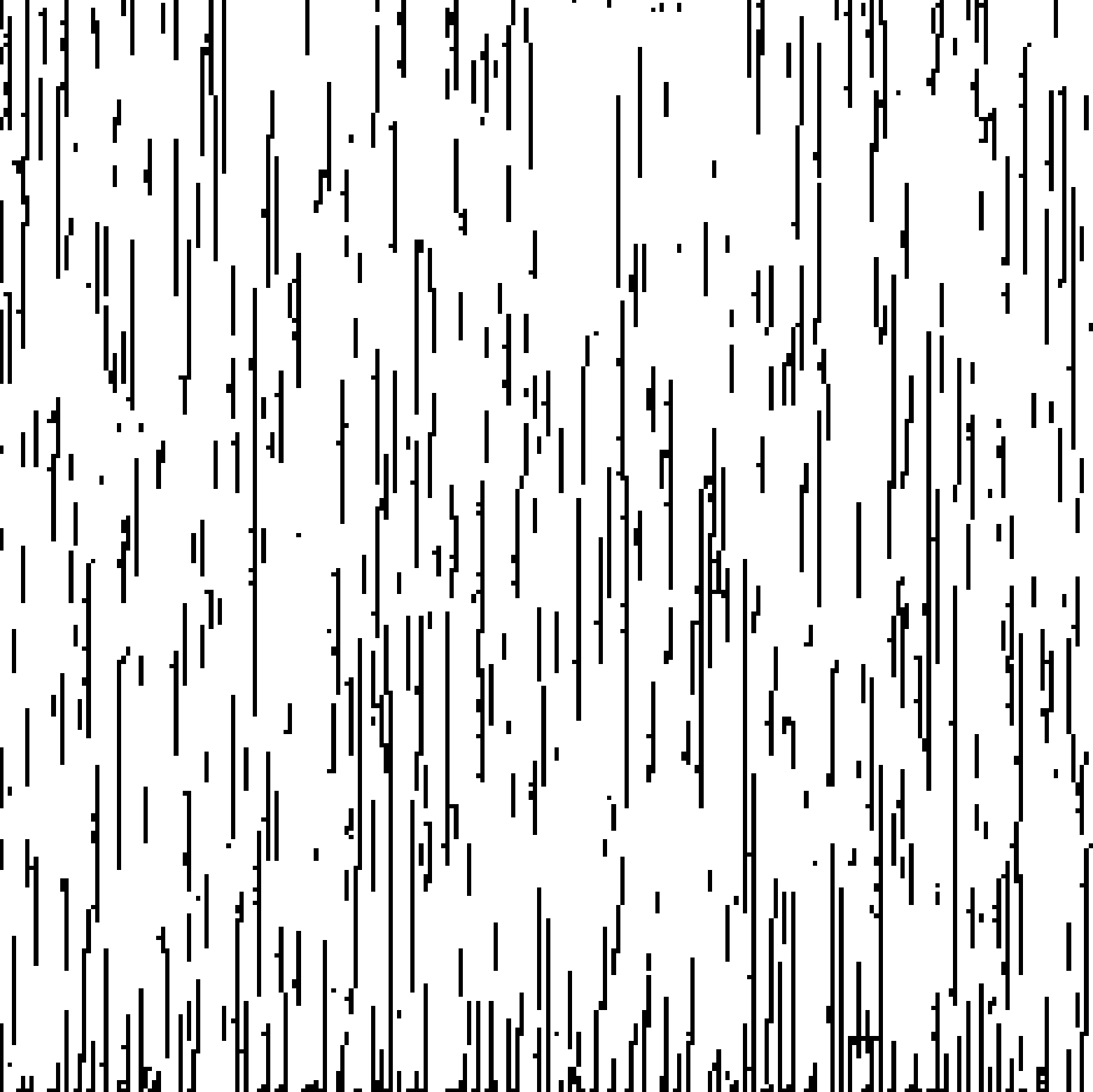}\hfill
       \includegraphics[scale=0.175]{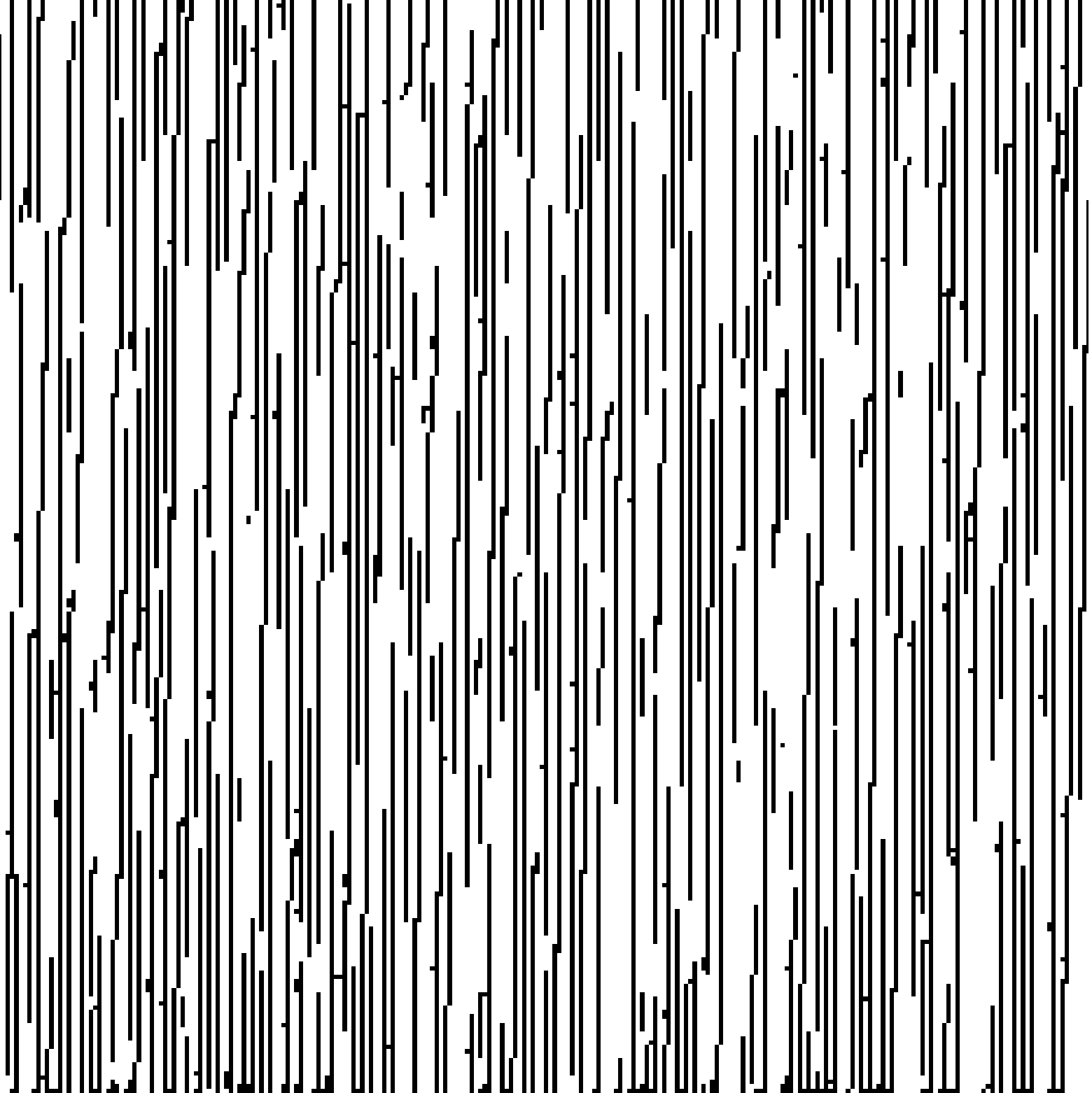}
        \caption{Trajectories of the Walls IPSs given by parameters \mbox{$(p_{1|11},\,p_{1|10},\,p_{1|01},\,p_{1|00})=(0,\,1,\,0,\,0)$}(top), \mbox{$(0,\,0.9,\,0.02,\,0.02)$} (left) and \mbox{$(0,\,0.99,\,0.02,\,0.02)$} (right). The longer survival times of ones in the IPSs on the right allowed the walls to connect into one infinite connected component. In the most extreme case on top, the trajectory quickly reaches a stationary configuration. Any configuration in which there are no consecutive ones is stationary.}
    \end{figure}

The ergodicity of these IPSs in low noise regimes remains an open problem. There is some overlap between Walls IPS and cancellative IPS and in some cases, they can be modified to equivalent IPSswith periodic monotone rules. We elaborate on this further in Section~\ref{sec:application}.

\bigskip \section{Time Scaling}\label{sec:time-scaling}

%weitermachen

The main goal of this article is to study the ergodicity of IPS. It natural question to ask under which transformations of the transition matrix ergodicity is preserved. There is an obvious invariant: scaling time cannot change ergodicity. The main result of this section, namely the Time-Scaling Lemma, shows that scaling time is equivalent to taking a convex combination of the original transition matrix with the identity matrix.
\\
\newpage
\begin{lem}[Time-Scaling]\label{pro:time-scaling} Let us consider a given transition matrix~$P$ (cf.~Definition~\ref{def:IPS}) and let~$I$ denote the identity transition matrix. Then for all $\lambda \in (0,1]$ it holds:
$$\bb{\zeta_t}_{t\geq 0} \in \text{IPS}\bb{P} \text{ if and only if } \bb{\zeta_{\lambda t}}_{t\geq 0} \in \text{IPS}\bb{\lambda P + (1-\lambda))I}.\\$$
\end{lem}

This lemma implies that the set of \hyperref[def:invariant measure]{invariant measures} and properties like \hyperref[def:ergodicity]{ergodicity} are preserved under scaling time. This allows to extend past results about these properties (see Section~\ref{sec:extensions} below). For example, the following statement is an immediate consequence of the Time-Scaling Lemma. \\

\begin{coro}
    \hyperref[def:ergodicity]{Ergodicity} and the set of \hyperref[def:invariant measure]{invariant measures} are the same for IPS$(P)$ and IPS$(\lambda P + \bb{1-\lambda}I)$. The IPS$(P)$ is exponentially ergodic with rate $\alpha$ if and only if IPS$(\lambda P + \bb{1-\lambda}I)$ is exponentially ergodic with rate $\lambda\alpha$.
\end{coro}

\begin{proof}[Proof of Lemma~\ref{pro:time-scaling}]

    The idea is to split the points in the $\mathbf{Update}$ set of IPS$(\lambda P + \bb{1-\lambda}I)$ into two categories - those where the IPS is updated according to the transition matrix $P$ and those where it is updated by $I$. Updating by the identity matrix is equivalent to not updating at all, so after erasing those points from the $\mathbf{Update}$ set, the trajectory remains the same. However, now all remaining updates are governed by the transition matrix $P$ and the $\mathbf{Update}$ set has become more sparse, which corresponds to time slowing down.\\

    The proof consists of three steps. First, we assign a Bernoulli random variable to each point in $\mathbf{Update}$, which determines whether a particular update is governed by $P$ or $I$. Then we construct a new trajectory $\xi$ which skips the 'failed' updates and we confirm that it behaves as a trajectory of an IPS with transition matrix $P$. The last step is to show that the new sparser $\mathbf{Update}$ set has exponentially distributed increments between points.\\

Since the statement of the lemma concerns only the distribution of $\bb{\zeta_t}_{t\geq 0}$ let us construct this IPS in the following way:  First, we construct \hyperref[def:general construction]{trajectory}  $\xi_t$ of $\text{IPS}\bb{\lambda P + (1-\lambda))I}$.
Let $\bb{\rho_{n}^j}_{n\geq 0}^{j\in \Lambda}$ be i.i.d.~and distributed according to $\text{Exp}(1)$
and set \mbox{$\tau_{n}^j := \rho_1^j +\dots \rho_n^j$}.
Define additional i.i.d.~$\text{Bern}(\lambda)$ variables $\bb{V_n^j}_{n\geq 0}^{j\in \Lambda}$ and independent from them i.i.d.~$\text{Unif}[0,1]$ variables $\bb{U_n^j}_{n\geq 0}^{j\in \Lambda}$.
Now set:\smallskip

\begin{itemize}
\item The initial configuration $\xi_0 := \zeta_0$.
\item If $t\in \left[\tau_n^j, \tau_{n+1}^j\right)$ then $\xi_t(j) := \xi_{\tau_n^j}(j)$.
\item Via the recursion principle we define $$\xi_{\tau_{n+1}^j}(j) := V_{n+1}^j \cdot F_j\bb{\xi_{\tau_{n+1}^{j-}}\bb{\mathcal{N}_j}, U_{n+1}^j}+\bb{1- V_{n+1}^j}\cdot\xi_{\tau_{n}^j}(j),$$
where $F_j: \mathcal{A}^\Lambda\times \Cb{0,1}\rightarrow \mathcal{A}$ are such that
$$\pP\bb{F_j\bb{\sigma,U_n^j} = a} = P_j\bb{a\;|\;\sigma}$$
for any $j\in\Lambda$, $a\in \mathcal{A}$ and $\sigma\in \mathcal{A}^\Lambda$.
\end{itemize}
The random variables $V_n^j$ act as blockers - when a clock rings at site $j$ it will fail to update (or rather, be updated by $I$) with probability $1-\lambda$. By construction $$\bb{\xi_t}_{t\geq 0}\in \text{IPS}\bb{\lambda P + (1-\lambda)I}.$$
Now define stopping times $\bb{\pi_{n}^j}_{n\geq 0}^{j\in \Lambda}$ by $\pi_0^j := 0$ and
$$\pi_{n+1}^j := \min\cb{\tau_k^j>\pi_n^j\,:\, V_k^j = 1}.$$
Put simply $\bb{\pi_n^j}_{n\geq 0}$ skips the failed updates.
The \hyperref[def:general construction]{trajectory}  $\bb{\xi_t}_{t\geq 0}$ satisfies: \smallskip

\begin{itemize}
    \item The initial configuration is still $\xi_0 = \zeta_0$.
    \item If $t\in \left[\pi_n^j, \pi_{n+1}^j\right)$ then $\xi_t(j) = \xi_{\pi_n^j}(j)$. 
    \item The recursion principle $$\xi_{\pi_{n+1}^j}(j) = F_j\bb{\xi_{\pi_{n+1}^{j-}}\bb{\mathcal{N}_j}, U_{n+1}^j}.$$
\end{itemize}
Therefore the trajectory $\bb{\xi_t}_{t\geq 0}$ is being updated according to transition matrix $P$. The last thing to check is that the family $\bb{\pi_n^j}_{n\geq 0}$ forms the $\mathbf{Update}$ set for some exponential clock, i.e.~that the increments $\pi_{n+1}^j-\pi_{n}^j$ are independent and distributed exponentially with the same rate (c.f.~Subsection~\ref{subsec:space-time}).
Let $\mathcal{F}_t$ be the canonical $\sigma$-algebra of events up to time $t$. Direct calculation yields
\begin{align*}
\pP & \bb{\pi_{n+1}^j - \pi_{n}^j>t\,|\,\mathcal{F}_{\pi_{n}^j}} = \sum_{k=0}^\infty \ind \bb{{\pi_{n}^j = \tau_{k}^j}}\CE{\ind\bb{\pi_{n+1}^j - \pi_{n}^j>t}}{\mathcal{F}_{\pi_{n}^j}}   \\
& = \sum_{k=0}^\infty \ind\bb{{\pi_{n}^j = \tau_{k}^j}} 
\\
& \qquad \times \CE{\sum_{m=1}^\infty \prod_{l=1}^{m-1}\ind\bb{V_{k+l}^j=0}\cdot\ind\bb{V_{k+m}^j=1}\cdot \ind\bb{\rho_{k+1}^j+\dots+\rho_{k+m}^j>t}}{\mathcal{F}_{\pi_{n}^j}}.
\end{align*}

Variables in the expectation are independent from the $\sigma$-algebra on the event $\cb{\pi_{n}^j = \tau_{k}^j}$. Additionally, since $V's$ are i.i.d. $\text{Bern}(\lambda)$ and $\rho's$ are i.i.d. $\text{Exp}(1)$ variables then the above equals

\begin{align*}
    \sum_{k=0}^\infty & \ind\bb{{\pi_{n}^j = \tau_{k}^j}} \sum_{m=1}^\infty \lambda(1-\lambda)^{m-1}\pP\bb{\text{Gamma}(m,1)>t} \\
    & = \sum_{m=1}^\infty \lambda(1-\lambda)^{m-1}\pP\bb{\text{Gamma}(m,1)>t}\\
    & =\sum_{m=1}^\infty \lambda(1-\lambda)^{m-1}\int_{[t,\infty)}\frac{1}{(m-1)!}x^{m-1}e^{-x}dx \\
    &= \int_{[t,\infty)}\lambda e^{-x} \sum_{m=0}^\infty \frac{\bb{x(1-\lambda)}^m}{m!}dx \\
    & = \int_{[t,\infty)}\lambda e^{-\lambda x}dx = \pP\bb{\text{Exp}(\lambda)>t}.
\end{align*}

Since the result is a constant with respect to $\mathcal{F}_{\pi_{n}^j}$ then the increment $\pi_{n+1}^j - \pi_{n}^j$ is independent from that $\sigma$-algebra. In addition we have shown that its distribution is $\text{Exp}(\lambda)$. As argued before this implies that
$$\bb{\xi_{t}}_{t\geq 0} \in \text{IPS}\bb{P, \text{Exp}(\lambda)}.$$
Setting $\zeta_t = \xi_{\frac{\lambda}{ t}}$ yields the desired result.
\end{proof}

The following example shows that the Time-Scaling Lemma does not hold for PCAs.\\
\begin{ex}
    Consider $\Lambda=\Z^d, \mathcal{A} = \cb{-1,+1}$ and transition matrices $P,Q$
    $$P_j(+1\;|\;\sigma) = \ind\bb{\sigma(j)=-1} \qquad \mbox{and} \qquad Q_j(+1\;|\;\sigma) = \frac{1}{2}.$$
    The rule $P$ flips states deterministically while $Q$ chooses it uniformly at random. By the \hyperref[pro:time-scaling]{Time-Scaling Lemma}, IPS$(P)$ is just $\text{IPS}(Q)$ with time sped up by a factor $\lambda = \frac{1}{2}$. The same is not true for $\text{PCA}(P)$ and $\text{PCA}(Q)$.
    Notice that $\text{PCA}(P)$ has a periodic trajectory "all +1" followed by "all -1", while $\text{PCA}(Q)$ turns every initial distribution into the $\text{Bern}\bb{\frac{1}{2}}$ product measure after a single time-step.
\end{ex}

\bigskip 
\section{Extensions of previous results via the Time-Scaling Lemma}\label{sec:extensions}

The \hyperref[pro:time-scaling]{Time-Scaling Lemma} implies that properties of an IPS that are invariant under scaling time are the same for all IPS whose transition matrices lie on a line crossing the identity matrix. These properties include the invariance of a given measure and (exponential) \hyperref[def:ergodicity]{ergodicity}. Thus to prove that an IPS$(P)$ has any of these properties it suffices to find some other $\text{IPS}(Q)$, with the transition matrix of the form $Q = \lambda P +\bb{1-\lambda}I$, which has it. This allows to extend previous results that do not respect this invariance with little additional effort. The strategy for this is straightforward: If a theorem states that IPS in some subset of the parameter space are \hyperref[def:ergodicity]{ergodic} then we can extend this subset along lines with roots at identity and have \hyperref[def:ergodicity]{ergodicity} proven in this (hopefully) larger subset. Similarly, this strategy also allows to proof theorems under stronger assumptions than stated. \\

In Section~\ref{sec:monotone_IPS}, we consider monotone IPS. We introduce the concept of weak and strong monotonicity. We show that in order to deduce Gray's theorem (see Theorem~\ref{Gray} below) one could assume w.l.o.g.~strong monotonicity. This potentially simplifies the proof of the theorem, especially when attempting to deduce similar statements for larger alphabets. \\

In Section~\ref{sec:ergo_additive_cancellative}, we consider cancellative and additive IPS, and show how two theorems of Griffeath (see Theorem~\ref{pro:additive} and Theorem~\ref{pro:cancellative} from below) can be extended using the Time-Scaling Lemma. \bigskip

\subsection{General results about monotone IPS and Gray's theorem}\label{sec:monotone_IPS}

We start with reviewing a couple of results on \hyperref[def:monotonicity]{monotone} IPS. This is a well-studied type of IPS and much is known about their behavior. Examples include the voter model, contact process, and ferromagnetic stochastic Ising models. For a standard reference on mononte IPS with alphabets of size two we refer to~\cite{Ligett:05}. In our exposition, we generalize the concept of monotone IPS to finite alphabets. After reviewing some basic facts on monotone IPS we introduce the concept of weak monotonicity. Then, we discuss Gray's theorem~(see Theorem~\ref{Gray} below) on weakly-monotone IPS, which plays an important role in Section~\ref{sec:application}. We close this section by showing that for the proof of Gray's theorem one can assume w.l.o.g.~strong monotonicity.\\

\begin{nota}
    For a transition matrix $P$ with a linearly ordered alphabet $\bb{\mathcal{A},\geq}$ we will write
    \begin{align*}
        P\bb{a_+\;|\;\zeta} &:= \sum_{b\geq a}P\bb{b\;|\;\zeta},\\
        P\bb{a_-\;|\;\zeta} &:= 1 - P\bb{a_+\;|\;\zeta}.
    \end{align*}
\end{nota}

\begin{defi}[Domination]\label{domination}
    For two transition matrices $P,Q$ over the same lattice $\Lambda$ and a linearly ordered alphabet $\bb{\mathcal{A},\geq}$ we say that $P$ \emph{dominates} $Q$ if for all $a\in\mathcal{A}$ it holds that for all~$j \in \Lambda$
    $$ P_j\bb{a_+\,|\,\zeta}\geq Q_j\bb{a_+\,|\,\xi} \text{ whenever }\zeta\geq\xi.$$
    We also say that an IPS$(P)$ \emph{dominates} $\text{IPS}(Q)$ if $P$ \hyperref[domination]{dominates} $Q$.\\
\end{defi}

\begin{defi}[Monotonicity]\label{def:monotonicity}
     A transition matrix is \emph{monotone} if 
     it \hyperref[domination]{dominates} itself.
     We will also say that an IPS is \emph{monotone} if its transition matrix is.\\
\end{defi}

We now discuss a few fundamental results regarding \hyperref[def:monotonicity]{monotone} IPS that provide great aid in their analysis. Even though the results are standard, we opted to provide simple proofs as that the literature only discusses alphabets of size two. Informally, the first lemma states that if one IPS \hyperref[domination]{dominates} another then their evolutions preserve stochastic dominance between their trajectories.\\

\begin{lem}[Domination lemma]\label{pro:domination lemma}
    Let IPS$(P)$ \hyperref[domination]{dominate} $\text{IPS}(Q)$. Then for any pair of initial distributions that satisfy $\zeta_0\geq \xi_0$ their trajectories $\bb{\zeta_t}_{t\geq 0}$ and $\bb{\xi_t}_{t\geq 0}$ under IPS$(P)$ and $\text{IPS}(Q)$ respectively can be constructed in such a way that
    $\zeta_t\geq \xi_t$ for all $t\geq 0$, almost surely.\\
\end{lem}

The proof consists out of a straight-forward coupling argument.
\begin{proof}[Proof of Domination Lemma]
    We show that the \hyperref[def:general construction]{construction} given in Section~\ref{sec:notation} satisfies this condition. First, we let $\zeta$ and $\xi$ share the clock and update random variables ($\tau$s and $U$s). Then wet set the update functions $\bb{f_j}_{j\in\Lambda},\bb{g_j}_{j\in\Lambda}$ to be the non-decreasing step functions on $[0,1]$ satisfying for all $\sigma\in \mathcal{A}^{\Lambda}$ and $a\in\mathcal{A}$
$$\pP \bb{f_j\bb{\sigma, U_n^j}=a} = P_j\bb{a\,|\,\sigma}\;\,$$
$$\pP \bb{g_j\bb{\sigma, U_n^j}=a} = Q_j\bb{a\,|\,\sigma}.$$
Then the $\zeta\geq \xi$ condition is preserved by each update.
\end{proof}

Armed with this lemma we now prove the first theorem about \hyperref[def:monotonicity]{monotone} IPS. It states that the distributions of trajectories of \hyperref[def:monotonicity]{monotone} IPS are themselves monotone. \\

\begin{nota}\mbox{}
\begin{itemize}
    \item When the lattice is known from the context then for any $a\in \mathcal{A}$ we will shorten the Dirac deltas
        $$\delta_a := \delta_{\cb{a}^{\Lambda}}.$$
    \item For a linearly-ordered alphabet, we will denote its minimal and maximal elements (if they exist) as $\ubar{a}$ and $\bar a$ respectively.
\end{itemize}  
\end{nota}
\mbox{}\\

\begin{theo}\label{pro:monoresult}
 If the IPS$(P)$ is \hyperref[def:monotonicity]{monotone} then the following holds:
 \begin{enumerate}
     \item For any increasing event $A$ (i.e.~one such that if $\zeta \geq \xi$ and $\xi \in A$ then $\zeta \in A$) the function
        $$t \mapsto \delta_a P_t\bb{A}$$
        is non-increasing for $a=\bar{a}$ and non-decreasing for $a = \ubar{a}$ (provided these exist).
    \item  The weak limits of $\bb{\delta_a P_t}_{t\geq 0}$ for $a\in\cb{\ubar{a},\bar{a}}$ exist. These limits will  be denoted as $\mu_{\ubar{a}}$ and $\mu_{\bar{a}}$ respectively.

    \item Any \hyperref[def:invariant measure]{invariant measure} stochastically dominates $\mu_{\ubar{a}}$ and is stochastically dominated by $\mu_{\bar{a}}$.
    
    \item The measures $\mu_{\ubar{a}}$ and $\mu_{\bar{a}}$ coincide if and only if IPS$(P)$ is \hyperref[def:ergodicity]{ergodic}.
 \end{enumerate}
\end{theo}

In particular, this result implies that for \hyperref[def:monotonicity]{monotone} IPS uniqueness of the \hyperref[def:invariant measure]{invariant measure} is equivalent to \hyperref[def:ergodicity]{ergodicity}.

\begin{proof}[Proof of Theorem~\ref{pro:monoresult}]
Argument for~(1): Without loss of generality we set $a=\bar{a}$.
Let us pick any $t\geq s\geq 0$. Let $\zeta_0 = a$ everywhere and $\xi_0$ be distributed according to $\delta_a P_{t-s}$. We have $\zeta_0 \geq \xi_0$ so by monotonicity of $P$ and Lemma~\ref{pro:domination lemma} we can construct their respective trajectories under $P$ in such a way that $\zeta_u \geq \xi_u$ for all $u\geq 0$. In particular $\zeta_s \geq \xi_s$. This in turn implies
$$\delta_a P_s\bb{A} = \pP\bb{\zeta_s \in A} \geq \pP\bb{\xi_s \in A} = \delta_{a} P_t\bb{A}.$$\smallskip

Argument for~(2): The inclusion-exclusion principle implies that any measure $\mu$ on $\mathcal{A}^G$ is uniquely determined by its values of
$$\mu\bb{\sigma \geq C \text{ on } K}$$
for all possible cylinders $C$ and their supports $K$. This becomes clear as the measure~$\mu(\hat C)$ of any cylinder~$\hat C$ with finite support is a finite linear combination of terms of this form.
The sequence $$\bb{\delta_a P_t\bb{\sigma \geq C\text{ on }K}}_{t\geq 0}$$ is monotone by $(1)$ because these are increasing events. It is also bounded so it has a limit. Thus the limit of $\delta_a P_t(\hat C)$ exists for any $\hat C$, which is exactly convergence in distribution.\medskip

Argument for~(3): Let $\mu$ be an invariant measure of IPS$(P)$. By symmetry, it is sufficient to prove that $\mu$ stochastically dominates $\mu_{\ubar{a}}$. Let $\zeta_0$ be distributed according to $\mu$, $\xi_0=\ubar{a}$ everywhere, and the processes $\bb{\zeta_t}_{t\geq 0}$, $\bb{\xi}_{t\geq 0}$ be coupled such that $\zeta_t \geq \xi_t$ almost surely. Then for any increasing event $A$, we have
$$\mu(A) = \mu P_t(A) = \pP\bb{\zeta_t\in A} \geq \pP\bb{\xi_t\in A}=\delta_{\ubar{a}}P_t(A)\xrightarrow{t\rightarrow\infty}\mu_{\ubar{a}}(A).$$

Argument for~(4): The Lemma~\ref{pro:domination lemma} implies that for any measure $\mu$ and all~$t \geq 0$
$$\delta_{\bar{a}} P_t \bb{\sigma \geq C \text{ on }K} \geq \mu P_t\bb{\sigma \geq C \text{ on }K}\geq \delta_{\ubar{a}} P_t\bb{\sigma \geq C \text{ on }K}.$$
As a consequence 
$$\mu_{\bar{a}} \bb{\sigma \geq C \text{ on }K} \geq \lim_{t\rightarrow\infty}\mu P_t\bb{\sigma \geq C \text{ on }K}\geq \mu_{\ubar{a}}\bb{\sigma \geq C \text{ on }K}.$$
If $\mu_{\ubar{a}} = \mu_{\bar{a}}$ then the limits of $\mu P_t\bb{C}$ exist and agree with them. This implies that the limiting measure of $\mu P_t$ is $\mu_a = \mu_b$. 
\end{proof}

 Let us introduce the concept of weak domination and weak monotonicity.\\

\begin{defi}[Weak domination]\label{def:weak domination}
    For two transition matrices $P,Q$ over the same lattice $\Lambda$ and a linearly ordered alphabet $\bb{\mathcal{A},\geq}$ we will say that $P$ \emph{weakly dominates} $Q$ if for all $a\in\mathcal{A}$ it holds that
    $$ P_j\bb{a_+\,|\,\zeta}\geq Q_j\bb{a_+\,|\,\xi}, \text{ whenever }\left\{\begin{array}{l}
\zeta \geq \xi, \text{ and }\\
\zeta(j), \xi(j)\geq a \text{ or }\zeta(j),\xi(j)<a.
\end{array}\right.$$
\end{defi}
This is a strictly weaker condition than (strong) \hyperref[domination]{domination} -  if $\zeta(j)>\xi(j)$ then for every $a\in\mathcal{A}$ such that $\zeta(j)\geq a>\xi(j)$ the inequality $P_j\bb{a_+\,|\,\zeta}\geq Q_j\bb{a_+\,|\,\xi}$ is no longer required.\\

\begin{defi}[Weak monotonicity]\label{def:weak monotonicity}
     A transition matrix is \emph{weakly monotone} if 
     it \hyperref[def:weak domination]{dominates itself weakly}.\\
\end{defi}

\hyperref[def:weak monotonicity]{Weak monotonicity} is of course a weaker condition than \hyperref[def:monotonicity]{monotonicity}. For example in the case of nearest neighbor IPS on $\Z$ with an alphabet of size two \hyperref[def:weak monotonicity]{weak monotonicity} requires $8$ independent inequalities to be satisfied, while \hyperref[def:monotonicity]{monotonicity} requires $12$.  In the next lemma, we verify that \hyperref[def:weak monotonicity]{weak monotonicity} is the correct extension of monotonicity.\\

\begin{lem}\label{pro:weak lemma}
    If IPS$(P)$ is \hyperref[def:weak monotonicity]{weakly monotone} then for $\lambda\leq\frac{1}{2}$ the $\text{IPS}\bb{\lambda P+\bb{1-\lambda}I}$ is \hyperref[def:monotonicity]{monotone}.
\end{lem}
\begin{proof}[Proof of Lemma~\ref{pro:weak lemma}]
    
For any $a\in\mathcal{A}$
$$\bb{\lambda P +\bb{1-\lambda}I}_j\bb{a_+\;|\;\zeta} =  \lambda P_j\bb{a_+\;|\;\xi} + \bb{1-\lambda}\ind\bb{\zeta(j)\geq a}.$$
For any $\zeta \geq \xi$ there are two cases to consider:
\begin{itemize}
    \item If $\zeta(j)\geq a >\xi(j)$ then for $\lambda\leq\frac{1}{2}$ the condition 
    $${\lambda P_j \bb{a_+\;|\;\zeta}+\bb{1-\lambda}I_j}\bb{a_+\;|\;\zeta} \geq \lambda P_j\bb{a_+\;|\;\xi} +\bb{1-\lambda}I_j\bb{a_+\;|\;\xi}$$
    is satisfied regardless of $P_j$.
    \item If on the other hand $\zeta(j)$ and $\xi(j)$ are both greater or both less than $a$ then the condition
    $$\bb{\lambda P +\bb{1-\lambda}I}_j\bb{a_+\;|\;\zeta} \geq\bb{\lambda P +\bb{1-\lambda}I}_j\bb{a_+\;|\;\xi}$$
    is equivalent to
    $$P_j\bb{a_+\;|\;\zeta} \geq P_j\bb{a_+\;|\;\xi}.$$
\end{itemize}
Thus if IPS$(P)$ is \hyperref[def:weak monotonicity]{weakly monotone} then for $\lambda \leq \frac{1}{2}$ the $\text{IPS}(\lambda P +\bb{1-\lambda}I)$ is monotone.
\end{proof}

We now apply the \hyperref[pro:time-scaling]{Time-Scaling Lemma} to extend the results of Theorem~\ref{pro:monoresult} to weakly monotone IPS. The main idea is that all the properties proven in Theorem~\ref{pro:monoresult} are invariant under scaling time. Therefore one can use 
the \hyperref[pro:time-scaling]{Time-Scaling Lemma} to prove that monotonicity and weak monotonicity are equivalent when studying \hyperref[def:ergodicity]{ergodic} properties. \\

\begin{theo}[Extension of results for monotone IPS]\label{pro:monoextend}
    \hyperref[def:weak monotonicity]{Weak monotonicity} is sufficient for the results of Theorem~\ref{pro:monoresult}.
\end{theo}
\begin{proof}[Proof of Theorem~\ref{pro:monoextend}]
    By the \hyperref[pro:time-scaling]{Time-Scaling Lemma}, it suffices to show that there exists a $\lambda>0$ such that $\lambda P +\bb{1-\lambda}I$ is \hyperref[def:monotonicity]{monotone}. By Lemma $\ref{pro:weak lemma}$, any $\lambda$ between $0$ and $\frac{1}{2}$ works.
\end{proof}

Before turning to Gray's criterion of ergodicity of monotone IPS, let us introduce the following notion. \\

\begin{defi}[Positive rates]\label{positive rates}
     A transition matrix $P$ has \emph{positive rates} if it differs from the identity on every entry.\\
\end{defi}

The name comes from the original construction of interacting particle systems with alphabet $\cb{0,1}$, where instead of a transition matrix determining updates each site had a clock with an exponential rate that varied depending on the site's neighborhood. Whenever this clock rang the site flipped its state between $\cb{0,1}$. Positive rates, in this case, mean that flipping the state is possible in any neighborhood.\\

\begin{theo}[Larry Gray, Theorem $1$, page 397, in~\cite{Gray:82}]\label{Gray} 
Every one-dimensional, periodic, weakly monotone IPS with nearest-neighbor interactions, \hyperref[positive rates]{positive rates} and alphabet of size two is \hyperref[def:ergodicity]{ergodic}.\\
\end{theo}

We observe that the \hyperref[positive rates]{positive rates} property is preserved by convex combinations with identity. Therefore in the proof of Theorem~\ref{Gray} one could assume w.l.o.g.~that the underlying IPS is monotone and not just weakly monotone.

\bigskip
\subsection{Ergodicity of additive and cancellative IPS}\label{sec:ergo_additive_cancellative}

In this section, we extend two well-known theorems on exponential ergodicity using the Time-Scaling Lemma. The two theorems are taken from Griffeath's lecture notes on "Additive and Cancellative Interacting Particle Systems"\cite{Griffeath:78}.\\

There are many ways to construct a \hyperref[def:general construction]{trajectory}  of an IPS and the construction in Griffeath's work differs greatly from our own. As a result, it's not immediately obvious what a transition matrix of an additive or cancellative IPS looks like. For the convenience of the reader, we directly give the correct definitions for our setting without a derivation. For details of the derivation we refer to Section~\ref{sec:char_additive_IPS} and~\ref{sec:char_cancellative_IPS} in the Appendix.\\

\begin{defi}[Additive IPS]\label{def:additive_IPS}
    An IPS is called \emph{additive}, if its transition matrix $P$ can be expressed as a convex combination of the $\mathbf{1}$ matrix and any matrices of the form
    \begin{align*}
    P_{S}\bb{1\,|\,\sigma} = 
    \begin{cases}
        1, & \mbox{if }S \cap \sigma^{-1}\bb{1}\neq \vn, \\
         0, &  \mbox{if }S\cap \sigma^{-1}\bb{1}= \vn,
    \end{cases}
\end{align*}
for any $S\subseteq \mathcal{N}$. \\
\end{defi}

Whenever there is an update at site $j$, an additive IPS either flips the state to one, or selects a subset~$S \subset \mathcal{N}$ of its interaction neighborhood~$\mathcal{N}$ via the convex combination. Then, the present state becomes~$1$, if there is a~$1$ present in the set~$j+S$.\\

\begin{theo}[Ergodicity of additive IPS, cf.~Corollary 2.5, page 20 in~\cite{Griffeath:78}]\label{pro:additive} 
    We consider an additive IPS that is homogeneous over $\Z^d$ with positive rates and alphabet $\cb{0,1}$. Then this IPS is exponentially \hyperref[def:ergodicity]{ergodic}. The rate of \hyperref[def:ergodicity]{ergodicity} is at least as large as the coefficient of the $\mathbf{1}$ matrix in the above convex combination.\\
\end{theo}  

Let us consider now cancellative IPS.\\

\begin{defi}[Cancellative IPS]
    An IPS is called \emph{cancellative}, if its transition matrix $P$ can be expressed as a convex combination of matrices of the form
    \begin{align*}
         &P_{0,S}\bb{1\,|\,\sigma} = \abs{S\cap \sigma^{-1}\bb{1}}\bb{\text{mod }2},\\
         &P_{1,S}\bb{1\,|\,\sigma} = 1-\abs{S\cap \sigma^{-1}\bb{1}}\bb{\text{mod }2}.
     \end{align*}  
\end{defi}

Cancelative IPS work similar to additive IPS. Instead of just choosing a subset~$S \subset \mathcal{N}$, one also chooses a type of update. The first type makes the present state a~$1$ if there is an odd number of $1$s present in~$j+S$. The second, if there is an even number of~$1$s.\\

\begin{theo}[Ergodicity of cancellative IPS, cf.~Corollary 2.3, page 72, in~\cite{Griffeath:78}]\label{pro:cancellative} 
     Let IPS$(P)$ be homogeneous over $\Z^d$ with positive rates and alphabet $\cb{0,1}$. If $P$ is a cancellative then IPS$(P)$ is exponentially \hyperref[def:ergodicity]{ergodic}. The rate of \hyperref[def:ergodicity]{ergodicity} is at least as large as twice the coefficient of the $\mathbf{1}$ matrix in the above convex combination.\\
\end{theo}

% The details of the translation from Griffeath's construction of additive/cancellative IPS can be found in Sections~\ref{sec:char_additive_IPS}~and~~\ref{sec:char_cancellative_IPS} in the appendix. 

We want to note that all of the matrices $P_S$ are deterministic and in the additive case \hyperref[def:monotonicity]{monotone}. The way we extend these theorems remains the same - we characterize the IPS$(P)$ for which there exists another of the form $\text{IPS}(\lambda P + \bb{1-\lambda}I)$ which satisfies the conditions in either theorem.\\
\begin{theo}[Extension of Griffeath's theorems]\label{griffeath extended}
   Theorems~\ref{pro:additive} and~\ref{pro:cancellative} also hold when the coefficient of $P_{\cb{0}}$ or~$P_{0,\cb{0}}$ (i.e.~the identity) is allowed to be negative (the coefficients must still sum to $1$).
\end{theo}

\begin{proof}[Proof of Theorem~\ref{griffeath extended}]
We only show how to extend Theorem~\ref{pro:additive}. The extension of Theorem~\ref{pro:cancellative} is similar. Let $P$ be a transition matrix with positive rates, expressible as
$$P =\alpha\mathbf{1} + \sum_{S\subseteq \mathcal{N}}\lambda_S P_S$$
where $\alpha + \sum \lambda_S =1$ and $\lambda_S\geq 0$, with the possible exception of $\lambda_{\cb{0}}$.
Consider the transition matrix 
$$Q =  \lambda P + \bb{1-\lambda}I = \lambda\cdot \alpha \mathbf{1}+ \sum_{S\subseteq\mathcal{N}}\lambda\cdot\lambda_S P_S + \bb{1-\lambda}P_{\cb{0}}.$$
For $\lambda$ sufficiently close to $0$ the coefficient of $P_{\cb{0}}=I$ becomes positive, which means that by Theorem~\ref{pro:additive} (or~\ref{pro:cancellative}) the $\text{IPS}(Q)$ is exponentially \hyperref[def:ergodicity]{ergodic} with rate $\lambda\cdot\alpha$. We conclude with the Time-Scaling Lemma which implies that IPS$(P)$ is exponentially \hyperref[def:ergodicity]{ergodic} with rate $\alpha$.
\end{proof}

\section{A new criterion for decay of covariances}\label{sec:covdec}

In this section, we develop a method to show exponential decay of correlations of an IPS. The main result of this section, namely Theorem~\ref{pro:main-result} from below, is an extension of a result by Leontovitch and Vaserstein \cite{VasLeo:70} for PCAs (see Theorem~\ref{pro:leontovitch} from below). For a reference in English see Chapter~$4$ and Chapter~$7$ in~\cite{Toom:90}. Our method allows to extend the criterion to arbitrary alphabets and demands fewer assumptions on the underlying lattice structure. For example, the interaction neighborhoods~$\mathcal{N}_j$ do not have to be uniformly bounded. As our method allows to study IPS in continuous time, we also apply the time-scaling trick which relaxes the central assumption on contractivity. \\

The main idea is as follows. Consider the covariance~$\cov{f(\zeta_0)}{g(\zeta_t)}$ between two functions on the configuration space, one at time $0$ and one at time $t$. Increasing $t$ has the effect of adding more updates to $\zeta_t$. If each of those updates acts like a contraction then an exponential decay of correlations can be deduced. Compared to discrete time IPS, this argument faces additional technical challenges when applied to continuous time. The reason is that the times of updates are random and not well-ordered.\\

\bigskip
\subsection{Auxillary definitions}

Before stating our theorem let us define a few auxiliary objects. We will start by defining a convenient semi-norm on the space of local functions for a given alphabet and lattice.\\

\begin{defi}[Representational semi-norm]\label{def:representational_seminorm}
    For any family $\mathcal{F}\subseteq \mathbf{F}$ such that $\text{Span}\bb{\mathcal{F}}= \mathbf{F}$ we define a \emph{representational semi-norm} on $F$ as follows:
    $$\norm{f}_\mathcal{F}:= \inf\cb{\sum_{1\leq k\leq n}\abs{a_k}\;:\;\exists\bb{g_1,\dots,g_n\in\mathcal{F}\setminus\cb{\text{const.}}}\Cb{\sum_{1\leq k\leq n} a_k\cdot g_k = f - \text{const.}}}.$$
\end{defi}

\begin{defi}[Product basis] \label{product basis}
    For a linearly ordered, finite alphabet $\bb{\mathcal{A}, \geq}$ let $\ubar{a}$ be its minimal element. Let us consider two functions $x: \mathcal{A}\setminus\{\ubar{a}\}\rightarrow \R$ and $y: \mathcal{A}\setminus\{\ubar{a}\} \rightarrow\R$ such that $x(a)\neq y(a)$ for all $a\in\mathcal{A}\setminus\{\ubar{a}\}$. We consider the family of functions
    $$\mathcal{F}_{\Pi}(x,y)=\cb{\bigchi_{K, a_K} \in \mathbf{F}\;:\;K\subseteq \Lambda,\; \abs{K}<\infty,\; a_K\in \bb{\mathcal{A}\setminus\{\ubar{a}\}}^K},$$
    where $\bigchi_{K, a_K}$ is defined as
    \begin{align*}
        &\bigchi_{\vn} := 1,\\
        &\bigchi_{K,a_K}:=\prod_{j\in K} \bigchi_{\cb{j},a_K(j)},\\
        &\bigchi_{\cb{j},a}\bb{\zeta} := \left\{\begin{array}{lc}
        x(a), &  \zeta(j) \geq a\\
        y(a), &  \zeta(j) < a.
    \end{array}\right.
    \end{align*}
    The next statement justifies to call $\mathcal{F}_{\Pi}$ a \hyperref[product basis]{product basis} of $\mathbf{F}$.\\
    
 \end{defi}

    \begin{lem}\label{pro:product_basis_is_a_basis}
        The family $\mathcal{F}_{\Pi}$ is a basis of the linear space $\mathbf{F}$.
    \end{lem}
    \begin{proof}[Proof of Lemma~\ref{pro:product_basis_is_a_basis}]
         To convince us that $\text{Span}\bb{\mathcal{F}_{\Pi}}=\mathbf{F}$, we observe that one can easily construct the indicator function of arbitrary cylinders by multiplying terms of the form
    $$\frac{\bigchi_{\cb{j},a}-y(a)}{x(a)-y(a)},\;\text{ or  }\;\frac{\bigchi_{\cb{j},a}-x(a)}{y(a)-x(a)},$$
    and then by expanding the result. These indicator functions span $\mathbf{F}$ in the following way:
    $$f(\zeta) = \sum_{\sigma\in \mathcal{A}^{\text{Dom}\bb{f}}}f\bb{\sigma} \ind\bb{\zeta\bb{\text{Dom}(f)} = \sigma}.$$
    It is left to show that the representations in $\mathcal{F}_{\Pi}$ are unique. For this it suffices to show that for any finite $K\subset \Lambda$ the functions $$\cb{\bigchi_{L,a_L} \;:\; L\subseteq K, a_L\in\bb{\mathcal{A}\setminus\{\ubar{a}\}}^L}$$
    form a basis of $\R^{\mathcal{A}^K}$. Indeed, if that were not the case, we would find a counterexample to their linear independence within some finite $K$. The dimension of $\R^{\mathcal{A}^K}$ is $\abs{\mathcal{A}}^{\abs{K}}$ and there are exactly as many functions in this family.
    Since they span $\R^{\mathcal{A}^K}$ they must be linearly independent.
    \end{proof}

\begin{nota}
 For convenience, we will write $\norm{\cdot}_{\Pi}$ rather than $\norm{\cdot}_{\mathcal{F}_{\Pi}}$.
 For the special case when the alphabet is of size two, we will identify the functions $x$, $y$ with their unique values and write
    $$\bigchi_{\cb{j}}:=\bigchi_{\cb{j},\bar{a}} \quad \mbox{and} \quad \bigchi_K:=\prod_{j\in K}\bigchi_{\cb{j}},$$
    as these indexes are not needed.\\
\end{nota}

The next remark shows how the semi-norm~$\| \cdot \|_{\mathcal{F}}$ can be used to characterize ergodicity.\\

\begin{rem}\label{rem:ergodicity_semi_norm}
Let us consider a basis~$\mathcal{F}$ that contains the constant function~$1$. Then, a function~$f$ can be uniquely written as linear combination of the basis functions i.e.
      \begin{align*}
          f = a_{\bold{1}} + \sum_{h \in \mathcal{F} \setminus \left\{ \bold{1} \right\}} a_h \  h.
       \end{align*}
       The semi-norm~$\|f \|_{\mathcal{F}}$ is given by $|f| = \sum_{h \in \mathcal{F} \setminus \left\{ \bold{1} \right\}} |a_h | $.  We observe that ergodicity means that $\CE{f(\zeta_t)}{\zeta_0}$ converges to a constant as~$t \to \infty$ for any initial condition~$\zeta_0$ and test function~$f$. This is implied by~$\| \CE{f(\zeta_t)}{\zeta_0}  \|_{\mathcal{F}} \to 0$. \\
\end{rem}

By the last remark, it becomes evident that in order to show ergodicity of the underlying IPS, one needs to understand the evolution of~$\CE{f(\zeta_t)}{\zeta_0}$ for arbitrary test-function~$f$. By the construction of the IPS, we observe that~$\zeta_t$ only changes if certain clocks ring and trigger an update. For this reason, we introduce the concept of update operators.  \\

\begin{defi}[Update operators]\label{def:update_operators}
        For an $\text{IPS}\bb{P,\mathcal{C}}$ on $\Lambda$ let us define the linear operators  $E_j, E:\mathbf{F}\bb{\mathcal{A},\Lambda}\rightarrow \mathbf{F}\bb{\mathcal{A},\Lambda}$ for $j\in\Lambda$ by
    
        $$E_j(f)(\sigma) := \CE{f\bb{\zeta_t}}{\zeta_{t^-}=\sigma\textit{ and }\bb{t,j}\in\mathbf{Update}} $$
        and
        \begin{align}\label{equ:update_operator_domain_f}
            E(f) := \frac{1}{\abs{\text{Dom}(f)}}\sum_{j\in \text{Dom}(f)}E_j(f).            
        \end{align}
\end{defi}

    \begin{rem}
        The update operators~$E_j$ specify the mean effect of a one-step update onto a test-function~$f$, assuming that an update happened at site~$j$. If the clock is $\text{Exp}(1)$ then a.s.~at most one site can be updated at the same time. Also, each site is equally likely to be the next one updated.\\
    \end{rem}

\begin{defi}[Update coefficients]
        We consider a product basis~ $\mathcal{F}_{\Pi}$ and  represent the update operator~$E_j\bb{\bigchi_{\cb{j},a}}$ in this basis, i.e. 
        \begin{align}\label{equ:mean_update_operator}
        E_j\bb{\bigchi_{\cb{j},a}}=\desum{S\subseteq \mathcal{N}_j}{a_S\in\bb{\mathcal{A}\setminus\{\ubar{a}\}}^S}C_{S,\;a_S}^{\bb{j,a}}\cdot \bigchi_{S,\;a_S}.          
        \end{align}
        The coefficients  $$\cb{C_{S,\;a_S}^{\bb{j,a}}\;:\;j\in\Lambda,a\in\mathcal{A}\setminus\{\ubar{a}\},\;S\subseteq\mathcal{N}_j,a_S\in\bb{\mathcal{A}\setminus\{\ubar{a}\}}^S}$$ 
        are called representational coefficients of $E_j\bb{\bigchi_{\cb{j},a}}$.\\
\end{defi}

\begin{nota}
    As before, for the special case when the alphabet is of size two, we will omit the indexes and write
    $$C_{S}^{\bb{j}}:=C_{S,\bar{a}_S}^{\bb{j,\bar{a}}}.$$
    If the IPS is homogeneous and the underlying lattice has neighborhoods of the form $\mathcal{N}_j = j + \mathcal{N}$ then we will also omit the position index

    $$C_{S}:=C_{S}^{\bb{j}}.$$
\mbox{}\\
\end{nota}

Let us consider the situation of an IPS with exponential clocks~$\text{Exp}(1)$. Then the operator $\CE{f(\zeta_t)}{\zeta_0}$  can be expressed via the update operators in the following way:
\begin{align}\label{equ:ergodicity_via_mean_observable_step_1}
    \CE{f(\zeta_t)}{\zeta_0} = \sum_{n \geq 0} \mathbb{P}(\mbox{there are $n$ \emph{relevant} updates until $t$}) \ \underbrace{E \circ \cdots \circ E}_{\mbox{$n$-times}}(f).
\end{align}
 Using~\eqref{equ:update_operator_domain_f} and~\eqref{equ:mean_update_operator} one can express the right hand side of~\eqref{equ:ergodicity_via_mean_observable_step_1} in terms of the update coefficients~$C_{S, a_S}^{(j,a)}$. Then, one would expect that a good control on~$|C_{S, a_S}^{(j,a)}|$ would yield ergodicity (cf.~Remark~\ref{rem:ergodicity_semi_norm}). Theorem~\ref{pro:main-result} below makes this heuristic argument rigorous.

\bigskip
\subsection{Decay of correlations by recursion}

In this section we will show how decay of correlations of an IPS can be deduced via a simple recursion argument. We will begin by stating the original theorem by Leontovitch and Vaserstein~\cite{VasLeo:70}. \\

\begin{theo}[Leontovitch \& Vaserstein] \label{pro:leontovitch} Let $R:=\max_{j\in\Lambda} \abs{\mathcal{N}_j}$ be finite and for different real numbers $x,y$ denote
$$\gamma := \max_{1\leq m \leq R}\bb{1,\max\bb{\abs{x},\abs{y}}^{\frac{R}{m}-1},\bb{\abs{\frac{x^m-y^m}{x-y}}+\abs{\frac{x^{m}y-xy^{m}}{x-y}}}^{\frac{R}{m}}}.$$
If for a $PCA$ with alphabet of size two over lattice $\Lambda$ and the product basis with values $x$, $y$ it holds that
$$\sup_{j\in\Lambda}\sum_{S\subseteq \mathcal{N}_j}\abs{C_{S}^{\bb{j}}} < \frac{1}{\gamma},$$
then this $PCA$ is uniformly, exponentially \hyperref[def:ergodicity]{ergodic}.\\
\end{theo}

The main result of this section is the following theorem. \\

\begin{theo}\label{pro:main-result}
    Assume that for an IPS$(P)$ with a finite alphabet $\mathcal{A}$ there exists a \hyperref[product basis]{product basis} with values $x,y:\mathcal{A}\setminus\{\ubar{a}\}\rightarrow\R$ such that 
    \begin{enumerate}
    \item \label{equ:main_theo_condition_basis_representation_1}The values $x$, $y$ are bounded by $\pm 1$.\\
    
    \item \label{equ:main_theo_condition_basis_representation_2} For all $a\in\mathcal{A}\setminus\{\ubar{a}\}$ it holds $x(a)y(a)\leq 0$.\\

    \item  \label{equ:main_theo_condition_basis} For all $a > b > \ubar{a}$ the inequality $\abs{x(b)}+\abs{y(a)} + \abs{x(b)y(a)}\leq 1$ holds.\\

    \item \label{equ:main_theo_condition_constants_bounded} The constants $\bb{C_{\cb{j},a}^{\cb{j,a}}}_{j\in\Lambda, a\in \mathcal{A}\setminus\{\ubar{a}\}}$ are uniformly bounded.\\
    
    \item \label{equ:main_theo_condition_constractivity} The following inequality is satisfied:
    
    \begin{align}\label{eq_contractivity_gen_leontovitch}      
    \alpha :=\sup_{j\in\Lambda, a\in\mathcal{A}\setminus\{\ubar{a}\}}\bb{C_{\cb{j},a}^{\bb{j,a}}+\desum{S\subseteq \mathcal{N}_j, a_S\in\bb{\mathcal{A}\setminus\{\ubar{a}\}}^S}{\bb{S,\;a_S}\neq \bb{\cb{j},a}}\abs{C_{S,\;a_S}^{\bb{j,a}}} }<1.
    \end{align} 

    \end{enumerate}
    Then IPS$(P)$ has uniform, exponential decay of correlations. More precisely, for any $f,g\in \mathbf{F}$ and $\zeta\in \text{IPS}(P)$
    
    $$\abs{\cov{f(\zeta_0)}{g(\zeta_t)}}\leq 2\norm{f}_{\infty}\norm{g}_{\Pi}e^{-\bb{1-\alpha}t}.$$\\

\end{theo}

We observe that Theorem~\ref{pro:main-result} extends Theorem~\ref{pro:leontovitch} in several ways. First, it extends the statement from PCA to IPS. Then, it does not require the interaction neighborhoods $\mathcal{N}_j$ to be uniformly bounded and alphabets can be of arbitrary finite length. Finally, the central inequality is relaxed. \\

\begin{rem}
    If the alphabet is of size two then Condition~\ref{equ:main_theo_condition_basis} of Theorem~\ref{pro:main-result} is null. If the IPS is homogeneous then Condition~\ref{equ:main_theo_condition_constants_bounded} is automatically satisfied. In the special case of homogeneous IPS with the alphabet size two Condition~\ref{equ:main_theo_condition_constractivity} becomes

    $$\alpha :=C_{\cb{0}}+\desum{S\subseteq \mathcal{N}}{S\neq \cb{0}}\abs{C_{S}} <1.$$\\
\end{rem}

The strategy for the proof of Theorem~\ref{pro:main-result} is as follows: First, we will prove the result assuming the stronger inequality 
\begin{align}\label{equ:main_theo_condition_contractivity_strong}
  \beta := \sup_{j\in\Lambda,a\in\mathcal{A}\setminus\{\ubar{a}\}}\desum{S\subseteq \mathcal{N}_j}{a_S\in\bb{\mathcal{A}\setminus\{\ubar{a}\}}^S}\abs{C_{S,\;a_S}^{\bb{j,a}}} <1.    
\end{align}
This condition is stronger than the inequality~\eqref{eq_contractivity_gen_leontovitch} as the terms~$C_{\cb{j},\;a_{\cb{j}}}^{\bb{j,a}}$ appear with absolute values. Then in the second main step, we will weaken this condition using the \hyperref[pro:time-scaling]{Time-Scaling Lemma}, getting rid of those absolute values.\\

Let us turn to the first step. We introduce auxiliary stopping times to label the updates in a given subset of $\Lambda$. \\

\begin{defi}[Times of previous updates]\label{def:time_prev_update}
    Let $\Lambda$ be a lattice such that $\bb{\Lambda,\text{Exp}(1)}$ is \hyperref[def:causal]{causal}.
    For any finite $K\subset \Lambda$ and $t\geq 0$ let~$\pi_K(t)$ denote the stopping time of the last update that happened somewhere in $K$ up to time $t$, i.e.~
    $$\pi_K(t):=\sup\cb{s<t\;:\; \cb{s}\times K \cap \mathbf{Update}\neq \vn}\cup\cb{0}.$$
\end{defi}

The core of the argument is to show that an update in $K$ contracts the covariance $\cov{f(\zeta_0)}{\bigchi_{K,a_K}(\zeta_t)}$. The central ingredient in deducing the contraction is  Lemma~\ref{iteration bound} below. After deducing the contraction property, it is left to estimate the distribution of the number of relevant updates up to a specified time.\\

\begin{lem}\label{iteration bound}
    If the \hyperref[product basis]{product basis} $\mathcal{F}_{\Pi}$ satisfies the Conditions~\ref{equ:main_theo_condition_basis_representation_1},~\ref{equ:main_theo_condition_basis_representation_2} and~\ref{equ:main_theo_condition_basis} of Theorem~\ref{pro:main-result}), then

    $$\norm{E(\bigchi_{K,\;a_K})}_{\Pi} \leq \beta.$$
    \mbox{} \\
    
\end{lem}

\begin{proof}[Proof of Lemma~\ref{iteration bound}]
The argument is a straightforward calculation though it is easy to get lost in notation. Therefore, we first carry out the proof in the most simple scenario and then in full generality. This also helps to point out the main steps of the argument.  \medskip 

Let us consider the an IPS with alphabet~$\left\{ 0, 1 \right\}$, on a one-dimensional lattice with right nearest-neighbor interaction (for more details see Section~\ref{sec:one_d_ips}). For the linear space~$\mathbf{F}$ we consider the basis functions
\begin{align*}
    \bigchi_K : = \prod_{j \in K} \bigchi_{\left\{j \right\} }
\end{align*}
given by 
\begin{align*}
    \bigchi_{{i}} (\sigma) : = 
    \begin{cases}
        +1, & \mbox{if } \sigma(i)= +1 \\
        -1, & \mbox{if } \sigma(i) = 0.
    \end{cases}
\end{align*}
We observe that in this situation the constant~$\beta$ (cf.~\eqref{equ:main_theo_condition_contractivity_strong}) becomes
\begin{align*}
    | C_{\emptyset}| + |C_{\{ 0 \}}| + |  C_{\{ 1 \}}| + | C_{\{ 0, 1 \}}| = \beta.
\end{align*}
Additionally, the representation~\eqref{equ:mean_update_operator} of the mean-update operator~$E_j\left(\bigchi_K \right)$ for a given update at site~$j$ simplifies to
\begin{align}
    E_j\left(\bigchi_K \right) & = E_j\left( \prod_{i \in K }\bigchi_{\left\{ i\right\} } \right) \notag \\
    & =  \prod_{i \in K \setminus \left\{ j \right\} } \bigchi_{i} \ E_j\left( \bigchi_{\left\{ j\right\} } \right) \notag  \\
     & = \bigchi_{K \setminus \left\{ j \right\}} \ E_j\left( \bigchi_{\left\{ j\right\} } \right) \notag  \\
     & = \bigchi_{K \setminus \left\{ j \right\}  } \left( C_{\emptyset} +C_{\{ 0 \}} \bigchi_{\{j\}} +   C_{\{ 1 \}} \bigchi_{\{j+1\}} +  C_{\{ 0, 1 \}} \bigchi_{\{ j,j+1 \}} \right) \notag \\
     & =  C_\emptyset \bigchi_{K \setminus \left\{ j \right\}  } + C_{\{ 0 \}} \bigchi_{K} +
     \begin{cases}
         C_{\{1\}} \bigchi_{K \setminus \left\{ j , j+1 \right\}} +  C_{\{ 0, 1 \}} \bigchi_{K \setminus \left\{  j+1 \right\}}, & \mbox{if } j+1 \in K \\
         C_{\{1\}} \bigchi_{K \setminus \left\{ j \right\} \cup \left\{ j +1 \right\}} +  C_{\{ 0, 1 \}} \bigchi_{K \cup \left\{  j+1 \right\}},  & \mbox{if } j+1 \notin K.
     \end{cases} \label{equ:repre_E_j_simple_case}
\end{align}
We observe that for any set~$M$ it holds~$\|\bigchi_M \|_{\Pi} =1$. Indeed, this follows from the definition of~$\| \cdot\|_{\Pi}$ and the observation that~$\bigchi_M$ is an element of the basis. Therefore, the identity~\eqref{equ:repre_E_j_simple_case} combined with the triangle inequality yields the desired inequality 
\begin{align}
  \norm{E( \bigchi_K)}_{\Pi} & = \norm{\frac{1}{|K|} \sum_{j \in K}  E_j \left( \bigchi_K \right)}_{\Pi} \notag \\
  & \leq  \frac{1}{|K|} \sum_{j \in K}   \norm{ E_j \left( \bigchi_K \right)  }_{\Pi} \notag \\
  & \leq \sup_j \norm{ E_j\left(\bigchi_K \right) }_{\Pi} \notag \\
  & \leq  | C_{\emptyset}| + |C_{\{ 0 \}}| + |  C_{\{ 1 \}}| + | C_{\{ 0, 1 \}}| = \beta. \label{equ:main_result_proof_simple_case_last_step}
\end{align}

The argument for the general case is similar but deviates in the last step (see~\eqref{equ:main_result_proof_simple_case_last_step}). Again, we observe that
$$\norm{E(\bigchi_{K,\;a_K})}_{\Pi} = \norm{\frac{1}{\abs{K}}\sum_{j\in K} E_j(\bigchi_{K,\;a_K})}_{\Pi}\leq \sup_{j\in K}\norm{E_j(\bigchi_{K,\;a_K})}_{\Pi}.$$
This last term can be bounded as follows:
 
\begin{align*}
& \norm{E_j\bb{\bigchi_{K,\;a_K}}}_{\Pi} \\
& \quad = \norm{\bigchi_{K\setminus\cb{j},\;a_{K\setminus\cb{j}}} \cdot E_j\bb{\bigchi_{\cb{j},\;a_{\cb{j}}}}}_{\Pi} \\
& \quad = 
\norm{\desum{S\subseteq \mathcal{N}_j}{b_S\in\bb{\mathcal{A}\setminus\{\ubar{a}\}}^S}C_{S, a_S}^{\bb{j,a_{\cb{j}}}}\bigchi_{K\setminus\cb{j} \triangle S,\; \bb{a_{K\setminus\cb{j}\cup S}\cup b_{S\setminus\bb{K\setminus \cb{j}}}}} \prod_{i\in K\setminus\cb{j}\cap S}\bigchi_{\cb{i},a_{\cb{i}}} \bigchi_{\cb{i},b_{\cb{i}}}}_{\Pi}
\end{align*}
Now we apply the triangle's inequality, bounding the sum of $\abs{C_{S, a_S}^{\bb{j,a_{\cb{j}}}}}$ by $\beta$ and the norm of the product of $\bigchi$ functions by the worst possible case. We obtain an upper bound of
$$\beta\cdot \sup_{\hat K, a_{\hat K}, \bb{b_n}, \bb{c_n}}\norm{\bigchi_{\hat K, a_{\hat K}}\prod_{n} \bigchi_{\cb{i_n}, b_n}\cdot \bigchi_{\cb{i_n}, c_n}}_{\Pi},$$
where the product in the last term is over at most~$|K|$ many terms.\medskip

Compared to the simple case (see~\eqref{equ:main_result_proof_simple_case_last_step}), the function in the semi-norm is not automatically an element of the basis. Therefore it is not obvious if the supremum does not exceed one. In the next paragraph, we check that this supremum does not exceed one.\medskip

First, we notice that the representational seminorm has the nice property that for any functions $f$, $g\in\mathbf{F}$ with disjoint domains 
\begin{align}\label{equ:seminorm_submultiplicativity}
\norm{f\cdot g}_{\Pi} \leq \norm{f}_{\Pi}^*\cdot\norm{g}_{\Pi}^*,    
\end{align}
where the $\norm{\;\cdot\;}_{\Pi}^*$ is the same as $\norm{\;\cdot\;}_{\Pi},$ except the constant is not ignored (c.f.~Definition~\ref{def:representational_seminorm}). \smallskip

Indeed, if the representations of these functions in the basis are
$$f = \desum{K\subseteq\text{Dom}(f)}{a_K\in\bb{\mathcal{A}\setminus\{\ubar{a}\}}^K}\alpha_{K, a_K}\cdot\bigchi_{K, a_K}$$
and
$$g = \desum{L\subseteq\text{Dom}(g)}{b_L\in\bb{\mathcal{A}\setminus\{\ubar{a}\}}^L}\beta_{L,b_L}\cdot\bigchi_{L,b_L},$$
then their product is can be written as
$$f \cdot g = \sum \alpha_{K, a_K}\beta_{L, b_L}\cdot \bigchi_{K\cup L, a_K\cup b_L}.$$
The crucial property used here is that the product basis is closed under the multiplication of $\bigchi$ functions with disjoint domains, from which~\eqref{equ:seminorm_submultiplicativity} follows. \medskip

Using the identity~\eqref{equ:seminorm_submultiplicativity} and obseving that the product only has at most~$|K|$ many terms, one obtains the inequality

$$\norm{\bigchi_{K, a_K}\prod_{n} \bigchi_{\cb{i_n}, b_n}\cdot \bigchi_{\cb{i_n}, c_n}}_{\Pi} \leq \sup_{j\in\Lambda}\sup_{a,b\in\mathcal{A}\setminus\{\ubar{a}\}} \left( \norm{\bigchi_{\cb{j},a}\cdot \bigchi_{\cb{j},b}}_{\Pi}^*\right)^{|K|}.$$

It is left to show that
\begin{align}
    \norm{\bigchi_{\cb{j},a}\cdot \bigchi_{\cb{j},b}}_{\Pi}^* \leq 1.
\end{align}
Indeed, this can be achieved by calculating the term explicitly. Without loss of generality assume that $a\geq b > \ubar{a}$. Then
$$\bigchi_{\cb{j},\;a}\cdot \bigchi_{\cb{j},\;b} = x(b)\cdot\bigchi_{\cb{j},\;a} + y(a)\cdot\bigchi_{\cb{j},\;b} -x(b)y(a).$$
There are two cases. If $a=b$ then
$$\norm{\bigchi_{\cb{j},\;a}\cdot \bigchi_{\cb{j},\;b}}_{\Pi}^* = \abs{x(a)+y(a)}+\abs{x(a)y(a)}\leq 1,$$
whenever $x(a)$ and $y(a)$ are bounded by $\pm 1$ and have different signs. If $a>b$ then the inequality required is
$$\abs{x(b)}+\abs{y(a)}+\abs{x(b)y(a)}\leq 1,$$
which follows from our assumption~\ref{equ:main_theo_condition_basis} and completes the argument.
\end{proof}
\mbox{}\\

\begin{proof}[Proof of the Theorem~\ref{pro:main-result}]

As outlined above, the proof is structured in two parts. In the first part, we make the stronger assumption $\beta <1$ instead of ~$\alpha <1$ (see~\eqref{equ:main_theo_condition_contractivity_strong}). We will show that 
\begin{align}\label{proof_main_result_weak}
    | \Cov(f(\zeta_0), g(\zeta_t)) | \leq 2 \ \|f \|_{\infty} \  \|g \|_{\Pi} \ e^{- (1- \beta)t}.
\end{align}
In the second step, using the time-scaling trick we will weaken the assumption to~$\alpha <1$ and improve the inequality~\eqref{proof_main_result_weak} to the desired estimate
\begin{align}\label{proof_main_result_strong}
    | \Cov(f(\zeta_0), g(\zeta_t)) | \leq 2 \ \|f \|_{\infty} \  \|g \|_{\Pi} \ e^{- (1- \alpha)t}.
\end{align}

Let us turn to the verification of~\eqref{proof_main_result_weak}. In the first step, writing the function~$g \in \mathbf{F}$ via the product basis~$\mathcal{F}_{\Pi}$ yields
\begin{align*}
    g = \sum_{K, a_K} C_{K, a_k} \bigchi_{K,a_K},
\end{align*}
where the sum is over finitely many sets~$K$ and tuples~$a_K$. Let us note that all sets~$K$ are non-empty and finite. We therefore get the estimate
\begin{align*}
    \abs{\Cov(f(\zeta_0), g(\zeta_t))} & = \abs{\sum_{K, a_K} C_{K, a_K} \Cov(f(\zeta_0), \bigchi_{K, a_K} (\zeta_t) )} \\
    & \leq \norm{g}_{\Pi} \ \sup_{K, a_K} \abs{\Cov(f(\zeta_0), \bigchi_{K, a_K} (\zeta_t)) }.
\end{align*}

Hence, in order to deduce~\eqref{proof_main_result_weak} it suffices to show that
\begin{align}\label{proof_main_result_NTS}
   \sup_{K, a_K} |\Cov(f(\zeta_0), \bigchi_{K, a_K} (\zeta_t)) | \leq   2\norm{f}_{\infty} e^{-\bb{1-\beta}t}.
\end{align}

 We turn to the verification of~\eqref{proof_main_result_NTS} and consider a finite set~$K_0$ with tuplet~$a_{K_0}$. Let us consider the stopping time~$\pi_{K_0}(t)$ which denotes the most recent time, a site in~$K_0$ got updated (see Definition~\ref{def:time_prev_update}). We observe that
\begin{align*}
    & |\Cov(f(\zeta_0), \bigchi_{K_0, a_{K_0}} (\zeta_t)) | \\
    & \qquad \leq  |\Cov(f(\zeta_0),  \ind(\pi_{K_0}(t) = 0) \ \bigchi_{K_0, a_{K_0}} (\zeta_t)) | + |\Cov(f(\zeta_0),  \ind(\pi_{K_0}(t) > 0) \ \bigchi_{K_0, a_{K_0}} (\zeta_t)) |  \\
    & \qquad  \leq \norm{f}_{\infty} \ \mathbb{P} \left( \pi_{K_0}(t) = 0 \right) + |\Cov(f(\zeta_0),  \ind(\pi_{K_0}(t) > 0) \  \bigchi_{K_0, a_{K_0}} (\zeta_t)) |.
\end{align*}
% As arrivals are iid.~$\text{Exp}(1)$ distributed and $K_0$ is finite but non-empty, we get the estimate
% \begin{align*}
%     \mathbb{P} \left( \pi_{K_0, a_{K_0} }(t) = 0 \right) \leq  e^{-|K_0| t} \leq e^{-t}.
% \end{align*}
Let $\mathcal{F}_t$ denote the $\sigma$-algebra generated by all the events up to time $t$. Conditioning on $\mathcal{F}_0$ yields the identity
\begin{align*}
    \Cov(f(\zeta_0),  \ind(\pi_{K_0}(t) > 0) \bigchi_{K_0, a_{K_0}} (\zeta_t)) = \Cov \left( f(\zeta_0),  \mathbb{E} \left[ \ind(\pi_{K_0}(t) > 0) \bigchi_{K, a_K} (\zeta_t) \ | \ \mathcal{F}_0 \right] \right).
\end{align*}
A combination of the last three formulas we arrive at the estimate

\begin{align}
        & \abs{\Cov(f(\zeta_0), \bigchi_{K_0, a_{K_0}} (\zeta_t))} \notag \\ 
        & \qquad \leq 2 \ \norm{f}_{\infty} \ \mathbb{P} \left( \pi_{K_0}(t) = 0 \right) + \left|  \Cov \left( f(\zeta_0),  \mathbb{E} \left[ \ind(\pi_{K_0}(t) > 0) \bigchi_{K_0, a_{K_0}} (\zeta_t) \ | \ \mathcal{F}_0 \right] \right) \right|. \label{proof_main_result_recursion_first_step}
\end{align}
Let us consider the expectation under the covariance. We get that
\begin{align} 
    & \mathbb{E} \left[ \ind(\pi_{K_0}(t) > 0) \bigchi_{K_0, a_{K_0}} (\zeta_t) \ | \ \mathcal{F}_0 \right] \notag \\
    & \qquad = \CE{\CE{\ind(\pi_{K_0} (t)>0)\ \bigchi_{K_0,a_{K_0}}(\zeta_{\pi_{K_0}(t)})}{\mathcal{F}_{\pi_{K_0}(t)^-}}}{\mathcal{F}_0} \notag \\
    & \qquad = \CE{\ind(\pi_{K_0}(t)>0)\ \CE{\bigchi_{K_0,a_{K_0}}(\zeta_{\pi_{K_0}(t)})}{\mathcal{F}_{\pi_{K_0}(t)^-}}}{\mathcal{F}_0} \notag \\
    & \qquad = \CE{\ind(\pi_{K_0}(t)>0)\ E\bb{\bigchi_{K_0,a_{K_0}}}(\zeta_{\pi_{K_0}(t)^-})}{\mathcal{F}_0} \notag\\
    & \qquad = \sum_{K_1, a_{\tilde K_1}} \tilde{C}_{{K_1}, a_{ K_1}} \CE{\ind(\pi_{K_0}(t)>0)\ \bigchi_{ K_1, a_{K_1}} (\zeta_{\pi_{K_0}(t)^-})}{\mathcal{F}_0}   .
    \label{recursion2}
\end{align}
Here, the first equality follows from the law of total expectation. The second equality holds because $\ind(\pi_{K_0}(t)>0)$ is $\mathcal{F}_{\pi_{K_0}(t)^-}$-measurable. The third one holds by the definition of the operator~$E$ (see Definition~\ref{def:update_operators}). For the last identity, we expressed the function~$\E \bb{\bigchi_{K_0,a_{K_0}}}$ using the product basis~$\mathcal{F}_{\Pi}$.\\

Using the formula\eqref{recursion2} and Lemma~\ref{iteration bound} we estimate the second term on the right hand side of~\eqref{proof_main_result_recursion_first_step} as
\begin{align}
    & \left|  \Cov \left( f(\zeta_0),  \mathbb{E} \left[ \ind(\pi_{K_0}(t) > 0) \bigchi_{K_0, a_{K_0}} (\zeta_t) \ | \ \mathcal{F}_0 \right] \right) \right| \notag \\
    & \leq \sum_{K_1 \neq \emptyset , a_{K_1}}  \left| \tilde C_{K_1, a_{K_1}} \right| \ \left|\Cov \left( f(\zeta_0),  \CE{\ind(\pi_{K_0}(t)>0)\ \bigchi_{K_1, a_{K_1}} (\zeta_{\pi_{K_0}(t)^-})}{\mathcal{F}_0} \right)\right| \notag \\
    & \leq  \norm{E\bb{\bigchi_{K_0,a_{K_0}}}}_{\Pi} \sup_{K_1 \neq \emptyset, a_{K_1}} \left| \Cov \left( f(\zeta_0),  \mathbb{E} \left[ \ind(\pi_{K_0}(t) > 0) \bigchi_{K_1 \neq \emptyset , a_{K_1}} ( (\zeta_{\pi_{K_0}(t)^-})) \ | \ \mathcal{F}_0 \right] \right) \right| \notag \\
        & \leq \beta \ \sup_{K_1 \neq \emptyset, a_{K_1}} \left| \Cov \left( f(\zeta_0),  \mathbb{E} \left[ \ind(\pi_{K_0}(t) > 0) \bigchi_{K_1, a_{K_1}} ( (\zeta_{\pi_{K_0}(t)^-})) \ | \ \mathcal{F}_0 \right] \right) \right| . \label{proof_main_result_recursion_first_time_step_1}
\end{align}
Now, let us fix a non-empty but finite set~$K_1$ and introduce the ad-hoc notation~$\pi_0(t) =  \pi_{K_0}(t)$ and $\pi_1 (t) := \bb{\pi_{K_1} \circ \pi_{K_0}}(t) $. Because~$\pi_0 (t) \geq \pi_1 (t)$, the term on the right-hand side of~\eqref{proof_main_result_recursion_first_time_step_1} can be further estimated as
\begin{align}
  &   \left| \Cov \left( f(\zeta_0),  \mathbb{E} \left[ \ind(\pi_{K_0}(t) > 0) \bigchi_{K_1, a_{K_1}} ( (\zeta_{\pi_{K_0}(t)^-})) \ | \ \mathcal{F}_0 \right] \right) \right| \notag \\
  & \qquad \leq     \abs{\cov{f(\zeta_0)}{\CE{\ind(\pi_{1}(t)=0\text{ and }\pi_0(t)>0)}{\mathcal{F}_0}\ \bigchi_{K_{1},a_{K_{1}}}(\zeta_0)}}   \notag \\
  & \qquad \qquad   +
     \abs{\cov{f(\zeta_0)}{\CE{\ind(\pi_{1}(t)>0)\  \bigchi_{K_{1},a_{K_{1}}}(\zeta_{\pi_{1}(t)})}{\mathcal{F}_0}}} \notag \\
    & \qquad \leq 2 \norm{f}_{\infty}\ \mathbb{P} \left( \pi_1(t) =0 \mbox{ and } \pi_0(t) >0 \right) \notag\\
  & \qquad  \qquad +       \abs{\cov{f(\zeta_0)}{\CE{\ind(\pi_{1}(t)>0) \ \bigchi_{K_{1},a_{K_{1}}}(\zeta_{\pi_{1}(t)})}{\mathcal{F}_0}}} . \label{proof_main_result_recursion_first_time_step_2}
\end{align}
Hence, a combination of~\eqref{proof_main_result_recursion_first_time_step_1} and~\eqref{proof_main_result_recursion_first_time_step_2} shows that
\begin{align}
     \sup_{K_0, a_{K_0}} & \left|  \Cov \left( f(\zeta_0),  \mathbb{E} \left[ \ind(\pi_{K_0}(t) > 0) \bigchi_{K_0, a_{K_0}} (\zeta_t) \ | \ \mathcal{F}_0 \right] \right) \right| \notag \\
    & \qquad \leq \beta \ 2 \norm{f}_{\infty}\ \mathbb{P} \left( \pi_1(t) =0 \mbox{ and } \pi_0(t) >0 \right) \notag\\
  & \qquad  \qquad +     \beta   \sup_{K_1, a_{K_1}} \abs{\cov{f(\zeta_0)}{\CE{\ind(\pi_{1}(t)>0) \ \bigchi_{K_{1},a_{K_{1}}}(\zeta_{\pi_{1}(t)})}{\mathcal{F}_0}}} \label{proof_main_result_first_step_completed}.
\end{align}
We observe that the term on the left hand side of~\eqref{proof_main_result_first_step_completed} and the second term on the right hand side is of the same form. This allows us to perform a recursive argument: \smallskip
For a sequence of finite but non-empty sets~$K_0$,~$K_1$, \ldots we introduce the notation
$$\pi_n(t):= \bb{\pi_{K_n}\circ\dots\circ\pi_{K_0}}(t).$$
We observe that~$\ind(\pi_n(t)>0)$ means that there were at least~$n$-updates up to time~$t$, first in the set $K_n$, then in $K_{n-1}$,.... With the same argument as used in~\eqref{proof_main_result_first_step_completed}, we can deduce that 
\begin{align}
     \sup_{K_n, a_{K_n}} & \left|  \Cov \left( f(\zeta_0),  \mathbb{E} \left[ \ind(\pi_{K_n}(t) > 0) \bigchi_{K_n, a_{K_n}} (\zeta_{\pi_n(t)}) \ | \ \mathcal{F}_0 \right] \right) \right| \notag \\
    & \leq \beta \ 2 \norm{f}_{\infty}\ \mathbb{P} \left( \pi_{n+1}(t) =0 \mbox{ and } \pi_n(t) >0 \right) \notag\\
  &   \qquad +     \beta \  \sup_{K_{n+1}, a_{K_{n+1}}}  \abs{\cov{f(\zeta_0)}{\CE{\ind(\pi_{n+1}(t)>0) \ \bigchi_{K_{n+1},a_{K_{n+1}}}(\zeta_{\pi_{n+1}(t)})}{\mathcal{F}_0}}} \label{proof_main_result_step_n_completed}.
\end{align}
Overall, a combination of~\eqref{proof_main_result_recursion_first_step},~\eqref{proof_main_result_first_step_completed} and~\eqref{proof_main_result_step_n_completed} yields the estimate (setting formally~$\pi_{-1} =1$)
\begin{align*}
& \sup_{K_0\neq \vn,\;a_{K_0}}  \abs{\cov{f(\zeta_0)}{\bigchi_{K_0,a_{K_0}}(\zeta_t)}} \\
&  \qquad \leq \sup_{K_0,\dots,K_n\neq \vn}\sum_{m=0}^{n}  2 \ \beta^m\norm{f}_{\infty}\cdot\pP\bb{\pi_m(t)=0\text{ and }\pi_{m-1}(t)>0} \\
& \qquad \quad + \sup_{K_0,\dots,K_n\neq \vn,\;a_{K_n}} \beta^n\cdot\abs{\cov{f(\zeta_0)}{\CE{\ind(\pi_n(t)>0)\cdot \bigchi_{K_{n},a_{K_n}}(\zeta_{\pi_n(t)})}{\mathcal{F}_0}}}.
\end{align*}

Let us now analyze the limit~$n \to \infty$. This last term is bounded by $ 2\beta^n\norm{f}_{\infty}$ and disappears.
Thus we need only concern ourselves with the series of probabilities. For a fixed sequence of sets $(K_n)_{n\geq 0}$ we get
\begin{align*}
    \sum_{m=0}^{n} \beta^m & \cdot\pP\bb{\pi_m(t)=0\text{ and }\pi_{m-1}(t)>0} \\
    & = \sum_{m=0}^{n} \bb{\beta^m-\beta^{m+1}}\cdot\sum_{i\leq m}\pP\bb{\pi_i(t)=0\text{ and }\pi_{i-1}(t)>0} \\
    & \qquad + \beta^{n+1} \cdot\sum_{i\leq n}\pP\bb{\pi_i(t)=0\text{ and }\pi_{i-1}(t)>0}\\
    & \leq \sum_{m=0}^{n} \bb{\beta^m-\beta^{m+1}}\cdot\pP\bb{\pi_m(t)=0} + n\beta^{n+1} \\
    & = \sum_{m=0}^{n} \bb{\beta^m-\beta^{m+1}}\cdot\pP\left( \pi_{K_m}\circ\dots\circ\pi_{K_0}(t)=0 \right) + n\beta^{n+1} .
\end{align*} 
The probabilities are estimated in the following way. The increments $(\pi_{i}(t)-\pi_{i+1}(t))$ are independent exponential random variables with rates~$\abs{K_i}$. Since $K_i\neq \vn$ we may bound the probability by the worst case scenario, where all $K$s are of minimal size $1$, which yields
\begin{align*}
    \sum_{m=0}^{n} & \bb{\beta^m-\beta^{m+1}}  \cdot\pP\bb{\underbrace{\pi_{\cb{j}}\circ\dots\circ\pi_{\cb{j}}}_{m+1}(t)=0} \\
    & \qquad \leq \sum_{m=0}^{n} \bb{\beta^m-\beta^{m+1}}\cdot\sum_{k=0}^m \frac{t^k}{k!}e^{-t} \leq  e^{-\bb{1-\beta}t}.
\end{align*}
% This last equality holds because 
% $$\cb{\underbrace{\pi_{\cb{0}}\circ\dots\circ\pi_{\cb{0}}}_{m+1}(t)=0}$$ 
% is exactly the event that a sum of $m+1$ independent $\text{Exp}(1)$ random variables is greater than $t$. The probability of this event is the same as the probability that a Poisson point process with rate $1$ will not exceed $m$ at time $t$.
Overall, we have obtained the desired estimate~\eqref{proof_main_result_weak} which finishes the first part of the proof.\\

Let us turn to the second part, namely deducing the estimate~\eqref{proof_main_result_strong} under the weaker assumption~$\alpha <1$. The sole difference between~$\alpha$ and~$\beta$ is that in the definition of~$\alpha$ there is one absolute value missing in the coefficients  $C_{\cb{j}, a}^{\bb{j, a}}$ (see~\eqref{equ:main_theo_condition_contractivity_strong}). We observe that the coefficients $C$s are linear functions of the entries of the transition matrix, as they are the solutions to a linear system of equations
$$\desum{S\subseteq\mathcal{N}_j}{a_S\in\bb{\mathcal{A}\setminus{\{\ubar{a}\}}}^S}C_{S,\;a_S}^{\bb{j,a}}\cdot\bigchi_{S,\;a_S}\bb{\sigma} = E_j\bb{\bigchi_{\cb{j},a}}\bb{\sigma}=P_j(a_+\;|\;\sigma)\cdot x(a) + P_j(a_-\;|\;\sigma)\cdot y(a),$$
for every configuration $\sigma$.
For the identity transition matrix, the values of $C$s are $C_{S,\;a_S}^{\bb{j,a}}=1$ for $\bb{S,\;a_S}=\bb{\cb{j},a}$ and $0$ otherwise. To use the \hyperref[pro:time-scaling]{Time-Scaling Lemma} we will find some $\lambda\in (0,1]$ such that
$$1-\lambda + \lambda C_{\cb{j},a}^{\bb{j,a}}\geq 0\text{ for all }j\in\Lambda, a\in\mathcal{A}\setminus\{\ubar{a}\}$$
and
$$\sup_{j\in\Lambda,\; a\in\mathcal{A}\setminus\{\ubar{a}\}}\bb{\abs{1-\lambda + \lambda C_{\cb{j},a}^{\bb{j,a}}} + \desum{S\subseteq \mathcal{N}_j, a_S\in\bb{\mathcal{A}\setminus\{\ubar{a}\}}^S}{\bb{S,\;a_S}\neq \bb{\cb{j},a}}\lambda\abs{C_{S,a_S}^{\bb{j,a}}}} <1.$$
If the first condition is satisfied then the second one is equivalent to the assumption $\alpha <1$.
Of course, such a $\lambda$ exists since $C_{\cb{j},a}^{\bb{j,a}}$ are uniformly bounded.
Set \mbox{$Q = \lambda P + \bb{1-\lambda}I$}.
The bound for \hyperref[def:covdec]{covariance decay} holds for $\text{IPS}(Q)$ with rate $\lambda\alpha$. By the \hyperref[pro:time-scaling]{Time-Scaling Lemma} it also holds for IPS$(P)$ with rate $\alpha$, which concludes the second part of the proof.
\end{proof}

\subsection{Example: The two-stage contact process~}\label{sec:two_stage_contact_process}
This generalization of the original contact process was first introduced by Stephen M. Krone in~\cite{Krone:99}. It was meant as a toy model of population growth with a distinction between young and adult members, where only the adults can spread. It is a homogeneous IPS with parameters $\bb{\lambda, \gamma, \delta}$, which are the rates of respectively spreading to a neighboring site by an adult, maturing to adulthood, and "infant mortality". Besides the last factor, all members of the population die out at a constant rate. For a homogeneous lattice, the two-stage contact process with those parameters is given by a transition matrix

$$P_j(0\;|\;\sigma) = \left\{\begin{array}{cc}
    1-\frac{\lambda}{\beta}\abs{\sigma^{-1}(2)\cap\mathcal{N}_j}, &  \sigma(j)=0 \\
    \frac{1+\delta}{\beta}, & \sigma(j)=1\\
    \frac{1}{\beta}, & \sigma(j) = 2
\end{array}\right.,$$

$$P_j(1\;|\;\sigma) = \left\{\begin{array}{cc}
    \frac{\lambda}{\beta}\abs{\sigma^{-1}(2)\cap\mathcal{N}_j}, &  \sigma(j)=0 \\
    1-\frac{\gamma}{\beta}-\frac{1+\delta}{\beta}, & \sigma(j)=1\\
    0, & \sigma(j) = 2
\end{array}\right.,$$
and
$$P_j(2\;|\;\sigma) = \left\{\begin{array}{cc}
    0, &  \sigma(j)=0 \\
    \frac{\gamma}{\beta}, & \sigma(j)=1\\
    1-\frac{1}{\beta}, & \sigma(j) = 2
\end{array}\right.,$$
where $\beta = 1+\delta + \gamma + \lambda\abs{\mathcal{N}}$, zeroes represent the vacant sites, and ones and twos represent the young and adult members of the population respectively.\\

The ergodicity of this process is equivalent to the population dying out, as $\delta_{0}$ is an invariant measure. We will use Theorem~\ref{pro:main-result} to find a sufficiency condition for that to happen. We use product bases of $\mathbf{F}$ given by
\begin{align*}
     \bigchi_{j,1}(\sigma) = A\cdot\ind\bb{\sigma(j)\geq 1} \qquad \mbox{and} \qquad
      \bigchi_{j,2}(\sigma) = B\cdot\ind\bb{\sigma(j)= 2},
\end{align*}
i.e.~their values will be $x(1)=A$, $x(2)=B$ and $y \equiv 0$. If $A$ and $B$ are bounded by $\pm 1$ then the Conditions $1 - 4$ of Theorem~\ref{pro:main-result} are met. Let us now check when Condition $5$ holds. First, we note that

$$E_j\bb{\bigchi_{j,a}} = x(a)\cdot P_j\bb{a_+\;|\;\sigma},$$
and 
$$\abs{\sigma^{-1}(2)\cap\mathcal{N}_j} = \sum_{i\in\mathcal{N}_j}\frac{1}{B}\bigchi_{i,2}(\sigma).$$
Then we can express indicators of single-site cylinders as
\begin{align*}
 \ind\bb{\sigma(j)=0} & = 1-\frac{1}{A}\bigchi_{j,1},\\
 \ind\bb{\sigma(j)=1} & = \frac{1}{A}\bigchi_{j,1}-\frac{1}{B}\bigchi_{j,2}, \mbox{ and} \\
\ind\bb{\sigma(j)=2} & = \frac{1}{B}\bigchi_{j,2}.
\end{align*}
Let us now determine condition $5$. We observe that
\begin{align*}
    E_j\bb{\bigchi_{j,1}} & = A\bb{1-\frac{1}{A}\bigchi_{j,1}}\frac{\lambda}{\beta}\sum_{i\in\mathcal{N}_j}\frac{1}{B}\bigchi_{i,2}\\
    &\qquad +A\bb{\frac{1}{A}\bigchi_{j,1}-\frac{1}{B}\bigchi_{j,2}}\bb{1-\frac{1+\delta}{\beta}} + A\bb{\frac{1}{B}\bigchi_{j,2}}\bb{1-\frac{1}{\beta}}\\
    &=\bigchi_{j,1}\bb{1-\frac{1+\delta}{\beta}} + \bigchi_{j,2}\bb{\frac{A\lambda}{B\beta}-\frac{\lambda}{B\beta}+\frac{A\delta}{B\beta}}\\
    & \qquad +\sum_{i\in\mathcal{N}_j\setminus\cb{j}} \bigchi_{i,2}\frac{A\lambda}{B\beta} - \bigchi_{\cb{i,j},\bb{2,1}}\frac{\lambda}{B\beta} ,
\end{align*}
and
\begin{align*}
    E_j(\bigchi_{j,2}) & =B\bb{\frac{1}{A}\bigchi_{j,1}-\frac{1}{B}\bigchi_{j,2}}\frac{\gamma}{\beta} + B\bb{\frac{1}{B}\bigchi_{j,2}}\bb{1-\frac{1}{\beta}}\\
    &=\bigchi_{j,1}\frac{B\gamma}{A\beta} +\bigchi_{j,2}\bb{1-\frac{1}{\beta}-\frac{\gamma}{\beta}}.
\end{align*}
Then Condition $5$ is equivalent to theinequalities:
\begin{align}
     1-\frac{1+\delta}{\beta} + \abs{\frac{A\lambda}{B\beta}-\frac{\lambda}{B\beta}+\frac{A\delta}{B\beta}} + \bb{\abs{\mathcal{N}}-1}\frac{\bb{1+\abs{A}}\lambda}{\abs{B}\beta} & < 1, \mbox{ and}\label{eq:two_stage_condition_one}\\
     \abs{\frac{B\gamma}{A\beta}}+1-\frac{1}{\beta}-\frac{\gamma}{\beta} & <  1.\label{eq:two_stage_condition_two}
\end{align}

Our goal now is to find $A$ and $B$ such that this condition is satisfied. We may choose any $A, B\in[-1,1]\setminus\cb{0}$. We note that the sign of $B$ does not matter. Choosing $A$ of the same sign as $B$ minimizes the third term in (\ref{eq:two_stage_condition_one}) without affecting anything else. Hence we may restrict our search to positive $A$, $B$. It's clear that whatever values $A$, $B$ we choose we may relax (\ref{eq:two_stage_condition_one}) by increasing $B$. There is a bound on $B$ imposed by (\ref{eq:two_stage_condition_two}) and $\abs{B}\leq 1$ i.e.

$$\abs{B}< \min\bb{1,\;A\frac{1+\gamma}{\gamma}}.$$
We may proceed with $B = \min\bb{1,\;A\frac{1+\gamma}{\gamma}}$ since if such $B$ permits a choice of $A$ satisfying  (\ref{eq:two_stage_condition_one}) then by decreasing $B$ slightly we will arrive at a pair $(A,B)$ satisfying both inequalities. To find the optimal $A$ is now just a question of algebra. Under the assumption that $\abs{\mathcal{N}}\geq 2$ the optimal $A$ is given by $\max\bb{\frac{\gamma}{1+\gamma}, \frac{\lambda}{\lambda + \delta}}$. The case $\abs{\mathcal{N}}=1$ may be disregarded since for $\abs{N}=1$ the IPS has no interactions between particles. We observe that this value of $A$ forces $B$ to equal one. Hence the condition we arrive at is

    $$\max\bb{\frac{\gamma}{1+\gamma}, \frac{\lambda}{\lambda + \delta}} < \frac{1+2\lambda + \delta -  \lambda\abs{\mathcal{N}}}{\delta + \lambda\abs{\mathcal{N}}}.$$

If this inequality holds then by Theorem~\ref{pro:main-result} the two-stage contact process with parameters $\bb{\lambda,\gamma,\delta}$ dies out.  
\bigskip

{}

\section{Positive rates conjecture for one-sided nearest-neighbor interactions} \label{sec:application}

In this section we study an open problem, namely the positive rates conjecture for IPS with an alphabet of size two and one-sided nearest-neighbor interactions. Despite G\'acs counterexample, we believe the PRC holds for "simple" enough IPS. His construction required the alphabet and interaction range to be on the order of $2^{300}$ to generate a phase transition. The IPS we will be dealing with lack sufficient complexity for non-ergodic behavior. Additionally, Głuchowski and Miekisz have carried out extensive simulations of these simple IPS looking for any counterexamples among them. They have simulated the standard coupling between trajectories of these IPS starting from different initial configurations. The results suggest that their disagreement does not percolate through spacetime implying \hyperref[def:ergodicity]{ergodicity}. Section~\ref{sec:one_d_ips} summarizes the findings of this numerical study, categorizing the different types of behavior. \\

When studying ergodicity it is natural to consider IPS with an alphabet of size two and one-sided nearest-neighbor interactions. All simpler IPS are ergodic. If the lattice is finite then the assumption of positive rates implies ergodicity, as the PCA or IPS is a finite, irreducible Markov chain.
 % It is a textbook exercise that such Markov chains are \hyperref[def:ergodicity]{ergodic}. 
Similarly, if the neighborhood is smaller, i.e.~$\abs{\mathcal{N}_j}\leq1$ for all $j\in\Lambda$, the resulting PCA or IPS is ergodic, irrespective of the underlying lattice structure. Indeed, Theorem~\ref{pro:leontovitch} yields that conclusion for PCAs and Theorem~\ref{pro:main-result} for IPS (using the product basis with values $x=1$, $y=-1$). This leaves $\bb{\Lambda,\mathcal{N}} = \bb{\Z,\cb{0,1}}$ as the simplest lattice for which it is not known whether positive rates imply ergodicity, even for alphabets of size two. \\

Ergodicity of this class of IPS was studied by Toom et al. \cite{Toom:90} applying a variety of methods including Theorem~\ref{pro:leontovitch} by Leontovitch \& Vaserstein. To be more precise, they considered the following class of PCAs:\\

\begin{itemize}\label{item:conditions}
    \item The lattice is $\bb{\Lambda,\mathcal{N}} = \bb{\Z,\cb{0,1}}$.
    \item The alphabet $\mathcal{A}=\cb{0,1}$.
    \item The transition matrix $P$ is homogeneous and has positive rates.
    \item The transition matrix is additive (only depends on the sum of states in the interaction neighborhood), which in this case means
    $$P\bb{1\;|\;10} = P\bb{1\;|\;01}.$$
    (This is a different property than Griffeath's "additive".)\\
\end{itemize}
    The sum of their methods proved uniform exponential \hyperref[def:covdec]{covariance decay} for roughly $90$\% of the parameter space, but is missing parts of the low-noise regime. They still conjectured that this class of PCA is ergodic. Instead of PCAs, we consider the same class of IPS in continuous time, i.e.~choosing update clocks to be iid.~$\mbox{Exp}(1)$. As a corollary of Theorem~\ref{pro:main-result}, we show below that the whole class is ergodic. 
    It is not sufficient to use the less powerful assumption~\eqref{equ:main_theo_condition_contractivity_strong} instead of~\eqref{eq_contractivity_gen_leontovitch}. Therefore the time-scaling trick substantially increased the power of the original criterion of Leontovitch and Vaserstein in the continuous time setting. \\

\begin{coro}\label{compare}
    An IPS with the properties above is uniformly exponentially \hyperref[def:ergodicity]{ergodic}.
\end{coro}
\begin{proof}[Proof of Corollary~\ref{compare}]
    
First, we give the parameters of an IPS more convenient names, namely
$$p_{1|11}:=P(1\;|\;11),\, p_{1|10}:=P(1\;|\;10),\, p_{1|01}:=P(1\;|\;01),\, p_{1|00}:=P(1\;|\;00).$$
Second, one can reduce the parameter space from a four-dimensional hypercube to four three-dimensional cubes by projecting each point $(p_{1|11},\,p_{1|10},\,p_{1|01},\,p_{1|00})$ along the line joining it with $(1,1,0,0)$ (the identity) onto one of the four faces opposite to $(1,1,0,0)$.
By the \hyperref[pro:time-scaling]{Time-Saling Lemma} uniform, exponential \hyperref[def:ergodicity]{ergodicity} is preserved along those lines, with only the rate of decay changing.
Thus we may without loss of generality assume that $p_{1|11}=0$ or $p_{1|10}=0$ or $p_{1|01}=1$ or $p_{1|00}=1$. Further, the cases $p_{1|11}=0$, $p_{1|00}=1$ and $p_{1|10}=0$, $p_{1|01}=1$ are equivalent by renaming the state $0$ to $1$ and vice versa. In total, we may assume that $p_{1|11}=0$ or $p_{1|10}=0$.
The assumption of \hyperref[positive rates]{positive rates} remains unaffected i.e.
$$p_{1|11}<1,\quad p_{1|10}<1,\quad p_{1|01}>0,\quad p_{1|00}>0.$$
The additive parametrs $\cb{p_{1|10}=p_{1|01}}$ map onto the set $\cb{p_{1|10}\leq p_{1|01}}$.\\

\textit{Case 1: }$p_{1|11}=0.$\smallskip

We verify by direct calculation that for the choice of basis with values $x=1$, $y=-1$ the update coefficients are
\begin{align*}
& C_{\vn} = \frac{p_{1|10}+p_{1|01}+p_{1|00}}{2}-1, \;  C_{\cb{0}} = \frac{p_{1|10}-p_{1|01}-p_{1|00}}{2}, \; \mbox{and} \\
& C_{\cb{1}}=\frac{-p_{1|10}+p_{1|01}-p_{1|00}}{2}, \; C_{\cb{0,1}} = \frac{-p_{1|10}-p_{1|01}+p_{1|00}}{2}.
\end{align*}
The condition
$$\abs{C_{\vn}} + C_{\cb{0}} + \abs{ C_{\cb{1}}} + \abs{ C_{\cb{0,1}}} <1 $$
is in this case equivalent to the assumption of \hyperref[positive rates]{positive rates} and $p_{1|10}< p_{1|01}+p_{1|00}$.
This last requirement is met when $p_{1|10} \leq p_{1|01}$ and $p_{1|00}>0$.

\begin{figure}[H]
        \centering
        \includegraphics[scale=0.175]{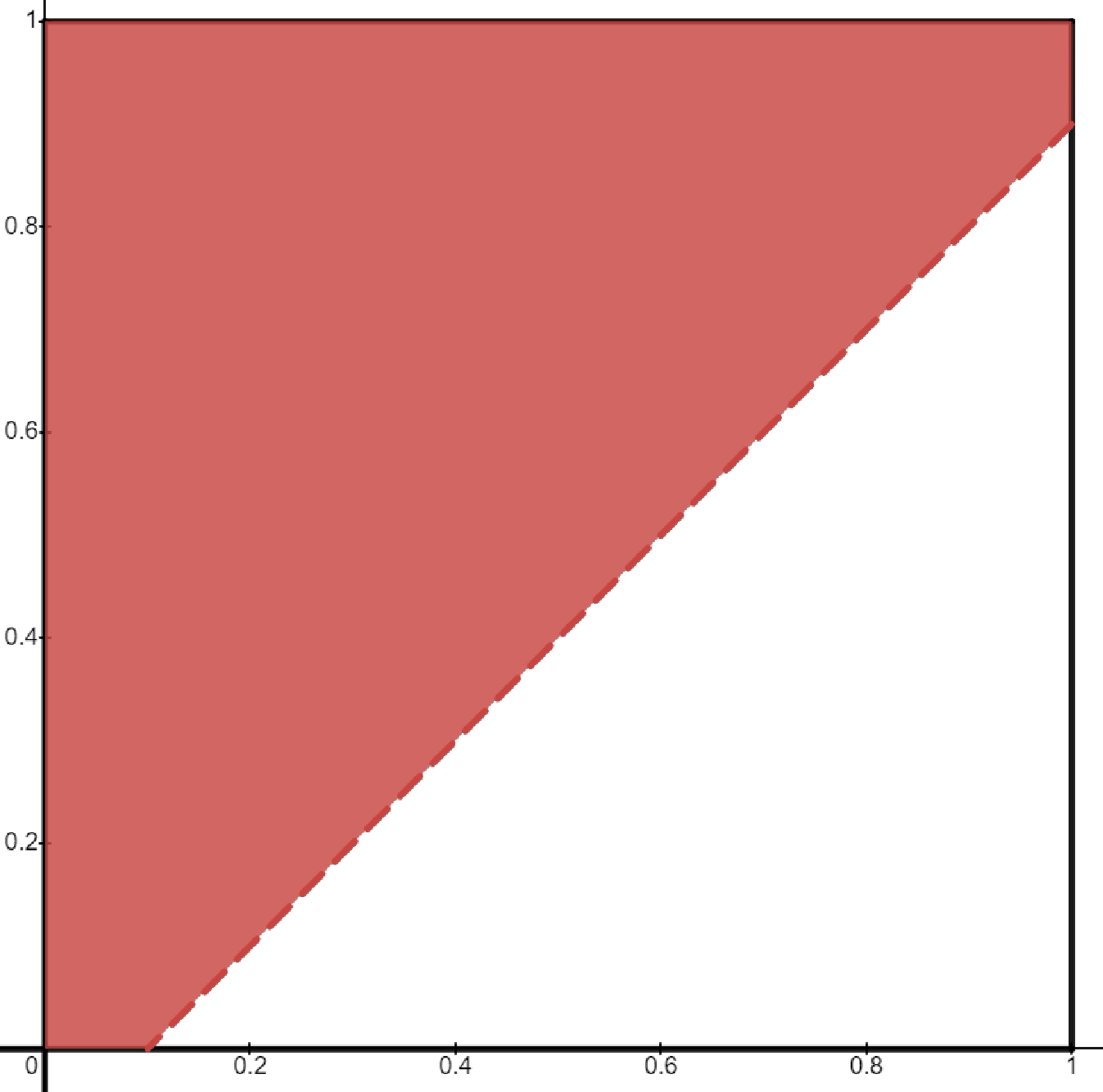}
        \caption{Region of the parameter space where the inequality holds. Parameters $p_{1|10}$ and $p_{1|01}$ vary over $X$ and $Y$ axis respectively. The value of $p_{1|00}$ is set to $0.1$.}
    \end{figure}

\textit{Case 2: }$b=0$.\smallskip

To tackle this case we will use two different (albeit similar) bases. These are the bases with values $x=1$, $y=-p_{1|00}$ and $x=-p_{1|00}$, $y=1$.
Together their respective regions cover the entire $\cb{p_{1|10}=0}$ cube, except for some points on the border $\cb{p_{1|11}=1}\cup \cb{p_{1|00}=0}$ where the \hyperref[positive rates]{positive rates} assumption fails.
\end{proof}
\begin{figure}[H]
        \centering
        \includegraphics[scale=0.175]{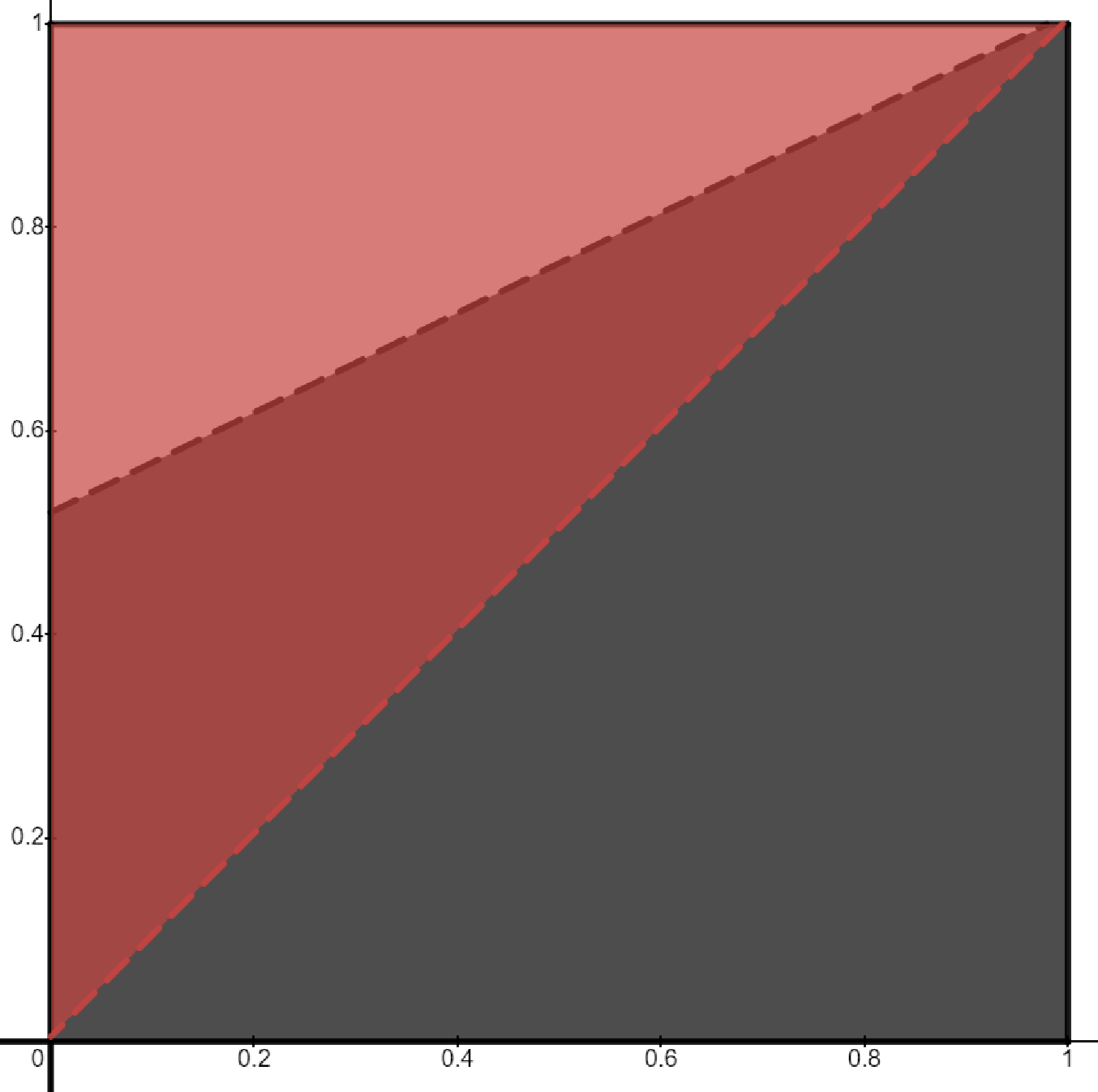}
        \caption{Regions of the parameter space where the criteria for these bases hold (black for the former, red for the latter). Parameters $p_{1|11}$ and $p_{1|01}$ vary over $X$ and $Y$ axis respectively. The value of $p_{1|00}$ is set to $0.01$. For $p_{1|00}>0.25$, the first base criterion fills the entire square.}
\end{figure}

One might be curious if the results now known are sufficient to prove that positive rates imply \hyperref[def:ergodicity]{ergodicity} in the case $\mathcal{A} = \cb{0,1}$, $\mathcal{N} = \cb{0,1}$, with no further restrictions. In the example above the problem was reduced by a dimension and solved in all cases except when $p_{1|11}=0$ and $p_{1|10} \geq p_{1|01}+p_{1|00}$. 
Unfortunately, our efforts fall short. Trying out different bases in Theorem~\ref{pro:main-result}  extends the covered region of parameters slightly, but there are points for which no product base will work. For example for $p_{1|11}=0$, $p_{1|10}=0.99$, $p_{1|01}=0.05$, $p_{1|00}=0.01$ there are no solutions in $x\in \R\setminus\cb{0}$, $y\in \R$ to the inequality
$$\abs{C_{\vn}} + C_{\cb{0}} + \abs{ C_{\cb{1}}} + \abs{ C_{\cb{0,1}}} <1.$$
We have also experimented with making the product basis periodic, as in choosing different values for $\bigchi_j$ depending on the parity of $j\in\Z$, but that proved fruitless.
\begin{figure}[H]
        \centering
        \includegraphics[scale=0.175]{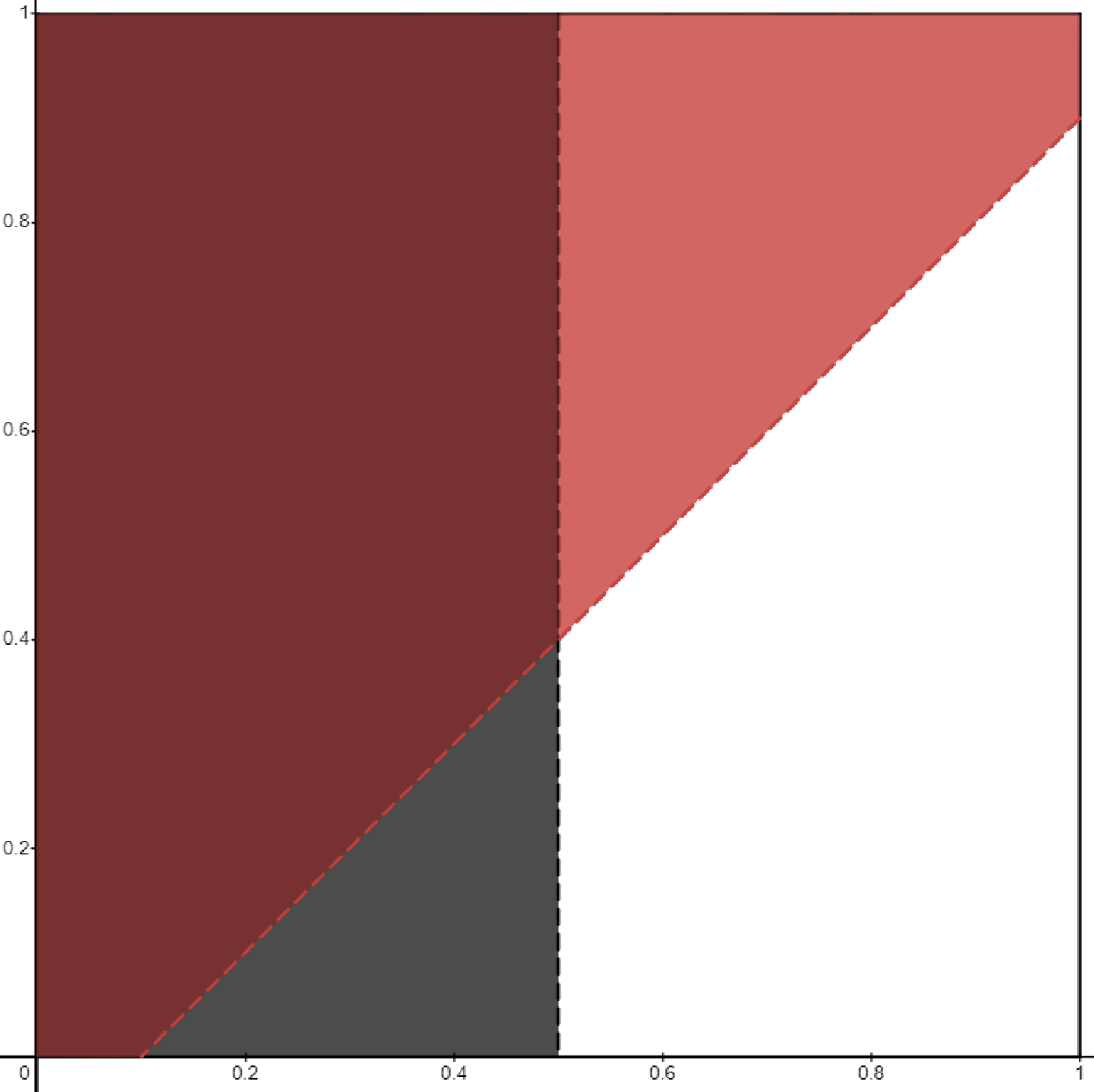}
        \caption{Using Theorem~\ref{pro:main-result} with $x=1$ and $y\in [-1,0]$ one can additionally prove \hyperref[def:ergodicity]{ergodicity} for the cases $\cb{p_{1|11}=0, p_{1|10}<\frac{1}{2}, p_{1|01}>0, p_{1|00}>0}$, but our theorem is powerless in the neighborhood of $\text{IPS}(0,1,0,0)$. Parameters $p_{1|10}$ and $p_{1|01}$ vary over $X$ and $Y$ axis respectively. The value of $p_{1|00}$ is set to $0.1$.}
\end{figure}

It is not entirely surprising that this method fails to deduce \hyperref[def:ergodicity]{ergodicity} on the entire parameter space. Theorem~\ref{pro:main-result} is blind to the structure of $\Lambda$ and in particular does not take into account its dimension. There are no known counterexamples of "positive rates imply ergodicity" among nearest-neighbor IPS on $\Z$, but there are simple ones on $\Z^2$. A classic example is the Stochastic Ising Model - see chapter IV of Ligett's textbook on IPS \cite{Ligett:05}. Since dimension seems to be critical we should expect that one needs to strongly exploit the one-dimensionality of $\Z$ heavily to cover the entire parameter space. Gray's theorem about monotone IPS i.e.~Theorem~\ref{pro:monoextend} allows to deduce \hyperref[def:ergodicity]{ergodicity} for parts of the remaining region $$\cb{p_{1|11}=0, p_{1|10}\approx 1, p_{1|01}\approx 0, p_{1|00}\approx 0}.$$ Not suprisingly, the proof of~Theorem~\ref{pro:monoextend} strongly uses the dimension of $\Z$, namely path crossing properties of a two-dimensional spacetime. AS a corollary we deduce the following statement.  \\

\begin{coro}\label{using gray}
    A homogeneous IPS$(P)$ on lattice $\bb{\Z,\cb{0,1}}$ with alphabet $\cb{0,1}$ and parameters that satisfy $p_{1|11}\leq p_{1|10}<1$ and $0<p_{1|01}\leq p_{1|00}$ is exponentially \hyperref[def:ergodicity]{ergodic}. 
\end{coro}

\begin{proof}[Proof of Corollary~\ref{using gray}]
The strategy of the proof follows the idea of Gray found in Section $3$ of \cite{Gray:82}. The IPS$(P)$ is not itself \hyperref[def:weak monotonicity]{weakly monotone}. However, if we flip $0$s and $1$s at every other site the resulting $\text{IPS}(Q)$ will have a periodic transition matrix with the same parameters as $P$, just in a different order. While it may not be \hyperref[def:monotonicity]{monotone} (and in the case $a=0$ it cannot) the inequalities we imposed guarantee that it will be \hyperref[def:weak monotonicity]{weakly monotone} and have positive rates. Then by Theorem~\ref{pro:monoextend} this new $\text{IPS}(Q)$ will be exponentially \hyperref[def:ergodicity]{ergodic}. But of course, renaming states doesn't change the underlying dynamics, so the original IPS$(P)$ must also be exponentially \hyperref[def:ergodicity]{ergodic}. \\

Let us turn to the formal proof: Let $\zeta \in \text{IPS}(P)$ be a trajectory constructed as in Definition~\ref{def:general construction}. We will define a modification of this trajectory
$$\xi_t(j) = \left\{\begin{array}{lc}
     \zeta_t(j), & j=0\bb{\text{mod }2},  \\
     1-\zeta_t(j), & j=1\bb{\text{mod }2},
\end{array}\right. \qquad \text{for all }j\in\Z\text{ and }t\geq 0.$$
Clearly if $t\in\left[\tau_n^j, \tau_{n+1}^j\right)$ then $\xi_t(j) = \xi_{\tau_n^j}(j)$ because $\zeta$ satisfies this condition and $\xi_t$ depends only on $\zeta_t$. Additionally $\xi$ evolves according to a transition matrix $Q$ given by
\begin{align*}
    Q_j\bb{z\;|\;x,y} & := \pP\bb{\xi_{\tau_n^j}=z\;|\;\xi_{\tau_n^{j-}}(j,j+1)=(x,y)}\\
    \mbox{} & \\
    &=\left\{\begin{array}{lc}
     \pP\bb{\zeta_{\tau_n^j}=z\;|\;\zeta_{\tau_n^{j-}}(j,j+1)=(x,1-y)}, & j=0\bb{\text{mod }2}, \\
     \pP\bb{\zeta_{\tau_n^j}=1-z\;|\;\zeta_{\tau_n^{j-}}(j,j+1)=(1-x,y)}, & j=1\bb{\text{mod }2}
\end{array}\right.\\
\mbox{} & \\
& = \begin{cases}
    P_j\bb{z\;|\;x,1-y}, & j=0\bb{\text{mod }2}, \\
    P_j\bb{1-z\;|\;1-x,y}, & j=1\bb{\text{mod }2}.
\end{cases}
\end{align*}

Thus $\xi$ is by Definition~\ref{def:IPS} a trajectory of $\text{IPS}(Q)$. The matrices $Q_j$ have parameters $(p_{1|10},p_{1|11},p_{1|00},p_{1|01})$ when $j$ is even and $(1-p_{1|01},1-p_{1|00},1-p_{1|11},1-p_{1|10})$ when $j$ is odd. Their \hyperref[def:weak monotonicity]{weak monotonicity} and positive rates are equivalent to the conditions we set. By Theorem~\ref{pro:monoextend} the $\text{IPS}(Q)$ is exponentially \hyperref[def:ergodicity]{ergodic}. Equivalently it has exponential covariance decay. But clearly, if covariances decay for trajectory $\xi$ then they decay at the same rate for trajectory $\zeta$, because for any $f\in \mathbf{F}$ the function $$\sigma \mapsto f\bb{\left\{\begin{array}{lc}
    \sigma(j), &  j=0\bb{\text{mod }2}\\
   1-\sigma(j),  & j=1\bb{\text{mod }2}
\end{array}\right.}$$
is also a member of $\mathbf{F}$ (this transformation does not change $\text{Dom}(f)$). Thus IPS$(P)$ has exponential covariance decay. Equivalently it is exponentially \hyperref[def:ergodicity]{ergodic}.
\end{proof}
\begin{figure}[H]\label{fig:final_tally}
        \centering
        \includegraphics[scale=0.175]{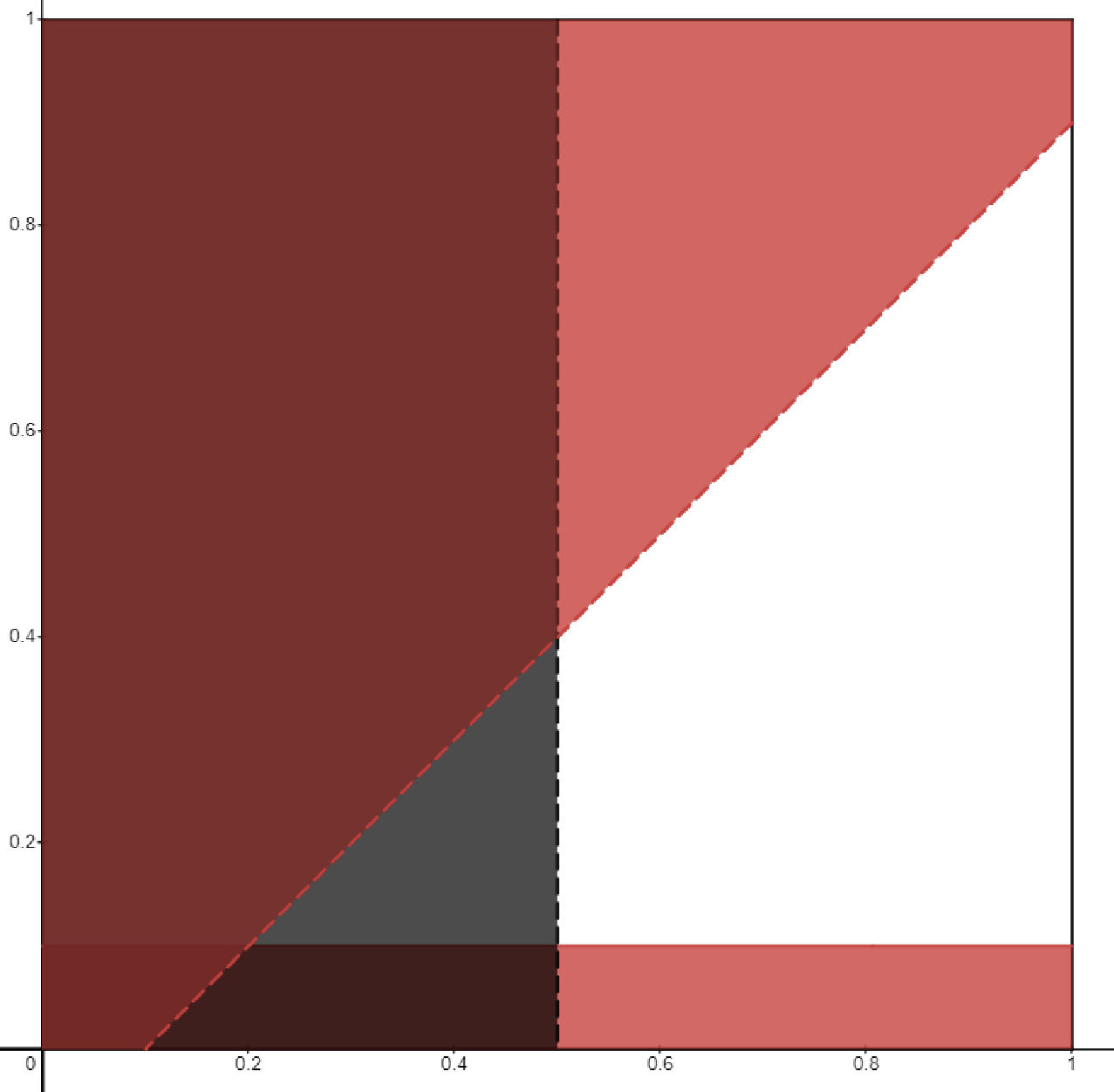}
        \caption{Union of regions in case $\cb{p_{1|11}=0}$ for which we have a proof of \hyperref[def:ergodicity]{ergodicity}. The lower red bar is the region covered in Example~\ref{using gray}. Parameters $p_{1|10}$ and $p_{1|01}$ vary over $X$ and $Y$ axis respectively. The value of $p_{1|00}$ is set to $0.1$.}
\end{figure}

We have been unable to prove \hyperref[def:ergodicity]{ergodicity} in the remaining region. The troublemaker IPS seem to be the "Walls IPS" that we described in Section \ref{subsec:walls}. What these IPS do is try to preserve homogenous blocks of one state while the other is preserved only in the presence of the first. As such these IPS are not monotone, additive or cancellative so ergodicity criteria developed for these types of IPS cannot be applied. Theorem \ref{pro:main-result} also fails here. We believe the reason why is that the method only accounts for the behavior of IPS after one update. Meanwhile, the Walls IPS produce trajectories that seem elongated in the time direction (see Figure~\hyperref[fig:walls]{11} in Section \ref{subsec:walls}), implying that information from just one time-step doesn't capture much of its long-term behavior. \\

However, simulations and heuristics suggest that these Walls IPS should also be \hyperref[def:ergodicity]{ergodic}. If an IPS was not \hyperref[def:ergodicity]{ergodic} one would expect that by decreasing noise in its rule one would get another non-\hyperref[def:ergodicity]{ergodic} IPS. Thus if there was a non-ergodic IPS anywhere in the interior of the parameter space then there should exist a path joining that point with the boundary, along which all IPS are non-ergodic.  Meanwhile, the least noisy, semi-deterministic IPS, located in the vertices of the parameter space, seem to be cut off from the remaining region (see Figure~\hyperref[fig:final_tally]{15}). The simplest case of the Positive Rates Conjecture, now reduced to the problem of proving the ergodicity of Walls IPS requires further analysis, but we believe it to be quite approachable. \\

\section{Appendix}\label{appendix}

This appendix is not necessary to understand the overall picture of the article. We add it to provide clarity of claimed and cited results in Section~\ref{sec:extensions}. The complication comes from the use of different settings and notation, which would have obstructed the understanding of the main ideas. We propose that before reading the subsequent sections readers familiarize themselves with the notation of the work referenced.\\

\bigskip

\subsection{Addendum to Theorem~\ref{pro:additive}: Characterization of additive IPS} \label{sec:char_additive_IPS}

The notation of Griffeath's construction can be found in the "percolation substructures" and "general construction" sections of \cite{Griffeath:78}.
This construction permits IPS which updates many sites at once. The first thing we need to do then is to identify the choices of $V_{i,x}$ and $W_{i,x}$ for which only one site can updated at a time.
It is easy to see that the possible choices for $V_{i,x}$ are
$$V_{i,x} = \vn \textit{  or  } V_{i,x}=\cb{x},$$
and the possible choices of $W_{i,x}$ are
$$W_{i,x}(y)=\left\{
    \begin{array}{lr}
         \cb{y} & y\notin S_x \\
         \cb{x,y}& y\in S_x
    \end{array}
\right. \text{, for any choice of } S_x\subseteq \mathcal{N}_x.$$
We will find it convenient to set the index family $I_x = \cb{0,1}\times 2^{\mathcal{N}_x}$ and set
$$V_{0,S,x} = \vn \textit{, } \qquad V_{1,S,x} = \cb{x},$$
and $W_{i,S,x}$ as above with $S_x = S$.
Each pair $\bb{V_{i,S,x},W_{i,S,x}}$ represents an update by a transition matrix $P_{i,S,x}$ with
\begin{align*}
    &P_{1,S,x}\bb{1\,|\,\sigma}= 1,\\
    &P_{0,S,x}\bb{1\,|\,\sigma} = \left\{
    \begin{array}{lr}
         1 & S \cap \sigma^{-1}\bb{1}\neq \vn \\
         0 & S\cap \sigma^{-1}\bb{1}= \vn
    \end{array}
\right..
\end{align*}
For a given choice of $\lambda_{i,S,x}$ the transition matrix of IPS thus constructed is
$$P_x = \frac{1}{\lambda}\Cb{\bb{\lambda-\lambda_x }I+\desum{i\in\cb{0,1}}{S\subseteq \mathcal{N}_x}\lambda_{i,S,x} P_{i,S,x}}$$
where
\begin{align*}
    \lambda_x & = \desum{i\in\cb{0,1}}{S\subseteq \mathcal{N}_x}\lambda_{i,S,x}, \\
    \lambda & = \sup_{x\in \Z^d}\lambda_x.
\end{align*}
The clock is $\text{Exp}\bb{\lambda}$, but that can be normalised by scaling all $\lambda_{i,S,x}$ by $\frac{1}{\lambda}$.

\bigskip

\subsection{Addendum to Theorem~\ref{pro:cancellative}: Characterization of cancellative IPS} \label{sec:char_cancellative_IPS}

Again the same pairs $\bb{V_{i, S,x}, W_{i, S,x}}$ are available for the construction. However, the transition matrices $P_{i, S,x}$ corresponding to them are different:
\begin{align*}
    &P_{0,S,x}\bb{1\,|\,\sigma} = \abs{S\cap \sigma^{-1}\bb{1}}\bb{\text{mod }2},\\
    &P_{1,S,x}\bb{1\,|\,\sigma} = 1-\abs{S\cap \sigma^{-1}\bb{1}}\bb{\text{mod }2}.
\end{align*}

\section*{Statements and Declarations}
The authors have no relevant financial or non-financial interests to disclose.

\bibliography{sn-bibliography}

\end{document}